\documentclass[1 [leqno,11pt]{amsart}
\usepackage{amssymb, amsmath}
\usepackage{color}
\usepackage{enumerate}
\def\refeq#1{~(\ref{#1})}
\def\ccite#1{~\cite{#1}}

\def\longformule#1#2{
\displaylines{ \qquad{#1} \hfill\cr \hfill {#2} \qquad\cr } }
\def\inte#1{
\displaystyle\mathop{#1\kern0pt}^\circ }

\let\al=\alpha

\let\de=\delta

\let\d=\partial
\let\e=\varepsilon

\let\lam=\lambda
\let\r=\rho
\let\s=\sigma
\let\f=\frac
\let\vf=\varphi
\let\p=\psi
\let\om=\omega

\let\D=\Delta

\let\Om=\Omega
\let\wt=\widetilde

\def\cA{{\mathcal A}}
\def\cB{{\mathcal B}}
\def\y{{\mathfrak y}}

\def\u{{\mathfrak u}}
\def\cC{{\mathcal C}}
\def\cD{{\mathcal D}}
\def\cE{{\mathcal E}}
\def\cF{{\mathcal F}}

\def\cH{{\mathcal H}}

\def\fa{{\frak a}}

\def\la{\lambda}

\def\rp{{\rm p}}
\def\rw{{\rm w}}

\newcommand{\dx}{{\rm d}x}
\newcommand{\rl}{{\rm \rho}^\ell}
\newcommand{\rr}{{\rm \rho}}
\newcommand{\rx}{{\rm X}}
\newcommand{\dt}{{\rm d}t}

\renewcommand{\div}{{\rm div}\,}
\newcommand{\loc}{{\rm loc}\,}

\newcommand\I{\rm{\uppercase\expandafter{\romannumeral1}}}
\newcommand\II{\rm{\uppercase\expandafter{\romannumeral2}}}
\newcommand\III{\rm{\uppercase\expandafter{\romannumeral3}}}
\newcommand\IV{\rm{\uppercase\expandafter{\romannumeral4}}}
\newcommand\V{\rm{\uppercase\expandafter{\romannumeral5}}}
\newcommand\VI{\rm{\uppercase\expandafter{\romannumeral6}}}
\newcommand\VII{\rm{\uppercase\expandafter{\romannumeral7}}}
\newcommand\VIII{\rm{\uppercase\expandafter{\romannumeral8}}}
\newcommand\IX{\rm{\uppercase\expandafter{\romannumeral9}}}

\def\dH{\dot{H}}
\def\dB{\dot{B}}

\def\virgp{\raise 2pt\hbox{,}}
\def\cdotpv{\raise 2pt\hbox{;}}

\def\eqdefa{\buildrel\hbox{\footnotesize def}\over =}

\def\C{\mathop{\mathbb C\kern 0pt}\nolimits}
\def\DD{\mathop{\mathbb D\kern 0pt}\nolimits}
\def\EE{\mathop{{\mathbb E \kern 0pt}}\nolimits}
\def\K{\mathop{\mathbb K\kern 0pt}\nolimits}
\def\N{\mathop{\mathbb N\kern 0pt}\nolimits}
\def\Q{\mathop{\mathbb Q\kern 0pt}\nolimits}
\def\R{\mathop{\mathbb R\kern 0pt}\nolimits}
\def\SS{\mathop{\mathbb S\kern 0pt}\nolimits}
\def\ZZ{\mathop{\mathbb Z\kern 0pt}\nolimits}
\def\TT{\mathop{\mathbb T\kern 0pt}\nolimits}
\def\PP{\mathop{\mathbb P\kern 0pt}\nolimits}

\newcommand{\Z}{{\ZZ}}

\def\dive{\mathop{\rm div}\nolimits}

\def\na{\nabla}
\def\p{\partial}

\def\v{{\rm v}}

\def\th{\theta}

\newcommand{\w}[1]{\langle {#1} \rangle}
\newcommand{\beq}{\begin{equation}}
\newcommand{\eeq}{\end{equation}}
\newcommand{\ben}{\begin{eqnarray}}
\newcommand{\een}{\end{eqnarray}}
\newcommand{\beno}{\begin{eqnarray*}}
\newcommand{\eeno}{\end{eqnarray*}}
\newcommand{\andf}{\quad\hbox{and}\quad}
\newcommand{\with}{\quad\hbox{with}\quad}
\newtheorem{defi}{Definition}[section]
\newtheorem{thm}{Theorem}[section]
\newtheorem{lem}{Lemma}[section]
\newtheorem{rmk}{Remark}[section]
\newtheorem{col}{Corollary}[section]
\newtheorem{prop}{Proposition}[section]

\begin{document}
\title[Global regularity of 2-D density patch]
{Global  regularities of two-dimensional density patch for inhomogeneous
incompressible viscous flow with  general density}
 \author[X. Liao]{Xian Liao}
\address [X. Liao]%
{Academy of Mathematics $\&$ Systems Science, Chinese Academy of Sciences, P.R. China\\
Present address: Mathematisches Institut, Universit\"at Bonn, Endenicher Allee 60, 53115 Bonn, Germany.}
\email{xian.liao@amss.ac.cn}

\author[P. ZHANG]{Ping Zhang}%
\address[P. Zhang]
 {Academy of
Mathematics $\&$ Systems Science and  Hua Loo-Keng Key Laboratory of
Mathematics,  Chinese Academy of Sciences, P.R. China}
\email{zp@amss.ac.cn}

\date{April 26, 2016}
\maketitle
\begin{abstract} Toward the open question proposed by P.-L. Lions in \cite{Lions96}  concerning the propagation of regularities
of density patch for viscous inhomogeneous flow,  we first establish  the global in time well-posedness of
two-dimensional inhomogeneous incompressible Navier-Stokes system with
initial density being of the form: $\eta_1{\bf 1}_{\Om_0}+\eta_2{\bf
1}_{\Om_0^c},$ for any pair of positive  constants $(\eta_1,\eta_2),$ and for any bounded, simply connected $W^{k+2,p}(\R^2)$  domain
$\Om_0.$   We then prove that
the time evolved  domain $\Om(t)$ also belongs to the class of
$W^{k+2,p}$ for any $t>0.$ Thus in some sense, we have solved the aforementioned Lions' question
%of density patch in \cite{Lions96}
in the two-dimensional case.
Compared with our previous paper \cite{LZ}, here we remove the smallness condition on the jump, $|\eta_1-\eta_2|,$
moreover,
the techniques used in the present paper are completely different from those in \cite{LZ}.
\end{abstract}

\par $\qquad\quad\ \ $
striated distributions, Littlewood-Paley theory.


 \setcounter{equation}{0}

\section{Introduction}\label{sec1}

We  consider the following two-dimensional inhomogeneous incompressible
Navier-Stokes equations:
\begin{equation}\label{inNS}
 \left\{\begin{array}{l}
\displaystyle \d_t\rho+\dive(\rho v)=0,\qquad (t,x)\in\R^+\times\R^2,\\
\displaystyle\d_t(\rho v)+\dive(\r v\otimes v)-\Delta v+\nabla\pi=0,\\
\displaystyle \dive v=0,\\
\displaystyle  (\r, v)|_{t=0}=(\r_0, v_0).
\end{array}\right.
\end{equation}
Here the unknown~$\rho$ is a positive function\footnote{We do want to avoid vacuum.}
 which represents the density of the fluid at
time~$t$ and point~$x$,  $v=(v_1,v_2)$
 stands for the velocity field   of the fluid, and $\pi$   designates the  pressure  at
time~$t$ and point~$x$,  which ensures the incompressibility of the fluid.

 This system \eqref{inNS} can be used as a model to describe  a viscous fluid that is incompressible but has nonconstant density.
Basic examples are mixtures of incompressible
and non-reactant components, flows with complex structure (e.g. blood
flow or model of rivers), fluids containing a melted substance, etc.
 \smallbreak

Let us notice that in the case where~$\rho_0\equiv 1$, the
system \eqref{inNS} turns out to be the classical incompressible
Navier-Stokes system $(NS).$ We have to keep in mind that the system
\eqref{inNS} is much more complex than $(NS).$

\smallbreak
This system \eqref{inNS} has three major basic  features, and let us state them in general space dimension   $d\geq 2$.
Firstly,
the incompressibility condition on  the convection velocity field in the density transport equation ensures  that
 \beq \label{transformrho}
\|\rho(t)\|_{L^\infty}
=\|\rho_0\|_{L^\infty}
\andf
\mbox{meas}\bigl\{\ x\in\R^d\ |\ \al\leq\rho(t,x)\leq \beta\ \bigr\}\ \
\mbox{is independent of}\ t\geq 0, \eeq
for any pair of non-negative  real numbers $(\al,\beta)$.
Secondly   the
 kinetic energy  is formally conserved \beq \label
{kineticenergy} \frac 12 \int_{\R^d}\rho|v(t)|^2 \textrm{d}x
+\int_{0}^t\|\nabla v(t')\|_{L^2(\R^d)}^2 \textrm{d}t'
=\frac 12
\int_{\R^d}\rho_0(x)|v_0(x)|^2 \textrm{d}x. \eeq
The third  basic feature is
the scaling invariance property.
 Indeed, if~$(\rho,
v ,\pi)$ is a solution of \eqref{inNS} on~$[0,T]\times \R^d$,
then the rescaled triplet~$(\rho, v ,\pi)_\lam$ defined by
\beq
\label{scaling}
 (\rho, v
,\pi)_\lam(t,x) \eqdefa \bigl(\rho(\lam^2t, \lam x) , \lam v
(\lam^2t, \lam x) ,\lam^2 \pi(\lam^2t,\lam x)\bigr),
\quad \lambda\in\R\eeq is also a
solution of \eqref{inNS}  on $[0,T/\la^2]\times \R^d$. This leads to the
notion of critical regularity.

\smallbreak
Based on the energy estimate\refeq {kineticenergy},  J.
Simon constructed in\ccite{Simon} the global weak solutions of \eqref{inNS} with finite energy.
  More generally,
 P.-L. Lions  proved the global existence of weak solutions to the inhomogeneous incompressible Navier-Stokes system
 with variable viscosity
in  the book \cite{Lions96}.

In the
case of smooth initial data with no vacuum, the existence result of strong unique
solutions goes back to the work of O.-A. Ladyzhenskaya and V.-A.
Solonnikov   \cite{LS}. Motivated by \eqref{scaling}, R. Danchin \cite{D1}
established the well-posedness of \eqref{inNS} in the whole
space $\R^d$ in the so-called \emph{critical functional framework}
for small perturbations of some positive constant density. The basic
idea in  \cite{D1} is to use functional spaces (or norms) that have the same
\emph{scaling invariance} as \eqref{scaling}.
 This result
was extended to more general Besov spaces   in
\cite{Abidi, AP, AGZ2, AGZ3, PZ2}.

\smallbreak Given that in all those aforementioned works, the density has to be at
least in the Besov space $\dot B^{\f dp}_{p,\infty}(\R^d),$
  one cannot capture  discontinuities  across some hypersurface. In fact,
  the Besov regularity of the characteristic function of a smooth domain is only $\dot B^{\f1p}_{p,\infty}(\R^d).$
  Therefore, those results do not apply to a  mixture flow composed of two separate fluids with different densities.

The first breakthrough along this line was made by   R. Danchin and P.-B.  Mucha in \cite{DM12}, where they
basically
proved the
global well-posedness of \eqref{inNS} with initial density
allowing  discontinuity across a $C^1$
interface   with a sufficiently small jump.
Later on, even  $\rho_0$ is only bounded and with small fluctuation to some positive constant,
 J. Huang, M. Paicu and the second author in
\cite{HPZ} could show the global existence of solutions to
\eqref{inNS}  in a \emph{critical functional framework}, and the uniqueness was
obtained provided assuming slightly more regularities on the initial velocity
field. R. Danchin
and the second author extended this result to the half-space setting in \cite{DZ14}.

In the general case where $\r_0\in L^\infty(\R^d)$ with a positive lower bound
and $v_0\in H^2(\R^d),$ R. Danchin and P.-B. Mucha \cite{DM13} proved that the
system \eqref{inNS} has a unique local in time solution. Furthermore,
with the initial density fluctuation being sufficiently small, for
any initial velocity $v_0\in B^1_{4,2}(\R^2)\cap L^2(\R^2)$, or $v_0\in B^{2-\f2q}_{q,p}(\R^d)$ with small size for
$1<p<\infty, d<q<\infty$ and $2-\f2p\neq\f1q,$ they also proved the
global well-posedness of \eqref{inNS}.
M. Paicu, Z. Zhang and the second author \ccite {PZZ1} improved the well-posedness results in \cite{DM13} with less regularity
assumptions on the initial velocity.

 \smallbreak
 A natural question to ask is whether it is  possible to
 propagate the boundary regularities of the interface of the fluids.
In particular, P.-L. Lions proposed the following open question in
 \cite{Lions96}: suppose the initial density $\r_0={\bf 1}_{D}$ for some
 smooth domain $D,$ Theorem 2.1 of \cite{Lions96} ensures at least
 one global weak solution $(\r,v)$ of \eqref{inNS} such that for all
 $t\geq 0,$ $\r(t)={\bf 1}_{D(t)}$ for some set $D(t)$ with ${\rm vol}(D(t))={\rm vol}(D).$
 Then whether or not the regularity of $D$ is preserved by  time
 evolution?
  Since this problem is very sophisticated due to the possible
 appearance of vacuum, as a first step toward this question, here we aim at
establishing the global well-posedness of \eqref{inNS} with  the initial
density $\r_0(x)=\eta_1{\bf 1}_{\Om_0}+\eta_2{\bf 1}_{\Om_0^c} $ for any pair of positive constants $(\eta_1,\eta_2),$ and for any
 bounded simply connected  $W^{k+2,p}(\R^2)$ domain $\Om_0.$

More precisely, let  $\Omega_0$ be a simply connected $W^{k+2,p}(\R^2),$
  $k\geq 1,$ $p\in ]2,4[,$ bounded domain. Let $f_0\in
W^{k+2,p}(\R^2)$ such that $ \d\Omega_0=f_0^{-1}(\{0\})$ and $\nabla
f_0 $ does not vanish on $\d\Omega_0. $ Then we can parametrize
$\d\Omega_0$ as
\beq\label{S3eq19} \gamma_0: {\Bbb S}^1\mapsto \d\Omega_0
\hbox{ via } s\mapsto \gamma_0(s) \hbox{ with
}\d_s\gamma_0(s)=\nabla^\bot f_0(\gamma_0(s)). \eeq For any $\eta_1,\eta_2>0,$  we take the
initial density $\r_0$ and the initial velocity $v_0$ of
\eqref{inNS} as follows \beq\label{S3eq18}
\begin{split}
 &  {\rho_0}(x)={\eta_1}{\bf 1}_{\Om_0}(x)+\eta_2{\bf 1}_{\Om_0^c}(x),
  \quad v_0\in L^2\cap\dB^{s_0}_{2,1}
   \andf\\
   & \d_{X_0}^\ell v_0\in
L^2\cap\dB^{s_\ell}_{2,1}\quad\mbox{for}\  \ell=1,\cdots,k, \andf X_0\eqdefa\na^\perp f_0,\
\p_{X_0}v_0\eqdefa X_0\cdot\na v_0,
    \end{split}\eeq
    where  $s_0\in ]0,1[,$ $ s_\ell\eqdefa s_0-\epsilon{ \ell}/{k}$ for some fixed
$\epsilon\in ]0,s_0[,$  and $
p\in
\bigl]2,{2}/{(1-s_k)}\bigr[.
$

\smallbreak
The main ideas for us to solve the above problem come from
J.-Y. Chemin
\cite{Chemin88, Chemin91, Chemin93, Chemin98}.
Indeed,
by using the idea of conormal distributions or striated
distributions, J.-Y. Chemin \cite{Chemin91, Chemin93} (see also
\cite{BC93}) proved the global regularities of the two-dimensional vortex patch
for ideal incompressible flow. One may also check \cite{Danchin97, GS95, Hmidi05, Hmidi06, ZQ}
for some extensions.
 We emphasize
that
%the ideas in \cite{Chemin88, Chemin91, Chemin93, Chemin98} will
% play an essential role in this paper. Indeed
  the initial velocity field $v_0$ satisfying \eqref{S3eq18} belongs to some striated distribution
 space. To avoid  technicalities, here we prefer not to present  details of such spaces. Interested readers may check the
 book  \cite{Chemin98} and the references therein.

\smallbreak The main result of this paper states as follows:

\begin{thm}\label{thmmain}
{\sl Let the initial data $(\r_0, v_0)$ be given by \eqref{S3eq19} and \eqref{S3eq18},
for any pair of positive constants $(\eta_1,\eta_2).$ Then \eqref{inNS} has a unique global
solution $(\r, v)$ such that $\r(t,x)={\eta_1}{\bf 1}_{\Om(t)}(x)+\eta_2{\bf
1}_{\Om(t)^c}(x)$,  with $\Om(t)$  being a bounded,  simply connected  $W^{k+2,p}(\R^2)$ domain for any $t>0.$
}
\end{thm}

\begin{rmk} For $\Om_0$ given by \eqref{S3eq19}, it follows from \cite{LZ} that the initial velocity $v_0$ with vorticity $\om_0={\bf 1}_{\Om_0}(x)$
satisfies the assumptions of \eqref{S3eq18}.  Hence in particular, Theorem \ref{thmmain} ensures
the global well-posedness of \eqref{inNS} with initial density
$\r_0=\eta_1{\bf 1}_{\Om_0}(x)+\eta_2{\bf 1}_{\Om_0^c}(x)$ and initial vorticity $\om_0={\bf 1}_{\Om_0}(x).$
 \end{rmk}

\begin{rmk} \begin{itemize}

\item[(1)]  Besides the difficulties encountered in \cite{LZ}, here we remove the smallness assumption
on the jump, $|\eta_1-\eta_2|,$ which is crucial in \cite{LZ} to exploit  the maximal regularity estimate
for the time evolutionary  Stokes operator
to handle the $W^{2,p}$ estimate of the velocity field.

\item[(2)] The method of time-weighted energy estimate  in \cite{PZZ1} will play an important role here.
However, in order to propagate the $W^{2,p}$ regularity of the tangential vector field $X_0$ of $\p\Om_0,$ we need to deal with
the energy estimate of $\na\p_t v,$ which we can not go through under the mere bounded density assumption.
We overcome this difficulty
by working with the energy estimate of $\na D_t v$, where $D_t=\p_t+v\cdot\na$ denotes
 the material derivative.

\item[(3)] With initial density being only in the bounded function space, we can not apply the classical ideas  in \cite{D1} to propagate
the Besov regularity of the velocity field as  in the previous papers \cite{DZ14,HPZ,PZZ1}. Namely,   given the initial velocity
field $v_0\in \dB^s_{p,r},$ they did not prove the solution $v$ belonging to $ C([0,\infty[;\dB^s_{p,r}).$
In this paper, motivated by
the characterisation of Besov spaces with positive regularity index, we succeed in establishing the propagation of the Besov regularity of the velocity field.
One may check Theorem \ref{thm1} below for details.

\end{itemize}
\end{rmk}

\setcounter{equation}{0}
\section {Structure and main ideas of the proof}
\label {structureproof}

Before proceeding,
let us first recall
the  definitions of  Besov norms from \cite{BCD} for
instance.

\begin{defi}
\label {S0def1} {\sl  Let us consider a smooth radial function~$\vf $
on~$\R,$ the support of which is included in~$[3/4,8/3]$ such that
$$
\forall
 \tau>0\,,\ \sum_{j\in\Z}\varphi(2^{-j}\tau)=1\andf \chi(\tau)\eqdefa 1 - \sum_{j\geq
0}\varphi(2^{-j}\tau) \in \cD([0,4/3]).
$$
Let us define
$$
\Delta_ja\eqdefa\cF^{-1}(\varphi(2^{-j}|\xi|)\widehat{a}),
 \andf S_ja\eqdefa\cF^{-1}(\chi(2^{-j}|\xi|)\widehat{a}).
$$
Let $(p,r,\kappa)$ be in~$[1,+\infty]^3$ and~$s$ in~$\R$. We define the Besov norms by
$$
\|a\|_{\dB^s_{p,r}}\eqdefa\big\|\big(2^{js}\|\Delta_j
a\|_{L^{p}}\big)_j\bigr\|_{\ell ^{r}(\ZZ)} \andf
\|a\|_{\wt{L}_T^\kappa(\dB^s_{p,r})}\eqdefa\big\|\big(2^{js}\|\Delta_j
a\|_{L^\kappa_t(L^{p})}\big)_j\bigr\|_{\ell ^{r}(\ZZ)}. $$
}
\end{defi}

We remark that in the particular case when $p=r=2,$  the Besov
spaces~$\dB^s_{p,r}$ coincide with the classical homogeneous
Sobolev spaces $\dH^s$.

For notational simplicity, we always denote $\cB^s\eqdefa \dB^s_{2,1}.$

 \smallbreak

 As a matter of fact, we shall prove a much more general version of Theorem \ref{thmmain}. More precisely,
 we assume  the initial data satisfying
\begin{equation}\label{initial}
0<\rho_\ast\leq \rho_0\leq \rho^\ast, \quad v_0\in L^2\cap\cB^{s_0} \,\,
\mbox{for some}\ s_0\in ]0,1[,
\end{equation}
along with the following striated regularity assumptions for
$\ell=1,\cdots,k,$
\begin{equation}\label{initialell}
\begin{split}
&  \d_{X_0}^{\ell-1}X_0\in W^{2,p},\quad
\d_{X_0}^\ell\rho_0\in{L^\infty},
\quad \d_{X_0}^\ell v_0\in
L^2\cap\cB^{s_\ell}, \with
\\
& s_\ell\eqdefa s_0-\epsilon{ \ell}/{k}, \,\textrm{ for some
}\epsilon\in ]0,s_0[,
\andf  p\in
\bigl]2,{2}/{(1-s_k)}\bigr[,
\end{split}
\end{equation}
 for some vector field $X_0=(X_0^1,X_0^2)$ and
  $\p_{X_0}f\eqdefa X_0\cdot\na f$.
Moreover, to avoid cumbersome calculations, without loss of generality, we assume the
divergence free condition on the initial vector field $X_0$:
\begin{equation}\label{divX0}
\dive X_0=0.
\end{equation}
Given the convection velocity field $v,$  we define $X(t)=(X^1(t),X^2(t))$ via \beq\label{eq:X}
 \left\{\begin{array}{l}
\displaystyle \p_tX+v\cdot\na X=X \cdot\na v,\\
\displaystyle X(0,x)=X_0(x).
\end{array}\right. \eeq

 \noindent\textbf{Conventions.}   In this whole context, we shall use the following conventions:
\beq \label{S1eq9}
\begin{split}
& \p_X \eqdefa X\cdot\na,\quad  D_t\eqdefa\p_t+v\cdot\na,
\quad f_t\eqdefa \d_t f,\quad \s(t)\eqdefa \min\left\{1,t\right\},
 \andf \theta_0\eqdefa \epsilon/k,\\\
& {\rm \r}^\ell\eqdefa \p_X^\ell\r,\quad {\v}^\ell\eqdefa \d_X^{\ell}v,
\quad {{\rm \pi}}^\ell\eqdefa \d_X^\ell \pi,
\quad {{\rm X}}^\ell\eqdefa \p_X^\ell X,
\andf \d_X^{-1} f\eqdefa 0,
\\
&\cC(v_0,u_0,s)\eqdefa\|u_0\|_{\cB^s}\exp\left(\cA_0\exp\bigl(\cA_0\|v_0\|_{L^2}^4\bigr)\right) \andf \cC(v_0,s)\eqdefa \cC(v_0,v_0,s),\\
&
 \cH_\ell(t)\eqdefa
 \underbrace{\exp\exp\exp\cdots\exp}_{\ell+1 \hbox{ times}}\left(\cA_\ell \w{t}^2\right)\with \w{t}\eqdefa \bigl(1+t^2\bigr)^{\f12},
\end{split} \eeq
 where the constant $\cA_\ell$, $1\leq\ell\leq k$ depends only on $\r_\ast, \r^\ast, s_0, \theta_0$ and $\sum_{0\leq m\leq \ell}\bigl(\|\p_{X_0}^{m-1}X_0\|_{W^{2,p}}+\|\p_{X_0}^{m}\r_0\|_{L^\infty}+
  \|\p_{X_0}^{m}v_0\|_{L^2\cap \cB^{s_m}}\bigr),$ which  may vary from lines to lines in the whole context.

 \begin{thm}\label{thm1}
 {\sl Let the initial data $(\r_0,v_0,X_0)$ fulfil the assumptions \eqref{initial}, \eqref{initialell}
 and \eqref{divX0}. Then  the coupled system \eqref{inNS} and \eqref{eq:X}
  has a unique global solution $(\r,v,\na\pi, X)$ so that for any $t>0$
  \beq \label{S1eq1} A_0(t)
\leq \cC(v_0,s_0) \with  A_0(t)=A_{01}(t)+A_{02}(t),
\eeq
where
  \beq \label{S1eq1p}
\begin{split}
A_{01}(t)\eqdefa & \|v\|_{L^\infty_t(L^2)\cap L^2_t(\dot{H}^1)}
+\|v\|_{ \wt{L}^\infty_t(\cB^{s_0})\cap  {L}^2_t(\cB^{1+s_0})}
+\bigl\|\s^{\f{1-s_0}2}(v_t, \na^2v,\na\pi)\bigr\|_{L^2_t(L^2)}\\
&+\|\s^{\f{1-s_0}2}v\|_{ \wt{L}^\infty_t(\cB^1)}+\bigl\|\s^{1-\f{s_0}2}(v_t, \na^2v,\na\pi)\bigr\|_{L^\infty_t(L^2)}+\|\s^{1-\f{s_0}2}\na v_t\|_{L^2_t(L^2)},\\
A_{02}(t)\eqdefa &
 \bigl\|\sigma^{ \frac{3-s_0}2}
\bigl(\nabla D_t v, D_t\nabla v, \nabla v_t\bigr) \bigr\|_{L^\infty_t(L^2)}+\bigl\|\sigma^{\frac{3-s_0}2}
 \bigl(D^2_t v, \nabla^2 D_t v,
\nabla D_t \pi\bigr)\bigr\|_{L^2_t(L^2)},
\end{split}
\eeq
 and
   for  $\ell\in \{1,\cdots, k\}$
 \begin{equation}\label{estimate:Jell}
 \begin{split}
&\qquad\quad A_\ell(t)\leq \cH_\ell(t) \with\\
 A_\ell(t)\eqdefa & \|\v^\ell\|_{\wt{L}^\infty_t(\cB^{s_\ell})\cap {L}^2_t(\cB^{1+s_\ell})}+\|\rx^{\ell-1}\|_{L^\infty_t(W^{2,p})}
  + A_{\ell1}(t)+A_{\ell2}(t),
  \end{split}
\end{equation}
where
\begin{equation}\label{Jell}
\begin{split}
A_{\ell1}(t)
\eqdefa &\|\rl\|_{L^\infty_t(L^\infty)}
+\|\v^\ell\|_{L^\infty_t(L^2)\cap L^2_t(\dH^1)}+
 \| \sigma ^{ \frac{1-s_\ell}{2}}
\nabla \v^\ell  \|_{L^\infty_t(L^2)}\\
&  + \bigl\| \sigma ^{ \frac{1-s_\ell}{2}}
  \bigl(\d_t\v^\ell,
 \nabla^2\v^\ell, \nabla{\rm \pi}^\ell) \bigr\|_{L^2_t(L^2)} \andf
  \\
A_{\ell2}(t)
\eqdefa & \bigl\| \sigma ^{1-\frac{s_\ell}{2}}
 (D_t\v^\ell, \nabla^2\v^\ell, \nabla{\rm \pi}^\ell)
 \bigr\|_{L^\infty_t(L^2)}
 + \| \sigma ^{1-\frac{s_\ell}{2}}
   \nabla D_t\v^\ell \|_{L^2_t(L^2)}\\
&+ \| \sigma ^{ \frac{3-s_\ell}{2}}
\nabla D_t\v^\ell\|_{L^\infty_t(L^2)}+\bigl\| \sigma ^{ \frac{3-s_\ell}{2}}
 (D_t^2\v^\ell,
 \nabla^2 D_t\v^\ell,
 \nabla D_t {\rm \pi}^{\ell})\bigr\|_{L^2_t(L^2)}.
\end{split}
\end{equation}
}\end{thm}

Admitting Theorem \ref{thm1} for the time being, we now turn to the  proof  Theorem \ref{thmmain}.

\begin{proof}[Proof of Theorem \ref{thmmain}] For $f_0$ given by \eqref{S3eq19}, we take $X_0=\na^{\perp}f_0.$  Then we deduce from
\eqref{S3eq19} and \eqref{S3eq18} that for $\ell=1,\cdots,k,$
\beno
\begin{split}
& \dive X_0=0
\andf \d_{X_0}^{\ell-1}X_0\in W^{2,p} \andf\\
&
0<\min\left\{\eta_1,\eta_2\right\}\leq \r_0
\leq \max\left\{\eta_1,\eta_2\right\} \andf
\p_{X_0}^\ell\r_0\equiv 0. \end{split} \eeno
Thus  by virtue of
Theorem \ref{thm1}, the coupled system \eqref{inNS} with \eqref{eq:X} has a unique global solution $(\r, v, \na\pi, X)$
so that the inequalities \eqref{S1eq1} and  \eqref{estimate:Jell} are valid for any $t>0.$

Let us denote by $\psi(t,\cdot)$  the flow associated with the
 vector field $v,$ that is
\begin{equation}\label{S1eq10}
 \left\{\begin{array}{l}
\displaystyle \f{d}{dt}\psi(t,x)=v(t,\psi(t,x)),\\
\displaystyle \psi(0,x)=x,
\end{array}\right.
\end{equation}
from which and Corollary \ref{S1col1} (see below), we infer for any $t>0$,
\beq\label{S1eq19}
\begin{split}
\|\na\psi(t)-Id\|_{L^\infty}
&\leq \exp\bigl(\|\na v\|_{L^1_t( L^\infty)}\bigr)
\leq \exp\bigl(\w{t}\|\s^{\f{1-s_0}2}\na v\|_{L^2_t(L^\infty)}\bigr)
\leq \exp\left(\cA_0\w{t}\right);\end{split}
\eeq
and
\beq\label{S1eq20}
\begin{split}
\|\na^2\psi(t)\|_{L^p}\leq &\|\na\psi\|_{L^\infty_t(L^\infty)}^2\|\na^2 v\|_{L^1_t(L^p)}\exp\bigl(\|\na v\|_{L^1_t(L^\infty)}\bigr)\\
\leq&
\exp\left(\cA_0\w{t}\right)\|\s^{-\left(1-\f1p-\f{s_0}2\right)}\|_{L^2_t}
\|\s^{\left(1-\f1p-\f{s_0}2\right)}\na^2v\|_{L^2_t(L^p)}
\leq \exp\left(\cA_0\w{t}\right).
\end{split}
\eeq
Here we used the fact that $2<p<2/{(1-s_k)}$ so that $1-1/p-{s_0}/2<1/2.$

Let $\Om(t)\eqdefa \psi(t,\Omega_0).$ Due to \eqref{S3eq18}, we deduce from the transport equation of \eqref{inNS} that
 \beq \label{density} \r(t,x)={\eta_1}{\bf 1}_{\Om(t)}(x)+\eta_2{\bf
1}_{\Om(t)^c}(x). \eeq
   Next we are going to prove that
$\Omega(t)$ is
 of class $W^{3,p}.$ Indeed
 let $\d\Omega(t)$
 be the level surface of $f(t,\cdot).$ Then $f$ solves
\beno
 \left\{\begin{array}{l}
\displaystyle \p_tf+v\cdot\na f=0,\\
\displaystyle f(0,x)=f_0(x),
\end{array}\right. \eeno
which implies that $X(t,x)\eqdefa \na^\perp f(t,x)$ solves
\eqref{eq:X}. We thus deduce from \eqref{eq:X} and \eqref{S1eq10}
that  the tangential vector field $X(t,x)=\na^\perp f(t,x)$ satisfies
\beq \label{S1eq11}
\begin{split} &X(t,\psi(t,x))=X_0(x)\cdot\nabla {\psi(t,x)} .
\end{split} \eeq
Moreover,   in view of  \eqref{estimate:Jell}, we have
\beq\label{S1eq13}
 \rx^{\ell-1}\in
L^\infty_\loc(\R^+; W^{2,p}),\quad \mbox{for}\ \ \ell=1,\cdots, k.
\eeq
Then it follows then from \eqref{S1eq19}, \eqref{S1eq20} and \eqref{S1eq13} that
\beno
\begin{split} \d_{i}\p_j\bigl(X(t,\psi(t,x))\bigr)
&=\d_{i}\p_j\psi(t,x)\cdot\nabla X(t,\psi(t,x))\\
&+(\d_i\psi\otimes\d_j\psi)(t,x):\nabla^2 X(t,\psi(t,x)) \in
L^\infty_\loc(\R^+;L^p),
\quad\forall\ 1\leq i,j\leq 2.
\end{split} \eeno
Therefore we deduce from \eqref{S3eq19} and \eqref{S1eq11} that \beq\label{S1eq12a}
 \d_s(\psi(t,\gamma_0(s)))=X_0(\gamma_0(s))
 \cdot\nabla{ \psi(t,\gamma_0(s))}
= X(t,\psi(t,\gamma_0(s)))\in L^\infty_\loc(\R^+; W^{2,p}),\eeq that
is, $\Omega(t)$ belongs to the class of $W^{3,p}$ for any $t>0.$

\smallbreak

Along the same line to the derivation of \eqref{S1eq12a}, we write
$$
\d_s^{\ell}(\psi(t,\gamma_0(s))) =\d_s^{\ell-1}\bigl(
(\d_{X_0}\psi)(t,\gamma_0(s)) \bigr) =\cdots
=(\d_{X_0}^{\ell}\psi)(t,\gamma_0(s)), \quad\forall\ \ell\geq 1 .
$$
Thus in order to prove that $\p\Om(t)=\psi(t,\gamma_0(\p\Om_0))$  belongs to the class of $
W^{k+2,p}$ for any $t>0$, it suffices to show that $\d_{X_0}^\ell\psi \in
L^\infty_\loc(\R^+; W^{2,p})$ for any  $\ell\in \{1,2,\cdots, k\ \}$.

To this end, we notice that
 \eqref{S1eq11} implies that
$$
X(t,x)=(\d_{X_0}\psi)(t,\psi^{-1}(t,x)),
$$
so that for any smooth function $f,$ we write
\begin{align*}
\d_{X} \bigl(f(\psi^{-1}(t,x))\bigr)
&=\sum_{i=1}^2\sum_{\al=1}^2
X^i(t,x)\p_\al f(\psi^{-1}(t,x))\f{\p(\psi^{-1})^\al(t,x)}{\p x_i}
\\
&=\sum_{i=1}^2\sum_{\al=1}^2\sum_{j=1}^2
(X_0^j\p_j\psi^i)(t,\psi^{-1}(t,x)) \f{\p(\psi^{-1})^\al(t,x)}{\p x_i} \p_\al f(\psi^{-1}(t,x))
\\
&=(\d_{X_0} f)(\psi^{-1}(t,x)).
\end{align*}
Then an inductive argument leads to
\begin{align*}
{\rx}^{\ell-1}(t,x)
=\d_{X}^{\ell-1}\bigl((\d_{X_0}\psi)(t,\psi^{-1}(t,x))\bigr)
=(\d_{X_0}^{\ell}\psi)(t,\psi^{-1}(t,x)), \quad
\forall\ \ell\geq 1,
\end{align*}
which together with \eqref{S1eq19}, \eqref{S1eq20} and \eqref{S1eq13} implies $ \d_{X_0}^\ell\psi\in L^\infty_\loc(\R^+; W^{2,p}) $ for any $\ell\in \{1,2,\cdots, k\ \}.$
 This
completes the proof of Theorem \ref{thmmain}.
\end{proof}

The rest of this section is devoted to the outline of the proof to Theorem \ref{thm1}.
Let us begin with the notations we are going to use
in the whole context.

%\smallbreak
 \noindent{\bf Notations:}  We denote $(a | b)\eqdefa \int_{\R^2}a|b\,\dx$ to be the $L^2(\R^2)$ inner product of $a$ and $b,$ and
$[A;B]=AB-BA$ to be the commutator of the operators $A$ and $B.$
%For $Y$ a Banach space
%and $I$ an interval of $\R,$ we denote by
% $L^q(I;\,Y)$
% the set of measurable functions on $I$ with values in
%$Y,$ such that $t\longmapsto\|f(t)\|_{Y}$ belongs to $L^q(I).$
For
$a\lesssim b$, we mean that there is a uniform constant $C,$ which
may be different on different lines, such that $a\leq Cb$. Finally we denote
$(d_q)_{q\in\Z}$ to be a generic element of $\ell^1(\Z)$ so that $\sum_{q\in\Z}d_q=1.$

\smallbreak
We shall first prove in Section \ref{sec:J0} that there holds \eqref{S1eq1} under the assumption \eqref{initial}.

\begin{prop}\label{S1prop1}
{\sl Let $(\r,v, \na\pi)$ be a smooth enough solution of \eqref{inNS} and $A_{01}(t)$ be the functional given by \eqref{S1eq1p}. Then under the hypothesis of \eqref{initial}, the
following inequality is valid for any $t\geq 0$ \beno
A_{01}(t)\leq \cC(v_0,s_0).\eeno
}
\end{prop}

Let us remark that the essence of the proof of Proposition \ref{S1prop1} is to derive the {\it a priori} estimates
for $\|v\|_{\wt{L}^\infty_t(\cB^{s_0})}$ and $\|\s^{\f{1-s_0}2}v_t\|_{L^2_t(L^2)},$ which is completely new compared with the
previous references \cite{DZ14,HPZ,PZZ1}. Due to the fact that the density function only belongs to the bounded function space, we can not use the classical ideas
 in \cite{D1}, namely, we can not apply the operator $\D_j$ to the momentum equation of \eqref{inNS} and then perform energy
estimate for $\D_j v.$

Here  the main idea of the proof  to the propagation of the Besov regularities of the velocity field is motivated by the characterization of Besov norms with positive regularity index. More
precisely, we write
\beq\label{S1eq14}
v=\sum_{q\in\Z} v_q \andf \na\pi=\sum_{q\in\Z} \na\pi_q,
\eeq
with $(v_q, \na\pi_q)$ solving
\begin{equation}\label{S1eq15}
 \left\{\begin{array}{l}
\displaystyle\rho \p_t v_q+\r v\cdot\na v_q-\Delta v_q+\nabla \pi_q=0,\\
\displaystyle \div v_q=0,\\
\displaystyle   v_q|_{t=0}=\D_q v_0.
\end{array}\right.
\end{equation}
Then by performing $H^1$ energy estimate to \eqref{S1eq15}, one has
\beq \label{S1eq16}
\|v_q\|_{L^\infty_t(L^2)}+\|\s^{\f12}\p_tv_q\|_{L^2_t(L^2)}
+2^{-q}\bigl(\|\na v_q\|_{L^\infty_t(L^2)}+\|\p_tv_q\|_{L^2_t(L^2)}\bigr)\lesssim d_q 2^{-qs_0}\|v_0\|_{\cB^{s_0}}.
\eeq
We thus deduce from Bernstein-type lemma (see Lemma 2.1 of \cite{BCD} for instance) that
\beno
\begin{split}
\|\D_j v\|_{L^\infty_t(L^2)}\leq
&\sum_{q>j}\|\D_j v_q\|_{L^\infty_t(L^2)}
+2^{-j}\sum_{q\leq j}
\|\D_j\na v_q\|_{L^\infty_t(L^2)}
\\
\lesssim &\sum_{q>j}
\| v_q\|_{L^\infty_t(L^2)}
+2^{-j}\sum_{q\leq j}
\|\na v_q\|_{L^\infty_t(L^2)}
\lesssim d_j
2^{-j s_0}\|v_0\|_{\cB^{s_0}},
\end{split}
\eeno
which together with Definition \ref{S0def1} ensures that $v\in \wt{L}^\infty_t(\cB^{s_0}).$
Indeed the same procedure would imply that $v\in \wt{L}^\infty_t(\dB^{s_0}_{2,r})$ for any $r\in [1,\infty]$ provided
that $v_0\in \dB^{s_0}_{2,r}.$ However in this case, \eqref{S1eq16} would become
\beno
\|\s^{\f12}\p_tv_q\|_{L^2_t(L^2)}
+2^{-q}\|\p_tv_q\|_{L^2_t(L^2)}\lesssim c_{q,r} 2^{-qs_0}\|v_0\|_{\dB^{s_0}_{2,r}},
\eeno
where $(c_{q,r})_{q\in\Z}$ is a generic element of $\ell^r(\Z)$ so that $\bigl\|(c_{q,r})_{q\in\Z}\|_{\ell^r(\Z)}=1.$
As a result, we infer
\beno
\|\s^{\f{1-s_0}2}\d_t v\|_{L^2_t(L^2)}\leq \sum_{q\in\Z}\|\p_tv_q\|_{L^2_t(L^2)}^{s_0}\|\s^{\f12}\p_tv_q\|_{L^2_t(L^2)}^{1-s_0}
\lesssim  \|v_0\|_{\dB^{s_0}_{2,r}}\sum_{q\in\Z}c_{q,r}.
\eeno
In order to guarantee the series, $\sum_{q\in\Z}c_{q,r},$ to be  convergent, the only choice is $r=1.$ That is the reason why
we choose to work with the initial velocity in the Besov space, $\cB^{s_0},$ instead of the classical Sobolev space $\dH^{s_0}.$

With Proposition \ref{S1prop1},  we shall exploit   the 2-D interpolation inequality
\beq \label{S2eq6} \|a\|_{L^r}\leq
C\|a\|_{L^2}^{\f2r}\|\na a\|_{L^2}^{1-\f2r}\quad \forall\ r\in
]2,\infty[, \eeq to prove that

\begin{col}\label{S1col1}
{\sl Let $r\in [2,\infty[.$ Then under the same assumptions of
Proposition \ref{S1prop1}, we have for any $t\geq 0$,
\begin{equation}\label{S1eq2}
\begin{split}
 &\bigl\|\sigma^{\frac{1-s_0}{2}}v\bigr\|_{L^\infty_t(L^\infty)}
+\bigl\|\sigma^{\frac{1-s_0}{2}\left(1-\frac
2r\right)}v\bigr\|_{L^\infty_t(L^r)}
+\bigl\|\sigma^{\frac{1-s_0}{2}}\nabla
v\bigr\|_{L^{\frac{2r}{r-2}}_t(L^r)}
\\
& +\bigl\|\sigma^{\bigl(1-\frac 1r-\frac{s_0}{2}\bigr)}\nabla
v\bigr\|_{L^{\infty}_t(L^r)}
+\bigl\| \sigma^{\bigl(1-\frac
1r-\frac{s_0}{2}\bigr)} \bigl(v_t, \na^2 v, \nabla \pi \bigr)
\bigr\|_{L^2_t(L^{r})}
\\
& +\bigl\|\sigma^{\bigl(1-\frac{s_0}2\bigr)}\bigl(v_t, \na^2 v, \nabla
\pi\bigr) \bigr\|_{L^{\frac{2r}{r-2}}_t(L^r)}
+\bigl\| \sigma^{\frac{1-s_0}2} \nabla v\bigr\|_{L^{2}_t(L^\infty)} \leq
\cC(v_0,s_0).
\end{split}
\end{equation}
}
\end{col}

In order to derive the {\it a priori} estimate of $\|\na X\|_{L^\infty_t(W^{1,p})},$ not only we need to perform the
energy estimate for $v_t$ but also the energy estimate of $\na v_t.$
To this end, we get, by applying $\p_t$ to the momentum equation of
\eqref{inNS}, that
\beno \label{S1eq18}
\rho\p_tv_t+\r v\cdot\na v_t -\Delta v_t +\nabla{\pi}_t
=-\r_t(v_t+v\cdot\na v)-\r v_t\cdot\na  v.
\eeno
To perform the energy estimate for $\na v_t,$ we need to deal with such terms as
$$\int_{\R^2} \r_t D_tv | v_{tt}\,\textrm{d}x
=-\int_{\R^2} v\cdot\na\r D_tv | v_{tt}\,\textrm{d}x,$$
which is impossible to go through
with non-Lipschitz density function $\r$.
This  is nevertheless the case here.

The idea to overcome this difficulty  is to
apply the material derivative $D_t$ instead of $\p_t$ to the momentum equation of \eqref{inNS}, which gives
\begin{equation}\label{eq:Dv}
\begin{split}
\rho D^2_t v
-\Delta D_t v+\nabla D_t \pi
=&F_D(v,\pi) \with\\
F_D(u,\Pi)\eqdefa&   -2\nabla v_\al\cdot \p_\al\nabla u-\Delta v\cdot\nabla u
+\nabla v_\al\p_\al \Pi.
 \end{split}
\end{equation}
Here and in what follows, repeated indices of $\al$   means summation of $\al$  from $1$ to $2.$ We remark that the advantage of exploiting  the operator $D_t$ is that $D_t\rr=0$, so that difficult terms mentioned before do not appear anymore.

Note that due to $\dive v=0,$ one has
\beq\label{S4eq7p}
\begin{split}
&\dive D_tv=\dive (v\cdot\nabla v)=\p_\al v\cdot\na v_\al \andf\\
&\div D_t^2 v
=
\div \frak{b}_0 \with \frak{b}_0
\eqdefa v\cdot(\nabla  v_t+   D_t\nabla v)+D_t v\cdot\nabla v.
\end{split} \eeq

In Subsection \ref{subs:Dv'}, we shall use the following lemma
to perform $\dH^1$ energy estimate for $D_tv:$

\begin{lem}\label{S4lem01}
{\sl Let  $(\rw,\na{\rm q})$ be a smooth enough solution of the following system:
\begin{equation}\label{S8eq1}
 \left\{\begin{array}{l}
\displaystyle\rho D_t^2 \rw-\Delta D_t \rw+\nabla \rm{q}=\rm{F},\\
\displaystyle \div D_t\rw=\dive \fa \andf \dive D_t^2\rw=\dive\frak{b}.
\end{array}\right.
\end{equation}
Then for any $s\in]0,1[,$ we have
\beq\label{S8eq2}
\begin{split}
 \|\sigma^{1-\frac{s}2}& D_t \rw \|_{L^\infty_t(L^2)\cap L^2_t(\dH^1)}^2
+\|\sigma^{\frac{3-s}2} \na D_t \rw \|_{L^\infty_t(L^2)}^2
+  \|\sigma^{\frac{3-s}2} (D_t^2 \rw,\na^2 D_t\rw,\na{\rm q})\|_{L^2_t(L^2)}^2\\
&\leq C\exp\left(C\|v_0\|_{L^2}^2\right)\bigl(\|\s^{\f{1-s}2}(D_t\rw,\fa)\|_{L^2_t(L^2)}^2+\|\s^{1-\f{s}2}\na \fa\|_{L^2_t(L^2)}^2\\
&\qquad\qquad\qquad\qquad\qquad\qquad\qquad\qquad+\|\s^{\f{3-s}2}(\na\dive\fa, \frak{b},{\rm F})\|_{L^2_t(L^2)}^2\bigr).
\end{split} \eeq
}
\end{lem}

The main result of  Subsection \ref{subs:Dv'} states as follows:

\begin{prop}\label{S1prop2}
{\sl Let $A_{02}(t)$ be the functional given by \eqref{S1eq1p}. Then under the assumptions of Proposition \ref{S1prop1}, the following estimate
is valid for any $t\geq 0$
\beno
A_{02}(t)\leq \cC(v_0,s_0). \eeno}
\end{prop}

This verifies the Estimate  \eqref{S1eq1}. Next let us turn to the proof of \eqref{estimate:Jell} for $\ell=1.$ To this end, let us first derive
the equations satisfied by $({\rm\r}^1, \v^1, \na{\rm \pi}^1).$
It is easy to observe  from \eqref{eq:X} and \eqref{S1eq9} that
 \beq\label{S3eq-1}
 [\d_X; D_t]=0,
 \hbox{ that is, }
  \p_X
D_tf=D_t\p_Xf.
%\quad \with  \p_X\eqdefa X\cdot\na \andf
%D_t\eqdefa\p_t+v\cdot\na.
\eeq
Then we get, by applying $\p_X^\ell$ to the transport equation
of \eqref{inNS}, that
$$
D_t\rho^\ell=\d_t\rr^\ell+v\cdot\nabla\rr^\ell=0, \quad
\ell=1,\cdots,k,
$$
which implies
\begin{equation}\label{brho}
\|\rr^\ell\|_{L^\infty(\R^+;L^\infty)} \leq
\|\d_{X_0}^\ell\rho_0\|_{L^\infty}, \quad \ell=1,\cdots,k.
\end{equation}

While due to $\dive X_0=0=\dive v,$ we deduce from  \eqref{eq:X} that
 \beq\label{S1eq12}
 \p_t(\dive X)+v\cdot\na (\dive X)=0 \with \dive X_0=0.
 \eeq
 Hence   $\div v=\div X=0,$  a straightforward calculation shows   that
\begin{equation}\label{S3eq1a}
\div\v^1=\div (\p_X v)=\p_\al X\cdot\nabla v_\al =\div(v\cdot\nabla X).
\end{equation}
It is also easy to observe that \beq\label{S3eq0}
\begin{split}
&[\p_X; \p_i]f=-\p_iX\cdot\na f \andf [\p_X;\p_i^2]f=-\p_i^2X\cdot\na f-2\p_iX\cdot\na \p_i f,
\quad i=1,2.
\end{split}
\eeq
In view of \eqref{S3eq-1}, \eqref{S3eq1a} and \eqref{S3eq0}, by taking $\p_X$ to the momentum equation of \eqref{inNS},  we find that $(\v^1, \na\rm{\pi}^1)$ solves
\begin{equation}\label{S1eq3}
 \left\{\begin{array}{l}
\displaystyle \rho\d_t \v^1+\r v\cdot\nabla \v^1 -\Delta \v^1 +\nabla{\rm{\pi}}^1
=F_1(v,\pi),\\
\displaystyle \div  \v^1=\div(v\cdot\nabla X),\\
\displaystyle  \v^1|_{t=0}=\p_{X_0}v_0,
\end{array}\right.
\end{equation}
where the source term $F_1(v,\pi)$ is given by
 \beq \label{S3eq1}
 F_1(v,\pi)\eqdefa
 -\rr^1 D_t v
-(\Delta X\cdot\nabla v+2\p_\al X\cdot\nabla\p_\al v) +\nabla
X^\al\p_\al\pi. \eeq

In Section \ref{sec:J1}, we shall use the following lemma to work with the $H^1$ energy estimate of $\v^1$:

\begin{lem}\label{S3lem2}
{\sl Let $(\r, v, \na \pi)$ be a smooth enough solution of
\eqref{inNS} and  $(u,\na \Pi)$ be determined by
\begin{equation}\label{S2eq1}
 \left\{\begin{array}{l}
\displaystyle\rho \p_t u+\r v\cdot\na u-\Delta u+\nabla \Pi=f, \quad (t,x)\in \R^+\times\R^2;\\
\displaystyle \div u=\dive g,
\end{array}\right.
\end{equation}
 with initial data $u_0=0.$ Then for any $s\in ]0,1[$ and $\de\in ]s,1[,$ one has \beq
\label{lemeq1}
\begin{split}
\bigl\|\s^{\f{1-{s}}2}&\na
u\bigr\|_{L^\infty_t(L^2)}^2+\bigl\|\s^{\f{1-{s}}2}(
u_t,\na^2 u,
\na \Pi)\bigr\|_{L^2_t(L^2)}^2\\
\leq& \cA_0\Bigl(\f{\w{t}}{\de-s}\|g(0)\|_{\cB^{\de}}^2+\|\s^{-\f{1-(s_0-s)}2}\na g\|_{L^1_t(L^2)}^2
\\
&+\bigl\|\s^{-\f{s}2}(f,g_t)\bigr\|_{L^1_t(L^2)}^2+\|\s^{-\f{s}2}\na g\|_{L^2_t(L^2)}^2
+
\bigl\|\sigma^{\frac{1-{s}}{2}}(g_t,\na\dive g,f)\bigr\|_{L^2_t(L^2)}^2\Bigr).
\end{split} \eeq }
\end{lem}

The main result concerning the $H^1$ energy estimate of $\v^1$ lists as follows:

\begin{prop}\label{prop:B1}
{\sl Suppose that  the initial data, $(\r_0,v_0, X_0),$ satisfies the assumption  \eqref{initial},
\eqref{divX0} and $\d_{X_0}\rho_0\in L^\infty,$
$ \d_{X_0} v_0\in L^2\cap
\cB^{s_1}.$
Let $(\rho, v, \nabla\pi, X)$ be a smooth enough solution of the
coupled  system \eqref{inNS} and \eqref{eq:X}.
Then one has
\beq\label{S3eq51}
  A_{11}^2(t)\leq
 \exp\bigl(\cA_0 \w{t}^2\bigr) \Bigl(\cA_1+\int_0^t\bigl(\frak{B}_0(t')
 +\s(t')^{-1+\frac{\theta_0}2}\bigr)
\|\na X(t')\|_{W^{1,p}}^2\,dt'\Bigr),
\eeq
where $\th_0$ is defined in \eqref{S1eq9},
  the functional $A_{11}(t)$ is determined by \eqref{Jell} and
\beq \label{S1eq25}
\begin{split}
\frak{B}_0(t)&\eqdefa\|\na v(t)\|_{L^2}^2
%+ \|\s^{\f{1-s_0}2}(v_t,\na^2v,\na\pi)(t)\|_{L^2}^2
+\bigl\|\s^{ \left(1-\f1r-\f{s_0}2\right) }(v_t,\na^2v,\na\pi, D_t v, v\otimes\nabla v)(t)\bigr\|_{L^r}^2
\\
&+\|\s^{1-\f{s_0}2}\na v_t(t)\|_{L^2}^2
+\|\s^{\f{1-s_0}2}\na v(t)\|_{L^\infty\cap L^{\frac{2p}{p-2}}}^2+\bigl\|\sigma^{1-\frac{s_0}2}(D_t v, v\otimes\nabla v, v_t)(t)\bigr\|_{L^\infty}^2\\
&+\bigl\|\sigma^{\frac{3-s_0}2}
 \bigl(D^2_t v, \nabla^2 D_t v, D_t\nabla^2 v,
%\nabla D_t\nabla v,
\nabla D_t \pi,
D_t\nabla \pi\bigr)(t) \bigr\|_{L^2}^2
\\
&
+\bigl\|\sigma^{\left(\frac 32-\frac 1r-\frac{s_0}2\right)}
(\nabla D_t v, D_t \nabla v, \nabla (v\otimes\nabla v),  \nabla v_t)(t)\bigr\|_{L^r}^2+\|\s^{\f{3-s_0}2}\frak{b}_0(t)\|_{L^2}^2,
\end{split}
\eeq
for any  $r\in ]2,\infty[$,  and where $\frak{b}_0$ is given by \eqref{S4eq7p}.
 }
\end{prop}

\smallbreak
To derive the equation of $(D_t\v^1,\na D_t{\rm \pi}^1),$ we get, by
applying the operator $D_t$ to the $\v^1$ equation of \eqref{S1eq3} and using $D_t\r=0,$ that
\begin{equation}\label{eq:Dv1}
\begin{split}
&\rho D_t^2 \v^1-\Delta D_t \v^1
+ \nabla D_t {\rm \pi}^1
=F_D(\v^1,{\rm \pi}^1)+D_t F_1(v,\pi)
\eqdefa F_{1D} ,
\end{split}
\end{equation}
where $F_D$ and $F_1$ are given by \eqref{eq:Dv}  and \eqref{S3eq1} respectively,  and we thus obtain
\beq\label{S4eq14p} \begin{split}
F_{1D}
&=-2\p_\al v\cdot\nabla\p_\al \v^1-\Delta v\cdot\nabla \v^1
+\nabla v_\al\p_\al\rm{\pi}^1
\\
&\quad-\rr^1D_t^2 v
-D_t(\Delta X\cdot\nabla v+2\p_\al X\cdot\nabla\p_\al v)
+D_t(\nabla X^\al\p_\al\pi).
\end{split} \eeq

In Section \ref{sectdtv}, we shall apply  Lemma \ref{S4lem01} to deal with the time-weighted $H^1$ energy estimate of $D_t\v^1.$
This together with Proposition \ref{prop:B1} leads to the energy estimate of $D_t\v^1$ and $\na D_t\v^1,$ namely

\begin{prop} \label{prop:Dv1}
{\sl
Let $(\rho, v, \nabla\pi, X)$ be a smooth solution of the coupled system \eqref{inNS} and \eqref{eq:X}.
Then under the hypothesis of \eqref{initial}, \eqref{initialell} for $\ell=1$ and \eqref{divX0}, the Estimate \eqref{estimate:Jell} is valid for $\ell=1.$  }
\end{prop}

To handle the estimate \eqref{estimate:Jell} for $\ell$ varying from $2$ to $k$ when $k\geq 2$, we need to derive the equations satisfied by $(\v^\ell,\na{\rm \pi}^\ell)$
and $(D_t\v^\ell,\na D_t{\rm \pi}^\ell)$ respectively. Indeed by  taking
$\d_X^{\ell-1}$ with $\ell\geq 2$ to \eqref{S1eq3}, we write
\begin{align}\label{eq:vell}
\rho\d_t\v^\ell+\rho v\cdot\nabla\v^\ell
-  \Delta\v^\ell+\nabla{\rm{\pi}}^\ell
=F_\ell(v,\pi),
\end{align}
where the source term $F_\ell(v,\pi)$ is determined inductively by
\begin{equation*}\label{Fell}
F_\ell(v,\pi)
=F_1(\v^{\ell-1}, {\rm{\pi}}^{\ell-1})
+\d_X F_{\ell-1}(v,\pi),
\quad \ell\geq 2,
\end{equation*}
with the function $F_1(\cdot,\cdot)$ being given by \eqref{S3eq1}.  We thus get by induction that
\beno\begin{split}
F_\ell(v,\pi)
&=\sum_{i=0}^{\ell-1}
\d_X^{i}F_1(\v^{\ell-1-i},{\rm{\pi}}^{\ell-1-i})
\\
&=\sum_{i=0}^{\ell-1} \d_X^{i}\left(
-\rr^1 D_t \v^{\ell-1-i}
-  2\p_\al X\cdot\nabla\p_\al\v^{\ell-1-i}
- \Delta X\cdot\nabla \v^{\ell-1-i}
+\nabla X\cdot\nabla{\rm{\pi}}^{\ell-1-i}
\right).
\end{split}
\eeno
Taking into account of the fact \eqref{S3eq-1} that $[\d_X; D_t]=0,$ we obtain
\beq\label{S1eq5}
\begin{split}
F_\ell(v,\pi)=\sum_{i=0}^{\ell-1}\sum_{j=0}^{i}
C^j_{i}\Bigl(& -\rr^{j+1}
D_t\v^{\ell-1-j}
-2\d_X^j\p_\al X \cdot\d_X^{i-j}\nabla\p_\al\v^{\ell-1-i}
\\
&-\d_X^j \Delta X\cdot\d_X^{i-j}\nabla\v^{\ell-1-i}
+\d_X^j\nabla X\cdot\d_X^{i-j}\nabla{\rm{\pi}}^{\ell-1-i}\Bigr) .
\end{split} \eeq

On the other hand, by applying the operator $D_t$ to  \eqref{eq:vell} and using $D_t\r=0$ once again,  we  find
  \begin{equation}\label{eq:Dvell}
  \rho D_t^2 \v^\ell
  -\Delta D_t \v^\ell+\nabla D_t {\rm{\pi}}^\ell
  =D_t F_\ell(v,\pi)+F_D(\v^\ell,{\rm{\pi}}^\ell)\eqdefa F_{\ell D},
  \end{equation}
  where $F_\ell(\cdot,\cdot)$ and $F_D(\cdot,\cdot)$ are given respectively by  \eqref{S1eq5} and \eqref{eq:Dv}.
 Again thanks to $[\d_X; D_t]=0$, we  have
  \beq\label{S1eq18}
  \begin{split}
  &F_{\ell D}
  =-2\p_\al v\cdot\nabla\p_\al \v^\ell-\Delta v\cdot\nabla \v^\ell +\nabla v\cdot\nabla{\rm{\pi}}^\ell
  +\sum_{i=0}^{\ell-1}\sum_{j=0}^i
  C_i^j
  \Bigl( -\rr^{j+1} D_t^2 \v^{\ell-1-j}
  \\
  &\, -D_t\bigl(2\d_X^j\p_\al X\cdot\d_X^{i-j}\na\p_\al \v^{\ell-1-i}
  +\d_X^j\Delta X\cdot\d_X^{i-j}\nabla \v^{\ell-1-i}
  -\d_X^j\nabla X\cdot\d_X^{i-j}\nabla {\rm{\pi}}^{\ell-1-i}\bigr)\Bigr).
  \end{split} \eeq

With \eqref{eq:vell}  and \eqref{eq:Dvell},  by repeating the proof of Proposition \ref{prop:Dv1} and through an inductive argument, we shall prove in Sections \ref{Sect6} and  \ref{sec:Dvell}
that

\begin{prop}\label{prop:Jell}
{\sl
Let $(\rho, v, \nabla\pi, X$) be the smooth solution to the coupled system \eqref{inNS} and\eqref{eq:X}.
Then under the assumptions of  \eqref{initial}, \eqref{initialell} and \eqref{divX0} with $k\geq 2$,
 the Estimate \eqref{estimate:Jell} is valid for $\ell=2,\cdots,k$.
}
\end{prop}

Now we are in the position to complete the proof of Theorem \ref{thm1}.

\begin{proof}[Proof of Theorem \ref{thm1}] By mollifying the initial data satisfying the assumptions \eqref{initial}, \eqref{initialell} and \eqref{divX0},  we first construct the global approximate solutions to \eqref{inNS}.
Then along the same lines to the proof of the above propositions, we obtain similar uniform estimates  as \eqref{S1eq1} and \eqref{estimate:Jell} 
for such approximate solutions. Finally a standard compactness argument as that used in the proof of Theorem 1.1 of \cite{HPZ}
completes the existence part of Theorem \ref{thm1}. The uniqueness of such solutions has been proved in \cite{PZZ1}.
We skip the details here. \end{proof}

\setcounter{equation}{0}
\section{The propagation of Besov regularities of the velocity field}\label{sec:J0}

 The goal of this section is to prove the Estimate \eqref{S1eq1}. Toward this, let us first present some preliminary estimates for the
linearized system \eqref{S2eq1}.

%%%%%%%%%%%%%%%%%%%%%
\subsection{Some preliminary energy estimates}\label{subs:Eu}
%%%%%%%%%%%%%%%%%%%%%%%%%%%%
Let $(\r,v,\na\pi)$ be a smooth enough solution of \eqref{inNS}.
We first deduce from \eqref{inNS} and \eqref{initial} that
 \beq \label{S2eq4} \r_\ast\leq \r(t,x)\leq
\r^\ast,\andf \f12\|\sqrt{\r} v(t)\|_{L^2}^2+\|\na
v\|_{L^2_t(L^2)}^2=\f12\|\sqrt{\r_0} v_0\|_{L^2}^2,
\quad\forall\ t\geq 0,\eeq
so that we deduce from  \eqref{S2eq6} that \beq \begin{split}
\|v\|_{L^4_t(L^4)}\leq C \|v\|_{L^\infty_t(L^2)}^{\f12}\|\na v\|_{L^2_t(L^2)}^{\f12}\leq C \|v_0\|_{L^2},
\quad\forall\ t\geq 0. \end{split}\label{S2eq5}\eeq

\begin{lem}\label{S2lem1}
{\sl   Let  $(u,\na\Pi)$ be a smooth enough
solution of \eqref{S2eq1}. Then we have
\beq \label{S2eq8}
\begin{split}
\f{d}{dt}\|\na u(t)\|_{L^2}^2+\|(u_t, \na^2u, \na\Pi)\|_{L^2}^2
\leq
C\bigl(\|v\|_{L^4}^4\|\na
u\|_{L^2}^2+\|(g_t,\na\dive g,f)\|_{L^2}^2\bigr).
\end{split}
\eeq}
\end{lem}

\begin{proof}
By taking $L^2(\R^2)$ inner product of the $u$ equation of
\eqref{S2eq1} with $u_t,$ we obtain \beno
\int_{\R^2}\r|u_t|^2\,\dx+\f12\f{d}{dt}\|\na
u(t)\|_{L^2}^2=\int_{\R^2}\bigl(f-\na\Pi-\r v\cdot\na u\bigr) | u_t \,\dx. \eeno
 It follows from the 2-D interpolation
inequality \eqref{S2eq6} that \beno
\begin{split}
\bigl|\int_{\R^2}\r v\cdot\na u | u_t \,\dx\bigr| \leq
&C\|v\|_{L^4}\|\na
u\|_{L^4}\|\sqrt{\r}u_t\|_{L^2}\\
\leq &C\|v\|_{L^4}\|\na u\|_{L^2}^{\f12}\|\na^2
u\|_{L^2}^{\f12}\|\sqrt{\r}u_t\|_{L^2}.
\end{split}
\eeno On the other hand, we get, by taking the space divergence operator to the $u$ equation of \eqref{S2eq1}, that
\beno
\D \Pi=\dive\bigl(f-\r u_t-\r v\cdot\na u\bigr)+ \D \dive u,
\eeno
which together with  the condition  $\dive u=\dive g$ yields
\beno
\|\na\Pi\|_{L^2}\leq C\bigl(\|\na\dive g\|_{L^2}+\|\sqrt{\r} u_t\|_{L^2}+\|v\|_{L^4}\|\na u\|_{L^4}+\|f\|_{L^2}\bigr),
\eeno
from which,  and the $u$ equation of \eqref{S2eq1}, we deduce that $\|\na^2 u\|_{L^2}$ satisfies the same estimate.
This together with \eqref{S2eq6}  ensures
 \beq \label{S2eq7}
 \|\na^2
u\|_{L^2}+\|\na\Pi\|_{L^2}\leq
C\bigl(\|\na\dive g\|_{L^2}+\|\sqrt{\r}u_t\|_{L^2}+\|v\|_{L^4}^2\|\na u\|_{L^2}+\|f\|_{L^2}\bigr).
\eeq
Hence we get, by using Young's inequality, that \beno
\begin{split} \bigl|\int_{\R^2}\r v\cdot\na u | u_t
\,\dx\bigr|\leq  C\bigl(\|v\|_{L^4}^4\|\na
u\|_{L^2}^2+\|\na\dive g\|_{L^2}^2+\|f\|_{L^2}^2\bigr)+\f16\|\sqrt{\r}u_t\|_{L^2}^2.
\end{split}\eeno
While we get, by using integration by parts and \eqref{S2eq7}, that
\beno
\begin{split}
\bigl|\int_{\R^2}\na\Pi | u_t\,\dx\bigr|
=&\bigl|\int_{\R^2}\na\Pi | g_t\,\dx\bigr|\leq \|\na\Pi\|_{L^2}\|g_t\|_{L^2}\\
\leq & C\bigl(\|v\|_{L^4}^4\|\na
u\|_{L^2}^2+\|(g_t,\na\dive g,f)\|_{L^2}^2\bigr)+\f16\|\sqrt{\r}u_t\|_{L^2}^2.
\end{split}\eeno
As a result,   it comes out
$$
\f{d}{dt}\|\na u(t)\|_{L^2}^2+\|\sqrt{\r}u_t\|_{L^2}^2\leq
C\bigl(\|v\|_{L^4}^4\|\na
u\|_{L^2}^2+\|(g_t,\na\dive g,f)\|_{L^2}^2\bigr).  $$
This together with \eqref{S2eq7} leads to \eqref{S2eq8}.
\end{proof}

\begin{prop}\label{S2col0}
{\sl Let   $(u,\na\Pi)$ be a smooth enough
solution of \eqref{S2eq1} with $g=0.$ Then  there hold for any $t\geq 0$,
\beq
\begin{split}
&\|\sqrt{\r} u\|_{L^\infty_t(L^2)}^2+\|\na
u\|_{L^2_t(L^2)}^2\leq C\bigl(\|\sqrt{\r_0}
u_0\|_{L^2}^2+\|f\|_{L^1_t(L^2)}^2\bigr),\\
&\|\na u\|_{L^\infty_t(L^2)}^2+\|(u_t,\na^2 u,\na
\Pi)\|_{L^2_t(L^2)}^2\leq C\bigl(\|\na
u_0\|_{L^2}^2+\|f\|_{L^2_t(L^2)}^2\bigr)\exp\bigl(C\|v_0\|_{L^2}^4\bigr),
\end{split} \label{S2eq2}\eeq
and
\beq\begin{split} \|\s^{\f12}\na
u\|_{L^\infty_t(L^2)}^2+&\bigl\|\s^{\f12}(u_t,\na^2 u,\na
\Pi)\bigr\|_{L^2_t(L^2)}^2\\
 \leq &
C\bigl(\|u_0\|_{L^2}^2+\|f\|_{L^1_t(L^2)}^2+\|\s^{\f12}f\|_{L^2_t(L^2)}^2\bigr)\exp\bigl(C\|v_0\|_{L^2}^4\bigr).\end{split}
\label{S2eq3} \eeq}
\end{prop}

\begin{proof}
By taking $L^2(\R^2)$ inner product of the $u$ equation of \eqref{S2eq1}
with $u$ and making use of the transport equation of \eqref{inNS}, we get
\beq\label{S2eq3t} \f12\f{d}{dt}\int_{\R^2}\r|u(t)|^2\,\dx+\|\na u\|_{L^2}^2=\int_{\R^2} f | u\,dx\leq\|f\|_{L^2}\|u\|_{L^2}.
\eeq
Integrating the above inequality over $[0,t]$ yields
\beno
\f12\|\sqrt{\rho}u\|_{L^\infty_t(L^2)}^2
+\|\na u\|_{L^2_t(L^2)}^2
\leq \f12\|\sqrt{\rho_0}u_0\|_{L^2}^2+C\|f\|_{L^1_t(L^2)}^2
+\f14\|\sqrt{\rho}u\|_{L^\infty_t(L^2)}^2,
\eeno
which leads to the first inequality of \eqref{S2eq2}.

In view of  \eqref{S2eq5}, the second inequality of \eqref{S2eq2} can be achieved by applying Gronwall's inequality to \eqref{S2eq8}.

While we get, by multiplying \eqref{S2eq8} by $\s(t)$ and then integrating the resulting inequality over $[0,t],$  that
\beq\label{S2eq9}\begin{split} \s(t)\|\na
u(t)\|_{L^2}^2&+\int_0^t\s(t')\bigl\|(u_t,\na^2
u,\na\Pi)\bigr\|_{L^2}\,\dt'\\
\leq &\int_0^t\|\na
u(t')\|_{L^2}^2\,\dt'+C\int_0^t\bigl(\|v\|_{L^4}^4\s(t')\|\na
u\|_{L^2}^2+\s(t')\|f(t')\|_{L^2}^2\bigr)\,\dt'.
\end{split} \eeq
Then in view of \eqref{S2eq5} and the first inequality of   \eqref{S2eq2}, we achieve \eqref{S2eq3} by applying
Gronwall's inequality.
\end{proof}

In particular, taking $u=v$ and $f=g=0$ in \eqref{S2eq3} gives rise to
\begin{col}\label{S2col1}
{\sl Let $(\r,v,\na\pi)$ be a smooth enough solution of
\eqref{inNS}. Then there holds \beq\begin{split} \|\s^{\f12}\na
v\|_{L^\infty_t(L^2)}^2+\bigl\|\s^{\f12}(v_t,\na^2 v,\na
\pi)\bigr\|_{L^2_t(L^2)}^2\leq
C\|v_0\|_{L^2}^2\exp\bigl(C\|v_0\|_{L^2}^4\bigr).\end{split}
\label{S2eq11} \eeq} \end{col}

\begin{lem}\label{S2lem2} {\sl Let $(u,\na\Pi)$ be a smooth enough solution of \eqref{S2eq1} with $g=0$ and $\tau(t)$ be a non-negative Lipschitz
function. Then we have
\beq \label{S2eq11q}
\begin{split}
\|\tau^{\f12}&u_t\|_{L^\infty_t(L^2)}^2+\|\tau^{\f12}\na
u_t\|_{L^2_t(L^2)}^2\leq
\exp\left(C\exp(C\|v_0\|_{L^2}^4)\right)\Bigl( \|\tau^{\f12}f_t\|_{L^1_t(L^2)}^2
\\
&
+\|\tau^{\frac 12}\sigma^{-\frac 12}f\|_{L^2_t(L^2)}^2
 + \|\tau^{\f12}\s^{-\f12}\na u\|_{L^\infty_t(L^2)}^2
 +\int_0^t\bigl(\tau'+\tau\s^{-1}\bigr)\|u_t\|_{L^2}^2\,\dt'\Bigr) .
\end{split}
\eeq}
\end{lem}

\begin{proof} We first get, by taking $\p_t$ to the $u$ equation of
\eqref{S2eq1}, that \beno \r \p_t^2u+\r v\cdot\na u_t-\D u_t+\na
\Pi_t=-\r_t(u_t+v\cdot\na u)-\r v_t\cdot\na u+f_t. \eeno Taking $L^2(\R^2)$
inner product of the above equation with $u_t$ gives
\beq \label{S2eq12}\f12\f{d}{dt}\|\sqrt{\r}u_t(t)\|_{L^2}^2+\|\na
u_t\|_{L^2}^2
=-\int_{\R^2}\bigl(\r_t(u_t+v\cdot\na u)+
 \r v_t\cdot\na u\bigr) | u_t\,\dx
 +\int_{\R^2}f_t | u_t\,\dx.
\eeq
By using the transport equation of \eqref{inNS} and
integrating by parts, one has \beno
\begin{split}
\bigl|\int_{\R^2}\r_tu_t | u_t\,\dx\bigr|=&2\bigl|\int_{\R^2}\r v
\cdot\na u_t | u_t\,\dx\bigr|
 \leq  C\|v\|_{L^4}\|\na u_t\|_{L^2}\|u_t\|_{L^4},
 \end{split}
 \eeno
 so that by applying the 2-D inequality \eqref{S2eq6} and Young's
 inequality, we obtain
 \beq \label{S2eq13}
\bigl|\int_{\R^2}\r_tu_t | u_t\,\dx\bigr|\leq \f16 \|\na
u_t\|_{L^2}^2+C\|v\|_{L^4}^4\| u_t\|_{L^2}^2. \eeq
Similarly, we get, by using integration by parts, that \beno
\begin{split}
\int_{\R^2}\r_t v\cdot\na u | u_t\,\dx=&\int_{\R^2}\r (v\cdot\na)
v\cdot\na u | u_t\,\dx\\
&+\int_{\R^2}\r (v\otimes v) : \na^2 u | u_t\,\dx +\int_{\R^2}
v\cdot\na u |\r v\cdot\na u_t\,\dx.
\end{split}
\eeno Applying 2-D interpolation inequality \eqref{S2eq6} and the
Estimate \eqref{S2eq7} gives \beno
\begin{split}
\bigl|\int_{\R^2}\r (v\cdot\na) v\cdot\na u | u_t\,\dx\bigr|\lesssim
&\|v\|_{L^4}\|\na v\|_{L^4}\|\na u\|_{L^4}\|u_t\|_{L^4}\\
\lesssim & \|v\|_{L^4}\|\na
v\|_{L^2}^{\f12}\left(\|v_t\|_{L^2}^{\f12}+\|v\|_{L^4}\|\na
v\|_{L^2}^{\f12}\right) \|\na
u\|_{L^2}^{\f12}\\
&\times\left(\| u_t\|_{L^2}^{\f12}+\|v\|_{L^4}\|\na
u\|_{L^2}^{\f12}+\|f\|_{L^2}^{\f12}\right)\|u_t\|_{L^2}^{\f12}\|\na u_t\|_{L^2}^{\f12},
\end{split}
\eeno from which and Young's inequality, we infer
$$\longformule{
\bigl|\int_{\R^2}\r (v\cdot\na v)\cdot\na u | u_t\,\dx\bigr|\leq
 \f1{18}\|\na u_t\|_{L^2}^2+C\bigl(\|v\|_{L^4}^4+\|\na
v\|_{L^2}^2\bigr)\| u_t\|_{L^2}^2}{{}+\|\na v\|_{L^2}^2\|f\|_{L^2}^2+C\left(\|v_t\|_{L^2}^2+\|v\|_{L^4}^4\|\na
v\|_{L^2}^2\right)\|\na u\|_{L^2}^2.} $$
Exactly along the same line, one has \beno
\begin{split}
\bigl|\int_{\R^2}\r (v\otimes v)& : \na^2 u | u_t\,\dx\bigr|\leq
C\|v\|_{L^8}^2\|\na^2 u\|_{L^2}\|u_t\|_{L^4}\\
\leq &C\|v\|_{L^2}^{\f12}\|\na
v\|_{L^2}^{\f32}\left(\| u_t\|_{L^2}+\|v\|_{L^4}^2\|\na
u\|_{L^2}+\|f\|_{L^2}\right)\|u_t\|_{L^2}^{\f12}\|\na u_t\|_{L^2}^{\f12}\\
\leq &\f1{18}\|\na u_t\|_{L^2}^2+C\|v\|_{L^2}^2\|\na
v\|_{L^2}^4\|\na
u\|_{L^2}^2\\
&\qquad\qquad
+ \|\na v\|_{L^2}^2\|f\|_{L^2}^2
+C\left(1+\|v\|_{L^2}^{2}\right)\|\na
v\|_{L^2}^2\| u_t\|_{L^2}^2,
\end{split}
\eeno and \beno
\begin{split}
\bigl|\int_{\R^2} v\cdot\na u | &\r v\cdot\na u_t\,\dx\bigr|\leq
C\|v\|_{L^8}^2\|\na u\|_{L^4}\|\na u_t\|_{L^2}\\
\leq &C\|v\|_{L^2}^{\f12}\|\na v\|_{L^2}^{\f32}\|\na
u\|_{L^2}^{\f12}\left(\|\sqrt{\r}u_t\|_{L^2}^{\f12}+\|v\|_{L^4}\|\na
u\|_{L^2}^{\f12}+\|f\|_{L^2}^{\f12}\right)\|\na u_t\|_{L^2}\\
\leq &\f1{18}\|\na u_t\|_{L^2}^2+C\left(\|v\|_{L^2}^2\|\na
v\|_{L^2}^4\|\na u\|_{L^2}^2+\|\na
v\|_{L^2}^2\bigl(\| u_t\|_{L^2}^2+\|f\|_{L^2}^2\bigr)\right).
\end{split}
\eeno
Hence it comes out
\beq\label{S2eq14}
\begin{split}
\bigl|\int_{\R^2}\r_t v\cdot\na u | u_t\,\dx\bigr|\leq & \f1{6}\|\na
u_t\|_{L^2}^2+C \bigl(1+\|v\|_{L^2}^{2}\bigr)\|\na
v\|_{L^2}^2 \| u_t\|_{L^2}^2
\\
&+C\left(\|v_t\|_{L^2}^2
+  \|v\|_{L^2}^2 \|\na
v\|_{L^2}^4\right)\|\na u\|_{L^2}^2+ C\|\na v\|_{L^2}^2\|f\|_{L^2}^2.
\end{split}
\eeq
Finally it is easy to observe that \beno
\begin{split}
\bigl|\int_{\R^2}\r v_t&\cdot\na u | u_t\,\dx\bigr|\leq
C\|v_t\|_{L^2}\|\na u\|_{L^4}\|u_t\|_{L^4}\\
\leq & C\| v_t\|_{L^2}\|\na
u\|_{L^2}^{\f12}\left(\| u_t\|_{L^2}^{\f12}+\|v\|_{L^4}\|\na
u\|_{L^2}^{\f12}+\|f\|_{L^2}^{\f12}\right)\| u_t\|_{L^2}^{\f12}\|\na
u_t\|_{L^2}^{\f12}.
\end{split}
\eeno Applying Young's inequality gives rise to \beq\label{S2eq15}
\begin{split}
\bigl|\int_{\R^2}\r v_t&\cdot\na u | u_t\,\dx\bigr|\leq
\f1{6}\|\na u_t\|_{L^2}^2
+\|v_t\|_{L^2}^2\|\na u\|_{L^2}^2\\
&+\s^{-1}(t)\bigl(\| u_t\|_{L^2}^2+\|f\|_{L^2}^2\bigr)
+C\left(\|v\|_{L^4}^4+\s(t)\|v_t\|_{L^2}^2\right)\| u_t\|_{L^2}^2.
\end{split}
\eeq
Inserting the Estimates \eqref{S2eq13}, \eqref{S2eq14} and
\eqref{S2eq15} into \eqref{S2eq12} gives rise to
 \beno \label{S2eq16}
\begin{split}
\f{d}{dt}\|\sqrt{\r}u_t(t)\|_{L^2}^2&+\|\na
u_t\|_{L^2}^2
\leq C\left( \bigl(1+\|v\|_{L^2}^{2}\bigr)\|\na
v\|_{L^2}^2+\s(t)\|v_t\|_{L^2}^2\right)\| u_t\|_{L^2}^2\\
&+C\left(\|v_t\|_{L^2}^2 +\|v\|_{L^2}^2\|\na
v\|_{L^2}^4\right)\|\na u\|_{L^2}^2\\
&+\s^{-1}(t)\| u_t\|_{L^2}^2
+C\bigl(\|\sigma^{\frac 12}\na v\|_{L^2}^2+1\bigr)\s^{-1}(t)\|f\|_{L^2}^2+\|f_t\|_{L^2}\|u_t\|_{L^2}.
\end{split}
\eeno
Multiplying the above inequality by $\tau(t)$ and then
applying Gronwall's inequality and Young's inequality to  the resulting inequality leads to
\beno
\begin{split}
\|\tau^{\f12}&u_t\|_{L^\infty_t(L^2)}^2+\|\tau^{\f12}\na
u_t\|_{L^2_t(L^2)}^2\leq C\Bigl(\|\tau^{\f12}f_t\|_{L^1_t(L^2)}^2+
\bigl(\|\sigma^{\frac 12}\nabla v\|_{L^\infty_t(L^2)}^2+1\bigr)
\|\tau^{\frac 12}\sigma^{-\frac 12}f\|_{L^2_t(L^2)}^2
\\
&+\bigl(\|\s^{\f12}v_t\|_{L^2_t(L^2)}^2
+
  \|v\|_{L^\infty_t(L^2)}^2
  \|\s^{\f12}\na v\|_{L^\infty_t(L^2)}^2
\|\na v\|_{L^2_t(L^2)}^2 \bigr)\|\tau^{\f12}\s^{-\f12}\na u\|_{L^\infty_t(L^2)}^2\\
&+\int_0^t\bigl(\tau'+\tau\s^{-1}\bigr)\|u_t\|_{L^2}^2\,\dt'
\Bigr)\times\exp\Bigl(C\int_0^t\bigl( \bigl(1+\|v\|_{L^2}^{2}\bigr)\|\na
v\|_{L^2}^2+\s \|v_t\|_{L^2}^2\bigr)\,\dt'\Bigr),
\end{split}
\eeno
from which, \eqref{S2eq4} and \eqref{S2eq11}, we conclude the proof of \eqref{S2eq11q}.
\end{proof}

\begin{prop}\label{S2colu}
{\sl Let $(u,\na\Pi)$ be a smooth enough solution of \eqref{S2eq1} with   $f=g=0.$ Then one has \beq
\label{S2eq10} \begin{split} &\bigl\|\s^{\f12}(u_t,\na^2 u,\na
\Pi)\bigr\|_{L^\infty_t(L^2)}^2+\|\s^{\f12}\na
u_t\|_{L^2}^2\leq  \|\na
u_0\|_{L^2}^2\exp\left(C\exp\bigl(C\|v_0\|_{L^2}^4\bigr)\right),
\end{split} \eeq
and
\beq\label{S2eq10.1}
 \begin{split}
 \bigl\|\sigma(u_t,\na^2 u,\na\Pi)\bigr\|_{L^\infty_t(L^2)}^2
 +\|\s \na u_t\|_{L^2_t(L^2)}^2
 \leq
 \|u_0\|_{L^2}^2\exp\left(C\exp\bigl(C\|v_0\|_{L^2}^4\bigr)\right).
\end{split} \eeq} \end{prop}

\begin{proof}
Taking $f=0$ and $\tau(t)=\s(t)$ in \eqref{S2eq11q} and using \eqref{S2eq2} gives
 \beq \label{S2eq17}
\|\s^{\f12}u_t\|_{L^\infty_t(L^2)}^2+\int_0^t\s(t')\|\na
u_t\|_{L^2}^2\,\dt'\leq \|\na
u_0\|_{L^2}^2\exp\left(C\exp\bigl(C\|v_0\|_{L^2}^4\bigr)\right).
 \eeq
Resuming the Estimate \eqref{S2eq17} into \eqref{S2eq7} and using
\eqref{S2eq4}, \eqref{S2eq2} and \eqref{S2eq11} yields \beno \begin{split}
\bigl\|\s^{\f12}(\na^2u,\na\Pi)\bigr\|_{L^\infty_t(L^2)}^2 \leq &
C\left(\|\s^{\f12}u_t\|_{L^\infty_t(L^2)}^2+\|v\|_{L^\infty_t(L^2)}^2\|\s^{\f12}\na
v\|_{L^\infty_t(L^2)}^2\|\na u\|_{L^\infty_t(L^2)}^2\right)\\ \leq&
C\|\na
u_0\|_{L^2}^2\exp\left(C\exp\bigl(C\|v_0\|_{L^2}^4\bigr)\right).
\end{split}
 \eeno
This achieves \eqref{S2eq10}.

While by taking $f=0$ and $\tau(t)=\s^2(t)$ in \eqref{S2eq11q} and using \eqref{S2eq3}   leads to
 \beno
\|\s u_t\|_{L^\infty_t(L^2)}^2
+\int_0^t\sigma^2\|\na u_t\|_{L^2}^2\,\dt'
\leq
\| u_0\|_{L^2}^2\exp\left(C\exp\bigl(C\|v_0\|_{L^2}^4\bigr)\right).
 \eeno
And hence thanks to   \eqref{S2eq7} and \eqref{S2eq4}, \eqref{S2eq3}, \eqref{S2eq11}, we obtain
  \beno \begin{split}
\bigl\|\s (\na^2u,\na\Pi)\bigr\|_{L^\infty_t(L^2)}^2 \leq &
C\left(\|\s u_t\|_{L^\infty_t(L^2)}^2
+\|v\|_{L^\infty_t(L^2)}^2
\|\s^{\f12}\na v\|_{L^\infty_t(L^2)}^2
\|\sigma^{\frac 12}\na u\|_{L^\infty_t(L^2)}^2\right)
\\ \leq&
 \| u_0\|_{L^2}^2\exp\left(C\exp\bigl(C\|v_0\|_{L^2}^4\bigr)\right).
\end{split}
 \eeno
This yields \eqref{S2eq10.1}, and we  complete the proof of the
proposition.
\end{proof}

In particular, taking $u=v$ in \eqref{S2eq10.1} gives rise to
\begin{col}\label{S2col0}
{\sl
The following estimate is valid for  smooth enough solution $(\r,v,\na\pi)$   of
\eqref{inNS}
\beq\label{S2eq11.1}\begin{split}
 \|\s(v_t, \nabla^2 v, \nabla \pi)\|_{L^\infty_t(L^2)}^2
 +\|\s\nabla v_t\bigr\|_{L^2_t(L^2)}^2
&\leq
 \exp\left(C\exp\bigl(C\|v_0\|_{L^2}^4\bigr)\right).
\end{split}
 \eeq}
\end{col}

%%%%%%%%%%%%%%%%%%%%%
\subsection{Time-weighted $\cB^s$ estimate of $u$}\label{subs:Bu}
%%%%%%%%%%%%%%%%%%%%%%%%%%%%

\begin{prop}\label{S2prop1}
{\sl Let $(\r,v,\na\pi)$ be a smooth enough solution of \eqref{inNS}
and $(u,\na\Pi)$ solve \eqref{S2eq1} with $f=g=0$ and with initial data $u_0\in \cB^s$ for some  $s\in ]0,1[$.
Then
there holds \ben &&\|u\|_{\wt{L}^\infty_t(\cB^s)}+\|\na
u\|_{\wt{L}^2_t(\cB^s)}\leq
C\|u_0\|_{\cB^s}\exp\bigl(C\|v_0\|_{L^2}^4\bigr), \label{S2eq18}\\
&&\|{\s}^{\f12}\na u\|_{\wt{L}^\infty_t(\cB^s)}+\|{\s}^{\f12}
u_t\|_{\wt{L}^2_T(\cB^s)}\leq
\cC(v_0,u_0,s),\label{S2eq19}\\
&&\|\s^{\f{1-s}2}
u\|_{\wt{L}^\infty_t(\cB^1)}+\bigl\|\s^{\f{1-s}2}(u_t,\na^2
u,\na\Pi)\bigr\|_{L^2_t(L^2)}\leq
\cC(v_0,u_0,s),\label{S2eq20}
\een where $\cC(v_0,u_0,s)$ is given by \eqref{S1eq9}.}
\end{prop}

\begin{proof} Let $(u_q, \na\Pi_q)$ be determined by
\begin{equation}\label{S2eq20a}
 \left\{\begin{array}{l}
\displaystyle\rho \p_t u_q+\r v\cdot\na u_q-\Delta u_q+\nabla \Pi_q=0,\\
\displaystyle \div u_q=0,\\
\displaystyle   u_q|_{t=0}=\D_q u_0.
\end{array}\right.
\end{equation}
Then by the uniqueness of solution to \eqref{S2eq1} with $f=g=0,$ we have
\beno\label{S2eq21}
u=\sum_{q\in\Z} u_q\andf \na\Pi=\sum_{q\in\Z}\na\Pi_q.
\eeno
While it follows from Definition \ref{S0def1}, Bernstein type inequality
(see \cite{BCD}) and \eqref{S2eq2} that
\beq\label{S2eq21a}
\|u_q\|_{L^\infty_t(L^2)}^2+\|\na u_q\|_{L^2_t(L^2)}^2
\leq C\|\D_q u_0\|_{L^2}^2
\lesssim d_q^2 2^{-2qs}\|u_0\|_{\cB^s}^2,
\eeq
 and
\beq\label{S2eq22}
\begin{split}
\|\na u_q\|_{L^\infty_t(L^2)}^2
+\bigl\|(\p_tu_q,\na^2 u_q,\na\Pi_q)\bigr\|_{L^2_t(L^2)}^2
\leq & C\|\D_q u_0\|_{\dH^1}^2\exp\bigl(C\|v_0\|_{L^2}^4\bigr)
\\
\lesssim
&d_q^2 2^{2q(1-s)}\|u_0\|_{\cB^s}^2\exp\bigl(C\|v_0\|_{L^2}^4\bigr).
\end{split} \eeq
Exploiting the idea used in the proof of the characterization of Besov
space with positive index (see Lemma 2.88 of \cite{BCD} for
instance) gives for any $j\in\Z$ that \beno
\begin{split}
\|\D_j u\|_{L^\infty_t(L^2)}+\|\D_j\na u\|_{L^2_t(L^2)}\leq
&\sum_{q>j}\left(\|\D_j u_q\|_{L^\infty_t(L^2)}
+\|\D_j\na u_q\|_{L^2_t(L^2)}\right)
\\
&+2^{-j}\sum_{q\leq j}
\left(\|\D_j\na u_q\|_{L^\infty_t(L^2)}
+\|\D_j\na^2 u_q\|_{L^2_t(L^2)}\right)
\\
\lesssim &\sum_{q>j}
\left(\| u_q\|_{L^\infty_t(L^2)}+\|\na u_q\|_{L^2_t(L^2)}\right)
\\
&+2^{-j}\sum_{q\leq j}
\left(\|\na u_q\|_{L^\infty_t(L^2)}+\|\na^2 u_q\|_{L^2_t(L^2)}\right)
\\
\lesssim &d_j
2^{-j s}\|u_0\|_{\cB^s}\exp\bigl(C\|v_0\|_{L^2}^4\bigr),
\end{split}
\eeno which together with Definition \ref{S0def1} ensures the Inequality \eqref{S2eq18}.

Similarly, we deduce from \eqref{S2eq3} that
\beq\label{S2eq23}\begin{split}
 \|\s^{\f12}\na u_q\|_{L^\infty_t(L^2)}^2
+\bigl\|\s^{\f12}(\p_tu_q,\na^2u_q,\na
\Pi_q)\bigr\|_{L^2_t(L^2)}^2
&\leq
C\|\D_q u_0\|_{L^2}^2\exp\bigl(C\|v_0\|_{L^2}^4\bigr)
\\
&\lesssim
d_q^22^{-2qs}\|u_0\|_{\cB^s}^2\exp\bigl(C\|v_0\|_{L^2}^4\bigr).\end{split}
 \eeq
 And it follows from \eqref{S2eq10} that
 \beq\label{S2eq23a} \begin{split}
 \|\s^{\f12}\na^2
u_q\|_{L^\infty_t(L^2)}^2+\|\s^{\f12}\na \p_tu_q\|_{L^2_t(L^2)}^2
&\leq \|\na\D_q u_0\|_{L^2}^2\exp\left(C\exp\bigl(C\|v_0\|_{L^2}^4\bigr)\right)\\
&\lesssim
d_q^22^{2q(1-s)}\cC^2(v_0,u_0,s).
\end{split} \eeq
Then exactly along the same line of the proof to the  Inequality
\eqref{S2eq18}, we achieve  the Inequality \eqref{S2eq19}.

Finally we deduce from  the following interpolation inequality in Besov spaces:
  \beno \|\s^{\f{1-s}2} u\|_{\wt{L}^\infty_t(\cB^1)}\lesssim
\|u\|_{\wt{L}^\infty_t(\cB^s)}^{s}\|\s^{\f12}\na
u\|_{\wt{L}^\infty_t(\cB^s)}^{1-s}, \eeno
 and
\eqref{S2eq18} and \eqref{S2eq19}, that \beq\label{S2eq24}
\|\s^{\f{1-s}2} u\|_{\wt{L}^\infty_t(\cB^1)} \leq
\cC(v_0,u_0,s).
\eeq
Whereas in view of \eqref{S2eq22} and \eqref{S2eq23}, we have
\beno
\begin{split}
\int_{\R^+}\s(t)^{1-s}\|\p_tu_q\|_{L^2}^2\,\dt
\leq &
\Bigl(\int_{\R^+}\s(t)\|\p_tu_q(t)\|_{L^2}^2\,\dt\Bigr)^{1-s}
\Bigl(\int_{\R^+}\|\p_tu_q(t)\|_{L^2}^2\,\dt\Bigr)^s
\\
\lesssim &
d_q^2\cC^2(v_0,u_0,s),
\end{split}
\eeno
which implies \beno
\begin{split}
\Bigl(\int_{\R^+}\s(t)^{1-s}\|u_t\|_{L^2}^2\,\dt\Bigr)^{\f12}\leq &
\sum_{q\in\Z}
\Bigl(\int_{\R^+}\s(t)^{1-s}\|\p_tu_q\|_{L^2}^2\,\dt\Bigr)^{\f12}
\lesssim \cC(v_0,u_0,s).
\end{split}
\eeno The same estimate holds for
$\bigl\|\s^{\f{1-s}2}(\na^2 u,\na\Pi)\bigr\|_{L^2_t(L^2)}.$
This together with \eqref{S2eq24} ensures \eqref{S2eq20}. We thus
complete the proof of the proposition.
\end{proof}

Let us remark that the estimate of $\|\s^{\f{1-s}2}u_t\|_{L^2_t(L^2)}$ presented in Proposition \ref{S2prop1} will be crucial for us to deal with the energy
estimate of $u_t.$

\begin{col}\label{S2col2}
{\sl Under the same assumptions of Proposition \ref{S2prop1}, one
has  \beq \label{S2eq25}
\begin{split}
\bigl\|\s^{1-\f{s}2}(u_t,\na^2 u,\na
\Pi)\bigr\|_{L^\infty_t(L^2)}
+\|\s^{1-\f{s}2}\na u_t\|_{L^2_t(L^2)}
\leq \cC(v_0,u_0,s).
\end{split}
\eeq
}
\end{col}

\begin{proof} Taking $f=0$ and $\tau(t)=\s(t)^{2-s}$ in Lemma \ref{S2lem2} gives rise to
 \beno
\begin{split}
\bigl\|&\s^{1-\f{s}2} u_t\|_{L^\infty_t(L^2)}^2+\|\s^{1-\f{s}2}\na
u_t\|_{L^2_t(L^2)}^2\\
&\qquad\leq
 \exp\bigl(C\exp\bigl(C \|v_0\|_{L^2}^4\bigr)\bigr)
\Bigl(\|\s^{\f{1-s}2}\na
u\|_{L^\infty_t(L^2)}^2
+\|\s^{\f{1-s}2} u_t\|_{L^2_t(L^2)}^2\Bigr) ,
\end{split}
\eeno which together with  Proposition
\ref{S2prop1} ensures that
 \beq \label{S2eq27}
\|\s^{1-\f{s}2}u_t\|_{L^\infty_t(L^2)}+\|\s^{1-\f{s}2}\na
u_t\|_{L^2_t(L^2)}\leq
\cC(v_0,u_0,s).
\eeq As a result, by virtue of  \eqref{S2eq7} and Corollary \ref{S2col1},
Proposition \ref{S2prop1}, we obtain \beno
\begin{split}
\bigl\|\s^{1-\f{s}2}&(\na^2 u,\na\Pi)\bigr\|_{L^\infty_t(L^2)}
\lesssim
\|\s^{1-\f{s}2}u_t\|_{L^\infty_t(L^2)}\\
&+\|v\|_{L^\infty_t(L^2)}\|\s^{\f12}\na
v\|_{L^\infty_t(L^2)}\|\s^{\f{1-s}2}\na u\|_{L^\infty_t(L^2)}
\lesssim  \cC(v_0,u_0,s).
\end{split}
\eeno By summing up \eqref{S2eq27} and the above inequality, we
achieve \eqref{S2eq25}.
\end{proof}

Thanks to \eqref{S2eq4}, Proposition \ref{S2prop1} and Corollary \ref{S2col2}, we conclude the proof of  Proposition \ref{S1prop1} by taking $u=v$ and $f=g=0$ in \eqref{S2eq1}.
Next let us present the proof of Corollary \ref{S1col1}.

\begin{proof}[Proof of Corollary \ref{S1col1}] We first deduce from the 2-D interpolation inequality \eqref{S2eq6}  that  for any $r\in [2,+\infty[$,
 \beno \begin{split}
 &\bigl\|\s^{\frac{1-s_0}{2}\left(1-\frac
2r\right)}v\bigr\|_{L^\infty_t(L^r)}\lesssim
 \| v\|_{L^\infty_t(L^2)}^{\f2r} \bigl\|\s^{\f{1-s_0}2}
 v\bigr\|_{L^\infty_t(L^\infty)}^{1-\f2r};\\
&\bigl\|\sigma ^{\frac{1-s_0}{2}}\nabla
v\bigr\|_{L^{\frac{2r}{r-2}}_t(L^r)}
\lesssim
\bigl\|\sigma ^{\frac{1-s_0}{2}}\nabla
v\bigr\|_{L^{\infty}_t(L^2)}^{\f2r}
\bigl\|\sigma ^{\frac{1-s_0}{2}}\nabla^2
v\bigr\|_{L^{2}_t(L^2)}^{1-\f2r};\\
 &\bigl\|\s^{\left(1-\f1r-\f{s_0}2\right)}\na v\bigr\|_{L^\infty_t(L^r)}
 \lesssim
 \bigl\|\s^{\f{1-s_0}2}\na v\bigr\|_{L^\infty_t(L^2)}^{\f2r} \bigl\|\s^{1-\f{s_0}2}\na^2
 v\bigr\|_{L^\infty_t(L^2)}^{1-\f2r};\\
&\bigl\|\s^{\left(1-\f1r-\f{s_0}2\right)}
v_t\bigr\|_{L^2_t(L^r)}\lesssim
 \bigl\|\s^{\f{1-s_0}2} v_t\bigr\|_{L^2_t(L^2)}^{\f2r} \bigl\|\s^{1-\f{s_0}2}\na
 v_t\bigr\|_{L^2_t(L^2)}^{1-\f2r};\\
&\bigl\|\s^{1-\f{s_0}2} v_t\bigr\|_{L^{\f{2r}{r-2}}_t(L^r)}\lesssim
 \bigl\|\s^{1-\f{s_0}2} v_t\bigr\|_{L^\infty_t(L^2)}^{\f2r} \bigl\|\s^{1-\f{s_0}2}\na
 v_t\bigr\|_{L^2_t(L^2)}^{1-\f2r},
 \end{split}
 \eeno
 from which and  Proposition \ref{S1prop1}, we infer
 \beq \label{S2eq29}
 \begin{split}
 &\bigl\|\sigma^{\frac{1-s_0}{2}}v\bigr\|_{L^\infty_t(L^\infty)}
 +\bigl\|\s^{\frac{1-s_0}{2}\left(1-\frac 2r\right)}v\bigr\|_{L^\infty_t(L^r)}
 +\bigl\|\sigma^{\frac{1-s_0}{2}}\nabla
v\bigr\|_{L^{\frac{2r}{r-2}}_t(L^r)}\\
&
+\bigl\|\s^{\left(1-\f1r-\f{s_0}2\right)}\na
v\bigr\|_{L^\infty_t(L^r)}
 +\bigl\|\s^{\left(1-\f1r-\f{s_0}2\right)}
v_t\bigr\|_{L^2_t(L^r)}
 +\bigl\|\s^{1-\f{s_0}2} v_t\bigr\|_{L^{\f{2r}{r-2}}_t(L^r)}\lesssim
\cC(v_0,s_0).
 \end{split}
 \eeq

While it follows from   the classical estimate of Stokes
operator, \eqref{S2eq6} and the 2-D interpolation inequality
$\|a\|_{L^{2r}}\lesssim \|a\|_{L^r}^{1-\f1r}\|\na a\|_{L^r}^{\f1r},$   that
\beno
\begin{split}
\bigl\| (\na^2 v,\na \pi)\bigr\|_{L^r}
&\lesssim
\|  v_t\|_{L^r}
+\|v\|_{L^{2r}}
\| \na v\|_{L^{2r}}
\\
&\lesssim
\|  v_t\|_{L^r}
+\|v\|_{L^2}^{\frac 1r}\|\nabla v\|_{L^2}^{1-\frac 1r}
\| \na v\|_{L^{r}}^{1-\frac{1}{r}}
\|\nabla^2 v\|_{L^r}^{\frac 1r},
\end{split}
\eeno
which implies
\beno
\begin{split}
\| (\na^2 v,\na \pi)\|_{L^r}
\lesssim
\| v_t\|_{L^r}
+\|v\|_{L^2}^{\frac{1}{r-1}}\|\nabla v\|_{L^2}
\| \na v\|_{L^{r}}  .
\end{split}
\eeno
We then deduce from \eqref{S2eq29}  that
\beq\label{S2eq34}
 \begin{split}
\bigl\|\s^{\left(1-\f1r-\f{s_0}2\right)}&(\na^2 v, \na
\pi)\bigr\|_{L^2_t(L^r)}
\lesssim \bigl\|\s^{\left(1-\f1r-\f{s_0}2\right)}
v_t\bigr\|_{L^2_t(L^r)}\\
&+
\|v\|_{L^\infty_t(L^2)}^{\frac{1}{r-1}}
\|\nabla v\|_{L^2_t(L^2)}
\bigl\|
\sigma^{\left( 1-\frac 1r-\frac{s_0}2\right)} \nabla v\bigr\|_{L^{\infty}_t(L^r)}
\lesssim \cC(v_0,s_0).
\end{split}
\eeq
Along the same line, it follows from  Corollary \ref{S2col1} and \eqref{S2eq29} that
\beno\label{S2eq33}
\begin{split}
\bigl\|\s^{1-\f{s_0}2}&(\na^2 v,\na
\pi)\bigr\|_{L^{\f{2r}{r-2}}_t(L^r)}
\lesssim \bigl\|\s^{1-\f{s_0}2}
v_t\bigr\|_{L^{\f{2r}{r-2}}_t(L^r)}\\
&+\| v\|_{L^\infty_t(L^2)}^{\frac{1}{r-1}}
\|\s^{\f{1}2}\nabla v\|_{L^\infty_t(L^2)}
\bigl\|\s^{\f{1-s_0}2}
\na v\bigr\|_{L^{\f{2r}{r-2}}_t(L^r)} \lesssim \cC(v_0,s_0).
\end{split}
 \eeno
Finally taking $r=\f2{1-s}$ in \eqref{S2eq34} and using the 2-D interpolation inequality
 \beno
 \|a\|_{L^\infty}\lesssim \|a\|_{L^{\f2{1-s}}}^s\|\na
 a\|_{L^{\f2{1-s}}}^{1-s}\lesssim \|a\|_{\dH^s}^s\|\na
 a\|_{L^{\f2{1-s}}}^{1-s}, \eeno
 and Proposition \ref{S1prop1}, we infer
 \beno
 \begin{split}
\bigl\| \sigma^{\frac{1-s_0}2} \nabla v\bigr\|_{L^{2}_t(L^\infty)}
\lesssim &\|\na v\|_{L^2_t(\dH^{s_0})}^{s_0}\bigl\|\s^{\f12}\na^2
v\bigr\|_{L^2_t(L^{\f2{1-s_0}})}^{1-s_0}
 \lesssim \cC(v_0,s_0).
 \end{split}
 \eeno
This ends the proof of Corollary \ref{S1col1}.
\end{proof}

%%%%%%%%%%%%%%%%%%%%%
\subsection{Time-weighted $H^1$ estimate of $D_t v$}\label{subs:Dv'}
%%%%%%%%%%%%%%%%%%%%%%%%%%%%

The main result of this subsection is to perform  the time-weighted $\dot H^1$ estimate of $D_t v.$ To the end, let us first present the following lemma:

\begin{lem}\label{lem:Dv}
{\sl Under the assumptions of   Proposition \ref{S1prop1},
for any  $ r\in [2,\infty[,$ we have
\begin{equation}\label{S4eq4p}
\begin{split}
& \| \sigma^{1 -\frac{s_0}2}
 (D_t v, v\otimes \nabla v ) \|_{L^\infty_t(L^2)}
 +\| \sigma ^{ \frac {3-s_0}2}
 (\nabla v\otimes \nabla v, v\otimes\nabla^2 v ) \|_{L^\infty_t(L^2)}
\\
&+
 \bigl\| \sigma^{1 -\frac{s_0}2} \bigl(\nabla v_t, \nabla D_t v, D_t\nabla v,
\nabla v\otimes\nabla v, v\otimes\nabla^2 v\bigr) \bigr\|_{L^2_t(L^2)}
\\
&+
 \bigl\|\sigma^{\frac{3-s_0}2}
 \bigl(v\otimes(\nabla v_t, \nabla D_t v), \nabla v\otimes\nabla^2 v,
\nabla v\otimes\nabla \pi, D_t(v\otimes\nabla v)\bigr ) \bigr\|_{L^2_t(L^2)}\\
&+\bigl\|\sigma^{\left(1-\frac 1r-\frac{s_0}2\right)}
 \bigl(v_t, \nabla^2 v, \nabla \pi, D_t v, v\otimes\nabla v\bigr)
\bigr\|_{L^2_t(L^r)}
\leq \cC(v_0,s_0).
\end{split}
\end{equation}  }
\end{lem}
\begin{proof}
By virtue  of \eqref{S2eq6} and
$$
\|a\|_{L^\infty}\leq C\|a\|_{L^2}^{\frac 12}\|\nabla^2 a\|_{L^2}^{\frac 12},
$$
 we deduce from \eqref{S2eq5}, \eqref{S2eq11}, \eqref{S2eq11.1} that for any $t\in \R^+$ and $r\in [2,\infty[$
 \beq\label{bound:v}\begin{split}
& \|v\|_{L^4_t(L^4)}
 +\|\sigma^{ \frac 12-\frac 1r }v\|_{L^\infty_t(L^r)}
+ \|\sigma^{\frac 12} v\|_{L^\infty_t(L^\infty)}+\|\sigma^{\frac 12}\nabla v\|_{L^4_t(L^4)}
\\
&+\|\sigma^{1-\frac 1r}\nabla v\|_{L^\infty_t(L^r)}+\|\sigma(v_t, \nabla^2v,  v\otimes\nabla v)\|_{L^\infty_t(L^2)}
 \leq \exp\left(C\exp\bigl(C\|v_0\|_{L^2}^4\bigr)\right) .
\end{split}\eeq
It is easy to observe that
\begin{equation*}
\begin{split}
 \bigl\| \sigma^{1 -\frac {s_0}2}
 (D_t v, v\otimes\nabla v ) \bigr\|_{L^\infty_t(L^2)}
& \leq
 \| \sigma^{1 -\frac{s_0}2}v_t \|_{L^\infty_t(L^2)}
+ \| \sigma^{\frac 14}v \|_{L^\infty_t(L^4)}
 \| \sigma^{ \frac 34 -\frac {s_0}2 }\nabla v \|_{L^\infty_t(L^4)},
\end{split}
\end{equation*}
and
\begin{align*}
 \bigl\| \sigma ^{ \frac {3-s_0}2}
 (\nabla v\otimes \nabla v, v\otimes\nabla^2 v ) \bigr\|_{L^\infty_t(L^2)}
\leq &
 \| \sigma ^{\frac 34 }\nabla v  \|_{L^\infty_t(L^4)}
 \|\sigma^{ \frac 34-\frac {s_0}2 }\nabla v\|_{L^\infty_t(L^4)}\\
& + \| \sigma^{\frac 12} v  \|_{L^\infty_t(L^\infty)}
 \| \sigma ^{1-\frac{s_0}2} \nabla^2 v \|_{L^\infty_t(L^2)}.
\end{align*}
Hence we deduce from Proposition \ref{S1prop1},  Corollary \ref{S1col1} and \eqref{bound:v}  that
\beno
\| \sigma^{1 -\frac {s_0}2}
 (D_t v, v\otimes\nabla v ) \|_{L^\infty_t(L^2)}+\bigl\| \sigma ^{ \frac {3-s_0}2}
 (\nabla v\otimes \nabla v, v\otimes\nabla^2 v ) \bigr\|_{L^\infty_t(L^2)}\leq \cC(v_0,s_0).
 \eeno
Exactly along the same line, we have
\begin{equation*}
\begin{split}
 \bigl\| \sigma^{ \frac {1-s_0}2}  (D_t v, v\otimes\nabla v) \bigr\|_{L^2_t(L^2)}
&\leq
 \| \sigma^{ \frac {1-s_0 }2}  v_t \|_{L^2_t(L^2)}
+
 \|  v \|_{L^4_t(L^4)}
 \| \sigma^{ \frac {1-s_0 }2}\nabla v\|_{L^4_t(L^4)}\leq
\cC(v_0,s_0),
 \end{split}
\end{equation*}
and
\begin{equation*}
\begin{split}
 \bigl\| &\sigma ^{1 -\frac{s_0}2}\bigl(\nabla D_t v, D_t \nabla v,
\nabla v\otimes\nabla v, v\otimes\nabla^2 v\bigr)\bigr\|_{L^2_t(L^2)}
\leq
 \| \sigma^{1 -\frac {s_0}2}\nabla v_t \|_{L^2_t(L^2)}
\\
&\quad+ \| \sigma^{\frac 12}\nabla v \|_{L^4_t(L^4)}
 \| \sigma^{ \frac {1-s_0}2}\nabla v \|_{L^4_t(L^4)}+ \| \sigma^{\frac 12}  v \|_{L^\infty_t(L^\infty)}
 \| \sigma^{\frac {1-s_0}2}\nabla^2 v\|_{L^2_t(L^2)}\leq
\cC(v_0,s_0),
\end{split}
\end{equation*}
from which and \eqref{S2eq6}, we infer   for any $r\in [2,\infty[$,
\begin{equation*}
\begin{split}
 \bigl\| \sigma ^{\left(1-\frac 1r -\frac{s_0}2\right)}
\bigl( D_t v,  v\otimes\nabla v \bigr)\bigr\|_{L^2_t(L^r)}
\leq &
 C\| \sigma^{ \frac {1-s_0}2}
  (D_t v, v\otimes\nabla v) \|_{L^2_t(L^2)}^{\frac 2r}\\
  &\times
  \| \sigma^{1-\frac{s_0}2}
  (\nabla D_t v, \nabla(v\otimes\nabla v)) \|_{L^2_t(L^2)}  ^{1-\frac 2r}
  \leq
\cC(v_0,s_0),
\end{split}
\end{equation*}
and
\begin{align*}
\bigl\|\sigma^{ \frac {3-s_0}2}
\bigl( v\otimes(\nabla v_t, \nabla D_t v)\bigr)\bigr\|_{L^2_t(L^2)}
   \leq
\|\sigma ^{ \frac{1 }{2}} v\|_{L^\infty_t(L^\infty)}
\|\sigma ^{1-\frac{s_0}{2}}
(\nabla v_t, D_t\nabla v)\|_{L^2_t(L^2)}
\leq \cC(v_0,s_0).
\end{align*}
Moreover, we deduce from  Proposition \ref{S2prop1} and Corollary \ref{S2col1} that
\begin{equation*}\label{bound:FD}
\begin{split}
\bigl\|\sigma^{ \frac {3-s_0}2}
\bigl(\nabla v\otimes\nabla^2 v,
   \nabla v\otimes\nabla \pi\bigr)\bigr\|_{L^2_t(L^2)}
\leq &
 \| \sigma^{\frac 12} \nabla v\|_{L^4_t(L^4)}
\|\sigma^{1-\frac{s_0}2} (\nabla^2 v,\nabla \pi)\|_{L^4_t(L^4)}
\leq
\cC(v_0,s_0),
\end{split}
\end{equation*}
and
\begin{equation*}\label{bound:FD}
\begin{split}
\bigl\|\sigma ^{ \frac {3-s_0}2}
 D_t(v\otimes\nabla v)\bigr\|_{L^2_t(L^2)}
 \leq &
 \| \sigma D_t v \|_{L^\infty_t(L^2)}
 \|\sigma ^{\frac {1-s_0}2}\nabla v\|_{L^2_t(L^\infty)}\\
&+
 \| \sigma ^{\frac 12}  v \|_{L^\infty_t(L^\infty)}
 \|\sigma ^{1-\frac {s_0}2}D_t\nabla v\|_{L^2_t(L^2)}
\leq
 \cC(v_0,s_0).
\end{split}
\end{equation*}
This completes the proof of the Lemma.
\end{proof}

Before proceeding the time-weighted $\dH^1$ energy estimate, let us first present the proof of Lemma \ref{S4lem01}.

\begin{proof}[Proof of Lemma \ref{S4lem01}]
By taking $L^2(\R^2)$ inner product of \eqref{S8eq1} with $D_t \rw,$ one has
\begin{equation*}\label{energy:v1''}
\begin{split}
\frac 12\frac{d}{dt}\|\sqrt{\rho}  D_t \rw(t)\|_{L^2}^2
+\|\nabla D_t \rw\|_{L^2}^2
=\int_{\R^2} ({\rm F}- \nabla\rm{q})  |   D_t \rw\,\dx.
\end{split}
\end{equation*}
Multiplying the above equality   by $\sigma(t)^{2-s}$ and then integrating the resulting equality over $[0,t]$ leads to
\begin{equation}\label{energy:Dv1}
\begin{split}
&\frac 12\|\sigma^{1-\frac{s}2}\sqrt{\rho}D_t \rw(t)\|_{L^2}^2
+ \|\sigma^{1-\frac{s}2} \nabla D_t \rw\|_{L^2_t(L^2)}^2
\\
&=
\frac{2-s}{2}\|\sigma^{\frac {1-s}2}
\sqrt{\rho}D_t \rw\|_{L^2_t(L^2)}^2+\int^t_0 \int_{\R^2}\sigma^{2-s}( {\rm F}- \nabla \rm{q}) |  D_t \rw\,\dx\,\dt'.
\end{split}
\end{equation}
It is easy to observe that
\begin{align*}\label{Dv1:F1}
\bigl|\int^t_0 \int_{\R^2}\sigma^{2-s} {\rm F} | D_t \rw\,\dx\,\dt'\bigr|
&\leq
\|\sigma^{\frac{3-s}{2}}{\rm F}\|_{L^2_t(L^2)}^2
+\|\sigma^{\frac{1-s}{2}}D_t \rw\|_{L^2_t(L^2)}^2.
\end{align*}
While due to
$\div D_t \rw=\div \frak{a},$  by virtue of \eqref{S8eq1}, we write
\begin{align*}
-\int^t_0 &\int_{\R^2}\sigma^{2-s} \nabla {\rm q} | D_t \rw\,\dx\,\dt'=
 \int^t_0\int_{\R^2} \sigma^{2-s}
 (\rho D_t^2\rw-\Delta D_t \rw -{\rm F})
  |  \frak{a}\,\dx\,\dt'.
\end{align*}
As a result, for any $\e>0,$ we obtain
\beno
\begin{split}
\bigl|\int^t_0\int_{\R^2}\sigma^{2-s}
 \nabla {\rm q} |  D_t\rw\,\dx\,\dt'&\bigr|
 \leq
\frac 12 \Bigl(\|\sigma^{1-\frac{s}{2}} \nabla D_t \rw\|_{L^2_t(L^2)}^2
+\e\|\sigma^{\frac{3-s}{2}} \sqrt{\r} D_t^2\rw\|_{L^2_t(L^2)}^2\Bigr)\\
&+C_\e\Bigl(\|\sigma^{ \frac{1-s}{2}}  \frak{a}\|_{L^2_t(L^2)}^2
+     \|\sigma^{1-\frac{s}{2}}\nabla \frak{a}\|_{L^2_t(L^2)}^2
+\|\sigma^{\frac{3-s}{2}}{\rm F}\|_{L^2_t(L^2)}^2\Bigr).
\end{split}\eeno
Inserting the above inequalities into  \eqref{energy:Dv1} yields
\beq
\label{S8eq3}
\begin{split}
&\|\sigma^{1-\frac{s}2}\sqrt{\rho}D_t \rw(t)\|_{L^2}^2
+ \|\sigma^{1-\frac{s}2} \nabla D_t \rw\|_{L^2_t(L^2)}^2\leq C\|\sigma^{\frac {1-s}2}
\sqrt{\rho}D_t \rw\|_{L^2_t(L^2)}^2\\
&\ +\e\|\sigma^{\frac{3-s}{2}} \sqrt{\r} D_t^2\rw\|_{L^2_t(L^2)}^2+C_\e\Bigl(\|\sigma^{ \frac{1-s}{2}}  \frak{a}\|_{L^2_t(L^2)}^2
+     \|\sigma^{1-\frac{s}{2}}\nabla \frak{a}\|_{L^2_t(L^2)}^2
+\|\sigma^{\frac{3-s}{2}}{\rm F}\|_{L^2_t(L^2)}^2\Bigr).
\end{split}\eeq

On the other hand,  due to $\dive v=0,$ one has
\begin{align*}
\int_{\R^2}-\Delta D_t \rw |  D_t^2 \rw\,\dx
&=\frac12\frac{d}{dt}\int_{\R^2}|\nabla D_t \rw|^2\,\dx
+\int_{\R^2}\nabla D_t \rw:(\nabla v_\al\p_\al D_t \rw)\,\dx,
\end{align*}
and
\begin{align*}
|\int_{\R^2} \nabla D_t {\rw}:(\nabla v_\al\p_\al D_t\rw)\,\dx|
&\leq \|\nabla v\|_{L^2}\|\nabla D_t \rw\|_{L^4}^2\\
&\leq C\|\nabla v\|_{L^2}\|\nabla D_t\rw\|_{L^2}\|\nabla^2 D_t \rw\|_{L^2},
\end{align*}
so that we get, by taking the $L^2(\R^2)$ inner product of \eqref{S8eq1} with $D_t^2 \rw,$ that
\beno
\begin{split}
\|\sqrt{\rho}D^2_t \rw\|_{L^2}^2
+\frac 12\frac{d}{dt}\|\nabla D_t \rw\|_{L^2}^2=&\int_{\R^2} ({\rm F}-\nabla {\rm{q}}) | D^2_t \rw\,\dx+ C\|\nabla v\|_{L^2}\|\nabla D_t\rw\|_{L^2}\|\nabla^2 D_t \rw\|_{L^2}.
\end{split} \eeno
For any $\eta>0$, multiplying the above equality   by $\sigma(t)^{3-s}$ and then integrating the resulting equality  over $[0,t]$ leads to
\begin{equation}\label{ee:D2v1}
\begin{split}
\frac 12\bigl\|\sigma^{\frac{3-s}2}&\nabla D_t \rw(t)\bigr\|_{L^2}^2
+\bigl\| \sigma^{\frac{3-s}2} \sqrt\rho D^2_t \rw\bigr\|_{L^2_t(L^2)}^2
\leq
\frac {3-s}{2}
\bigl\|\sigma ^{1-\frac{s}2}
\nabla D_t \rw\bigr\|_{L^2_t(L^2)}^2\\
&+\frac 14 \| \sigma^{\frac{3-s}2}\sqrt{\rho}  D^2_t \rw\|_{L^2_t(L^2)}
 +C
 \| \sigma^{\frac{3-s}2} {\rm F}\|_{L^2_t(L^2)}^2
 +\eta\|\sigma^{\frac{3-s}{2}}\nabla^2 D_t \rw\|_{L^2}^2\\
&+C_\eta\int^t_0
\|\nabla v\|_{L^2}^2\bigl\| \sigma^{\frac{3-s}2} \nabla D_t \rw\bigr\|_{L^2}^2\,\dt'
-\int^t_0\int_{\R^2} \sigma ^{3-s}
\nabla {\rm{q}} |  D_t^2 \rw \,\dx\,\dt'.
\end{split}
\end{equation}
By using integration by parts and the condition:  $\dive D_t^2\rw=\dive \frak{b},$ we get that
\beq \label{D2v1:pil} \begin{split}
\bigl|\int^t_0\int_{\R^2} \sigma ^{3-s}
\nabla {\rm q} |  D_t^2 \rw \,\dx\,\dt' \bigr|=&\bigl|\int^t_0\int_{\R^2} \sigma ^{3-s}
\nabla {\rm q} |  \frak{b}\,\dx\,\dt' \bigr|\\
&\leq
\eta\bigl\|\sigma ^{ \frac{3-s}{2}} \nabla {\rm q} \bigr\|_{L^2_t(L^2)}^2
+C_\eta\| \sigma^{\frac{3-s}2} \frak{b}\|_{L^2_t(L^2)}^2.
\end{split}\eeq
To deal with the estimate of $\nabla^2 D_t \rw$ and $\na{\rm q},$ in view of \eqref{S8eq1}, we write
\begin{align*}
&-\Delta \bigl(D_t \rw
  -\nabla\Delta^{-1}\div D_t \rw\bigr)+
\nabla {{\rm q}}
=-\rho D_t^2 \rw+\nabla\div D_t\rw+{\rm F}.
\end{align*} Since $\dive\bigl(D_t \rw
  -\nabla\Delta^{-1}\div D_t \rw\bigr)=0,$ we deduce from
the classical estimate on Stokes' operator that
\begin{align*}
& \| (\nabla^2 D_t \rw,
\nabla{\rm q}) \|_{L^2 }
 \lesssim
\|\rho D_t^2 \rw\|_{L^2 }
+\|\nabla\div D_t \rw\|_{L^2 }
+\|{\rm F}\|_{L^2 },
\end{align*}
which implies
\begin{equation}\label{estimate:Dpi1l}
\begin{split}
 & \bigl\|\sigma ^{ \frac{3-s}{2}}
  ( \nabla^2 D_t \rw,  \nabla {{\rm q}} ) \bigr\|_{L^2_t(L^2)}
\leq
C \bigl\|\sigma ^{ \frac{3-s}{2}}(\sqrt{\r} D_t^2 \rw, \na\dive\fa, {\rm F}) \bigr\|_{L^2_t(L^2)}.
\end{split}
\end{equation}
Inserting \eqref{estimate:Dpi1l} into \eqref{D2v1:pil} and then substituting the resulting inequality into
 \eqref{ee:D2v1}, we get, by  taking $\eta$ so small that $C^2\eta\leq \f14,$
 that $$\longformule{
 \bigl\|\sigma ^{\frac{3-s}2}\nabla D_t \rw \bigr\|_{L^\infty_t(L^2)}^2
+\bigl\| \sigma^{\frac{3-s}2} \sqrt{\r}  D^2_t \rw\bigr\|_{L^2_t(L^2)}^2
\leq
C\bigl\| \sigma^{\frac{3-s}2}(\na\dive\fa, \frak{b}, {\rm F})\bigr\|_{L^2_t(L^2)}^2}{{}+3\bigl\|\sigma ^{1-\frac{s}2}
\nabla D_t \rw\bigr\|_{L^2_t(L^2)}^2
+2\int^t_0
\|\nabla v\|_{L^2}^2\| \sigma^{\frac{3-s}2} \nabla D_t \rw\|_{L^2}^2\,\dt'.} $$
Taking $\e\leq \f18,$ we deduce \eqref{S8eq2} by summing up the above inequality with $4\times$\eqref{S8eq3} and  then applying Gronwall's inequality  and
using
\eqref{estimate:Dpi1l}. This completes the proof of  Lemma \ref{S4lem01}.
\end{proof}

Let us now turn to the proof of Proposition \ref{S1prop2}.

\begin{proof}[Proof of Proposition \ref{S1prop2}]
Taking into account of  \eqref{S4eq7p},  we get, by applying Lemma \ref{S4lem01} to  \eqref{eq:Dv}, that
\begin{equation}\label{S3ineq:D2v}
\begin{split}
 \bigl\|\sigma ^{\frac{3-s_0}2}&\nabla D_t v \bigr\|_{L^\infty_t(L^2)}^2
+\bigl\| \sigma^{\frac{3-s_0}2} (\sqrt{\r}  D^2_t v, \na^2D_tv, \na D_t\pi)\bigr\|_{L^2_t(L^2)}^2\\
\leq  &
 \cA_0\Bigl(\bigl\|\sigma ^{\frac{3-s_0}2}(D_tv,v\cdot\na v)\bigr\|_{L^2_t(L^2)}^2+\|\sigma ^{1-\frac{s_0}2}\na(v\cdot\na v)\|_{L^2_t(L^2)}^2\\
&\qquad\qquad\qquad\qquad+
\bigl\| \sigma^{\frac{3-s_0}2}( \nabla v\otimes\nabla^2 v,
\frak{b}_0, F_D(v,\pi))\bigr\|_{L^2_t(L^2)}^2\Bigr).
\end{split}
\end{equation}
It follows from Lemma \ref{lem:Dv} that
\beno
\bigl\|\sigma ^{\frac{3-s_0}2}(D_tv,v\cdot\na v)\bigr\|_{L^2_t(L^2)}^2+\|\sigma ^{1-\frac{s_0}2}\na(v\cdot\na v)\|_{L^2_t(L^2)}^2\leq\cC_0(v_0,s_0). \eeno
In view of \eqref{eq:Dv}, \eqref{S4eq7p}, we deduce from   Lemma \ref{lem:Dv} that
\begin{equation*}\label{est:FDa}
\begin{split}
\|\sigma ^{ \frac{3-s_0}{2}}F_D(v,\pi)\|_{L^2_t(L^2)}
&\leq  \|\sigma ^{ \frac{3-s_0}{2}}
(\nabla v\otimes\nabla^2 v, \nabla v\otimes\nabla\pi)\|_{L^2_t(L^2)}
\leq \cC(v_0,s_0),
\end{split}
\end{equation*}
and
\begin{equation} \label{bound:tildeD2v}
\begin{split}
\|\sigma ^{ \frac{3-s_0}{2}}\frak{b}_0\|_{L^2_t(L^2)}
&\leq
\|\sigma ^{ \frac{3-s_0}{2}}
\bigl(v\otimes(\nabla v_t, D_t\nabla v),D_t v\otimes\nabla v\bigr)\|_{L^2_t(L^2)}
\leq \cC(v_0,s_0).
\end{split}
\end{equation}
Substituting the above inequalities into \eqref{S3ineq:D2v} gives rise to
\begin{equation}\label{ee:Dv''}
\begin{split}
\bigl\|\sigma^{\frac{3-s_0}2}\nabla D_t  v\bigr\|_{L^\infty_t(L^2)}
+\bigl\|\sigma^{\frac{3-s_0}2}
&\bigl(D_t^2 v, \nabla^2 D_t v, \nabla D_t \pi\bigr) \bigr\|_{ L^2_t(L^2)}
\leq
\cC(v_0,s_0),
\end{split}
\end{equation}
from which and  Lemma \ref{lem:Dv} again, we infer
\begin{align*}
 \|\sigma^{ \frac{3-s_0}2}
 (  D_t \nabla v, \nabla v_t ) \|_{L^\infty_t(L^2)}
\leq &
 \| \sigma^{ \frac{3-s_0}2}\nabla D_t v \|_{L^\infty_t(L^2)}\\
&+ \bigl\| \sigma^{ \frac{3-s_0}2}
 (\nabla v\otimes \nabla v, v\otimes\nabla^2 v  )\bigr\|_{L^\infty_t(L^2)}
\leq
\cC(v_0,s_0).
\end{align*}
Along with \eqref{ee:Dv''}, we  complete the proof of Proposition \ref{S1prop2}.
\end{proof}

We conclude this section with the following corollary concerning
  the estimates of such terms as $\|\sigma^{1-\frac{s_0}{2}}\nabla v\|_{L^\infty_t(L^\infty)}$,
  $\|\frak{B}_0\|_{L^1_t}$,
  which will
  play an important role in the following context.

\begin{col}\label{S2col3}
{\sl Let $(\r,v, \na\pi)$ be a smooth enough solution of \eqref{inNS}. Let $r\in [2,\infty[$ and
 $\frak{B}_0(t)$  be determined by \eqref{S1eq25}. Then
 one has
\beq\label{S2col3eq1}\begin{split} \bigl\|\sigma^{\left(\frac 32-\frac 1r-\frac{s_0}2\right)}
& \bigl(D_t v,  v_t, v\otimes\nabla v, \nabla^2 v, \nabla \pi\bigr)
\bigr\|_{L^\infty_t(L^r)}^2\\
&+\|\sigma^{1-\frac{s_0}2}\nabla v \|_{L^\infty_t(L^\infty)}^2+\int_0^t\frak{B}_0(t')\,\dt'\leq \cC^2(v_0,s_0). \end{split} \eeq}
\end{col}

\begin{proof}

It follows from  \eqref{S2eq6} that for any $ r\in [2,\infty[,$
\begin{align*}
 &\| \sigma^{\left(\frac 32-\frac 1r-\frac{s_0}2\right)} D_t v  \|_{L^\infty_t(L^r)}
 \leq
C\|\s^{1-\f{s_0}2}D_tv\|_{L^\infty_t(L^2)}^{\f2r}\|\s^{\f{3-s_0}2}\na D_tv\|_{L^\infty_t(L^2)}^{1-\f2r},\\
&\|\sigma^{\left(\frac 32-\frac 1r-\frac{s_0}{2}\right)}
 \nabla D_t v  \|_{L^2_t(L^r)}
\leq
C\|\sigma^{1-\frac{s_0}{2}}
\nabla  D_t v \|_{L^2_t(L^2)}^{\frac 2r}
 \|\sigma^{ \frac{3-s_0}{2}}
\nabla^2 D_t v \|_{L^2_t(L^2)}^{1-\frac 2r},
\end{align*} which together with \eqref{S4eq4p}, \eqref{ee:Dv''} and
 classical estimates on Stokes operator
$-\Delta v+\nabla \pi=-\rho D_t v,$ ensures that for any $ r\in [2,\infty[,$
\begin{align}\label{S4eq8pa}
\bigl\| \sigma ^{\left(\frac 32-\frac 1r-\frac{s_0}2\right)}
  (D_tv, \nabla^2  v, \nabla \pi)\bigr\|_{L^\infty_t(L^r)}+\|\sigma^{\left(\frac 32-\frac 1r-\frac{s_0}{2}\right)}
 \nabla D_t v \|_{L^2_t(L^r)}
\leq
\cC(v_0,s_0).
\end{align}
Hence thanks to
\beno
\begin{split}
 &\|\sigma^{1-\frac{s_0}2}\nabla v\|_{L^\infty_t(L^\infty)}
\leq
C \|\sigma ^{ \frac {1-s_0}2}\nabla v\|_{L^\infty_t(L^2)}^{\frac 13}
 \|\sigma ^{ \frac 54-\frac{s_0}2 }\nabla^2v\|_{L^\infty_t(L^4)}^{\frac 23},
 \\
&\|\sigma^{1-\frac{s_0}{2}}
  D_t v  \|_{L^2_t(L^\infty)}
\leq
C\|\sigma^{ \frac{1-s_0}{2}}
  D_t v \|_{L^2_t(L^2)}^{\frac 12}
 \|\sigma^{ \frac{3-s_0}{2}}
\nabla^2 D_t v \|_{L^2_t(L^2)}^{\frac 12},
\end{split}
\eeno
 and Proposition \ref{S1prop1}, \eqref{S4eq4p}, \eqref{ee:Dv''}, \eqref{S4eq8pa},  we deduce that
\beq\label{S4eq8p}
\|\sigma^{1-\frac{s_0}2}\nabla v\|_{L^\infty_t(L^\infty)}+\|\sigma^{1-\frac{s_0}{2}}
  D_t v  \|_{L^2_t(L^\infty)}
\leq
\cC(v_0,s_0).
\eeq
While it follows  from  Corollary \ref{S1col1} and \eqref{bound:v}  that for any $ r\in [2,\infty[,$ \begin{align*}
 \| \sigma ^{\left(\frac 32-\frac 1r-\frac {s_0}2\right)}( v\otimes\nabla v) \|_{L^\infty_t(L^r)}
&\leq
C\| \sigma ^{\frac 12} v\|_{L^\infty_t(L^\infty)}
\| \sigma ^{ 1-\frac 1r-\frac{s_0}2} \nabla v \|_{L^\infty_t(L^r)}
\leq \cC(v_0,s_0),
\end{align*}
from which and $ v_t=D_t v-v\cdot\nabla v,$ we infer from \eqref{S4eq8pa} that for $r\in [2,+\infty[$
\begin{align*}
 \bigl\| \sigma ^{\left(\frac 32-\frac 1r-\frac{s_0}2\right)}
 (v\otimes\na v, v_t)\bigr\|_{L^\infty_t(L^r)}
\leq
\cC(v_0,s_0).
\end{align*}
Finally it is easy to observe that
\begin{align*}
\|\sigma^{1-\frac{s_0}{2}} (v\otimes\nabla v) \|_{L^2_t(L^\infty)}
&\leq
C\|\sigma^{\frac 12} v \|_{L^\infty_t(L^\infty)}
 \|\sigma^{ \frac{1-s_0}{2}}
\nabla  v \|_{L^2_t(L^\infty)}.
\end{align*}
and
\begin{align*}
\bigl\|\sigma ^{ \frac {3-s_0}2}
 (D_t \nabla^2 v, \nabla D_t\nabla v,  D_t\nabla \pi)& \bigr\|_{ L^2_t(L^2)}
\leq
\bigl\|\sigma ^{\frac {3-s_0}2}  (\nabla^2 D_t v, \nabla D_t \pi) \bigr\|_{ L^2_t(L^2)}\\
&+
C\bigl\|\sigma^{\frac {3-s_0}2}
  (\nabla v\otimes\nabla^2 v,
   \nabla v\otimes\nabla \pi) \bigr\|_{ L^2_t(L^2)},
\end{align*}
and \begin{align*}
\|\sigma^{ \frac 32-\frac 1r-\frac{s_0}{2} }
(\nabla v\otimes\nabla v, v\otimes\nabla^2 v) \|_{L^2_t(L^r)}
\leq &
\|\sigma^{1-\frac 1r} \nabla v \|_{L^\infty_t(L^r)}
 \|\sigma^{ \frac{1-s_0}{2}}
\nabla  v \|_{L^2_t(L^\infty)}\\
&
+\|\sigma^{\frac 12}v\|_{L^\infty_t(L^\infty)}
\|\sigma^{ 1-\frac 1r-\frac{s_0}2 }
\nabla^2 v\|_{L^2_t(L^r)}.
\end{align*}
As a consequence, we deduce from \eqref{S4eq4p},  \eqref{bound:v}, \eqref{ee:Dv''} and Corollary \ref{S1col1}  that for $r\in [2,+\infty[$
$$\longformule{
\|\sigma^{1-\frac{s_0}{2}} (v\otimes\nabla v) \|_{L^2_t(L^\infty)}
+\bigl\|\sigma ^{ \frac {3-s_0}2}
 (D_t \nabla^2 v, \nabla D_t\nabla v,  D_t\nabla \pi) \bigr\|_{ L^2_t(L^2)}
}{{}+\|\sigma^{\left(\frac 32-\frac 1r-\frac{s_0}{2}\right)}
(\nabla v\otimes\nabla v, v\otimes\nabla^2 v) \|_{L^2_t(L^r)}\leq
\cC(v_0,s_0).} $$
This together with
$$D_t\nabla v=\nabla D_t v-\nabla v\cdot\nabla v,
\quad v_t=D_t v-v\cdot\nabla v
$$
 and \eqref{S4eq8pa}, \eqref{S4eq8p} ensures that for $r\in [2,+\infty[$
\beno
\|\s^{1-\f{s_0}2}v_t\|_{L^2_t(L^\infty)}
+\|\sigma^{\left(\frac 32-\frac 1r-\frac{s_0}{2}\right)}
(\nabla D_t v, D_t\nabla v, \nabla v_t) \|_{L^2_t(L^r)}
\leq \cC(v_0,s_0). \eeno
This together with Propositions \ref{S1prop1} and \ref{S1prop2}, Corollary \ref{S1col1}, Lemma \ref{lem:Dv} and \eqref{bound:tildeD2v} ensures that
\beno \int_0^t\frak{B}_0(t')\,\dt'\leq \cC^2(v_0,s_0).
\eeno
This completes the proof of   \eqref{S2col3eq1}.
\end{proof}

\setcounter{equation}{0}
\section{Propagation of the time-weighted $H^1$ regularity of $\v^1$}\label{sec:J1}

The goal of this section is to study the propagation of time-weighted $H^1$ regularity of $\v^1.$ To this end,  in view of \eqref{S1eq3}, we split $(\v^1,
\na{\rm \pi}^1)$ as \beq \label{S3eq3} \v^1=\v_{11}+\v_{12} \andf
\na{\rm \pi}^1=\na \rp_{11}+\na \rp_{12}, \eeq with $(\v_{11}, \na \rp_{11})$
and  $(\v_{12}, \na \rp_{12})$ solving the following systems
respectively
\begin{equation}\label{S3eq4}
 \left\{\begin{array}{l}
\displaystyle \rho\d_t \v_{11}+\r v\cdot\nabla \v_{11} -\Delta \v_{11}
+\nabla \rp_{11} =0,\\
\displaystyle \div \v_{11}=0,\\
\displaystyle \v_{11}|_{t=0}=\p_{X_0}v_0,
\end{array}\right.
\end{equation}
and
\begin{equation}\label{S3eq5}
 \left\{\begin{array}{l}
\displaystyle \rho\d_t \v_{12}+\r v\cdot\nabla \v_{12} -\Delta \v_{12}
+\nabla \rp_{12} =F_1(v,\pi),\\
\displaystyle \div  \v_{12}=\div(v\cdot\nabla X),\\
\displaystyle \v_{12}|_{t=0}=0,
\end{array}\right.
\end{equation} where the source term $F_1(v,\pi)$ is given by \eqref{S3eq1}.

Then we deduce from Proposition \ref{S2prop1} that \beq
\|\v_{11}\|_{\wt{L}^\infty_t(\cB^{s_1})}
+\|\na \v_{11}\|_{\wt{L}^2_t(\cB^{s_1})}\leq
C\|\p_{X_0}v_0\|_{\cB^{s_1}}\exp\bigl(C\|v_0\|_{L^2}^4\bigr),
\label{S3eq6} \eeq and \beq
\begin{split} \|\s^{\f{1-s_1}2}
\v_{11}\|_{\wt{L}^\infty_t(\cB^1)}^2&+\bigl\|\s^{\f{1-s_1}2}(\p_t\v_{11},\na^2
\v_{11},\na \rp_{11})\bigr\|_{L^2_t(L^2)}^2\leq \cC(v_0,\p_{X_0}v_0,s_1).
\end{split} \label{S3eq7} \eeq

\subsection{Time-weighted $\dot H^1$  estimate for $\v_{12}$ }

Let us first present the proof of Lemma \ref{S3lem2}.

\begin{proof}[Proof of Lemma \ref{S3lem2}] We first get, by multiplying \eqref{S2eq8} by $\s^{1-{s}},$ that  \beno
\begin{split}
\f{d}{dt}\|\s^{\f{1-{s}}2}\na
u(t)\|_{L^2}^2+&\bigl\|\s^{\f{1-{s}}2}(u_t, \na^2 u, \na\Pi)\bigr\|_{L^2}^2
\leq (1-{s})\s^{-{s}}\|\na
u\|_{L^2}^2\\
&\qquad+C\s^{1-{s}}\bigl(\|v\|_{L^4}^4\|\na
u\|_{L^2}^2+\|(g_t,\na\dive g,f)\|_{L^2}^2\bigr),
\end{split}
\eeno so that applying Gronwall's inequality gives rise to
\beq\label{lemeq2}
\begin{split}
\|\s^{\f{1-{s}}2}\na
u\|_{L^\infty_t(L^2)}^2+&\bigl\|\s^{\f{1-{s}}2}(u_t, \na^2u, \na\Pi)\bigr\|_{L^2_t(L^2)}^2\leq
C\exp\bigl(C\|v\|_{L^4_t(L^4)}^4\bigr)\\
&\qquad\times\int_0^t\bigl(\s^{-{s}}\|\na
u\|_{L^2}^2+\s^{1-{s}}\|(g_t,\na\dive g,f)\|_{L^2}^2\bigr)\,\dt'.
\end{split}
\eeq To handle the term $\int_0^t\s^{-{s}}\|\na
u\|_{L^2}^2\,dt',$ let us denote  \beq\label{S3eq28} \y\eqdefa
u-g, \quad \y=\y_1+\y_2\andf
\na\Pi\eqdefa\na\Pi_{1}+\na\Pi_{2}, \eeq with
$(\y_1,\na\Pi_{1})$ and $(\y_2,\na\Pi_{2})$ solving respectively
\begin{equation}\label{S3eq30y}
 \left\{\begin{array}{l}
\displaystyle \rho\d_t \y_{1}+\r v\cdot\nabla \y_{1} -\Delta \y_{1}
+\nabla \Pi_{1} =0,\\
\displaystyle \div \y_1=0,\\
\displaystyle \y_1|_{t=0}=-g(0),
\end{array}\right.
\end{equation}
and
\begin{equation}\label{S3eq29}
 \left\{\begin{array}{l}
\displaystyle \rho\d_t \y_2+\r v\cdot\nabla \y_2 -\Delta \y_2
+\nabla \Pi_{2}= f-\left(\r\p_tg+\r
v\cdot\na g-\D g\right),\\
\displaystyle \div \y_2=0,\\
\displaystyle \y_2|_{t=0}=0.
\end{array}\right.
\end{equation}
For any $\de\in ]0,1[,$ it follows from
Proposition \ref{S2prop1}
that
 \beno
\begin{split}
\bigl\|\s^{\f{1-\de}2}\na\y_1\|_{L^\infty_t(L^2)}
\leq \cA_0\|g(0)\|_{\cB^{\de}},
\end{split}
\eeno
which in particular implies for $s\in ]0,\de[$ that
\beq\label{S3eq32}
\bigl\|\s^{-\f{s}2}\na\y_1\|_{L^2_t(L^2)}^2\leq \bigl\|\s^{\f{1-\de}2}\na\y_1\|_{L^\infty_t(L^2)}^2\|\s^{-1+\de-s}\|_{L^1_t}\leq \f{\cA_0\w{t}}{\de-s}\|g(0)\|_{\cB^{\de}}^2.
\eeq
Whereas for any fixed $t_0\in ]0,1[,$
we denote $\s_{t_0}(t)\eqdefa \left\{\begin{array}{l}
\s(t)\ \ \mbox{if}\ \ t\geq t_0,\\
\displaystyle t_0\quad \ \ \mbox{if}\ \ t\leq t_0.
\end{array}\right. $ Then we get, by first taking the $L^2$ inner product of \eqref{S3eq29} with $\y_2,$  then  multiplying the  equality
 by
$\s_{t_0}^{-{s}}(t)$ and finally integrating the resulting inequality over $[0,t],$  that
 \beno
\begin{split}
\|&\s_{t_0}^{-\f{s}2}\sqrt{\r}
\y_2(t)\|_{L^2}^2+\int_0^t\s_{t_0}^{-{s}}\|\na \y_2\|_{L^2}^2\,\dt'\\
=&-{s}\int_0^t\s_{t_0}^{-1-{s}}\s_{t_0}'\|\sqrt{\r} \y_2\|_{L^2}^2\,\dt'+
\int_0^t\int_{\R^2}\s_{t_0}^{-{s}}\left(f-\r\p_tg-\r
v\cdot\na g+\D g\right) | \y_2\,\dx\,\dt'\\
\leq&{C}\bigl\|\s_{t_0}^{-\f{s}2}(f,g_t)\bigr\|_{L^1_t(L^2)}^2+C\|\s^{\f{1-s_0}2}v\|_{L^\infty_t(L^\infty)}^2 \bigl\|\s^{-\f{1-(s_0-s)}2}
\na g\bigr\|_{L^1_t(L^2)}^2
\\
&\qquad\qquad\quad\qquad\qquad\ +\f12\|\s_{t_0}^{-\f{s}2}\sqrt{\r}
\y_2\|_{L^\infty_t(L^2)}^2-\int_0^t\int_{\R^2}\s_{t_0}^{-{s}}\na g | \na\y_2\,\dx\,\dt'.
\end{split}
\eeno
Noting that
\beno
\bigl|\int_0^t\int_{\R^2}\s_{t_0}^{-{s}}\na g | \na\y_2\,\dx\,dt'\bigr|\leq
\f12\bigl\|\s_{t_0}^{-\f{s}2}\na g\bigr\|_{L^2_t(L^2)}^2
+\f12\|\s_{t_0}^{-\f{s}2}\na\y_2
\|_{L^2_t(L^2)}^2.
\eeno
This together with Corollary \ref{S1col1} ensures that
\beno
\begin{split}
\int_0^t\s_{t_0}^{-{s}}\|\na \y_2\|_{L^2}^2\,\dt'
\leq
\cA_0\Bigl(
\bigl\|\s^{-\f{s}2}(f,g_t)\bigr\|_{L^1_t(L^2)}^2
+\bigl\|\s^{-\f{1-(s_0-s)}2}
\na g\bigr\|_{L^1_t(L^2)}^2
+\|\s^{-\f{s}2}\na g\|_{L^2_t(L^2)}^2\Bigr),
\end{split}
\eeno
from which, \eqref{S3eq28} and \eqref{S3eq32}, we infer  \beno
 \begin{split}
 \int_0^t\s_{t_0}^{-{s}}\|\na u\|_{L^2}^2\dt'
\leq
\cA_0\Bigl(&\f{\w{t}}{\de-s}\|g(0)\|_{\cB^{\de}}^2+\bigl\|\s^{-\f{s}2}(f,g_t)\bigr\|_{L^1_t(L^2)}^2\\
&+\bigl\|\s^{-\f{1-(s_0-s)}2}
\na g\bigr\|_{L^1_t(L^2)}^2
+\|\s^{-\f{s}2}\na g\|_{L^2_t(L^2)}^2\Bigr).
\end{split}
\eeno Applying Fatou's Lemma, we deduce that as $t_0\to 0_+$ in
the above inequality, there holds
\beno
 \begin{split}
 \int_0^t\s^{-{s}}\|\na u\|_{L^2}^2\dt'
\leq
\cA_0\Bigl(&\f{\w{t}}{\de-s}\|g(0)\|_{\cB^{\de}}^2+\bigl\|\s^{-\f{s}2}(f,g_t)\bigr\|_{L^1_t(L^2)}^2\\
&+\bigl\|\s^{-\f{1-(s_0-s)}2}
\na g\bigr\|_{L^1_t(L^2)}^2
+\|\s^{-\f{s}2}\na g\|_{L^2_t(L^2)}^2\Bigr).
\end{split}
\eeno
 Inserting the above estimate into \eqref{lemeq2} leads to \eqref{lemeq1}.
\end{proof}

Applying Lemma \ref{S3lem2} to the system \eqref{S3eq5} yields the time-weighted $\dH^1$ estimate of $\v_{12}.$

\begin{prop}\label{S3prop1}
{\sl Let $(\v_{12},\na \rp_{12})$ be a smooth enough solution of
\eqref{S3eq5}.
Then     there holds
\beq
\label{S3eq21}
\begin{split}
\bigl\|\s^{\f{1-s_1}2} \na \v_{12}&\bigr\|_{L^\infty_t(L^2)}^2
+\bigl\|\s^{\f{1-s_1}2}(\p_t\v_{12},\na^2
\v_{12},\na p_{12})\bigr\|_{L^2_t(L^2)}^2\leq \frak{G}_{1,X}(t)\with\\
\frak{G}_{1,X}&(t)
\eqdefa \exp\bigl(\cA_0 \w{t}^2\bigr)
 \Bigl(\cA_1+\int_0^t\bigl(\frak{B}_0(t')+\s(t')^{-1+ \frac{\theta_0}{2} }\bigr)
\|\na X(t')\|_{W^{1,p}}^2\,dt'\Bigr),
\end{split}
\eeq
where   $\frak{B}_0(t)$ is given by \eqref{S1eq25}.}
\end{prop}
\begin{proof}
Noticing  that for $\theta_0=s_0-s_1,$
\beno
\begin{split}
&\|\s^{-\f{1-\th_0}2}\na g\|_{L^1_t(L^2)}^2\leq \|\s^{-\f{s_1}2}\na g\|_{L^2_t(L^2)}^2\|\s^{-1+s_0}\|_{L^1_t}\leq C\w{t}\|\s^{-\f{s_1}2}\na g\|_{L^2_t(L^2)}^2,
\andf\\
&\|\s^{-\f{s_1}2}(f,g_t)\|_{L^1_t(L^2)}^2
\leq \|\s^{\f{1-s_0}2}(f,g_t)\|_{L^2_t(L^2)}^2
\|\s^{-1+\th_0}\|_{L^1_t}
\leq C\w{t}\|\s^{\f{1-s_0}2}(f,g_t)\|_{L^2_t(L^2)}^2
\end{split}
\eeno
we  get, by applying Lemma \ref{S3lem2}  (with $s=s_1$ and $\delta=s_0$) to the system \eqref{S3eq5},  that
 \beq \label{S3eq22}
 \begin{split}
 \bigl\|&\s^{\f{1-s_1}2} \na \v_{12}\bigr\|_{L^\infty_t(L^2)}^2
+\bigl\|\s^{\f{1-s_1}2}(\p_t\v_{12},\na^2
\v_{12},\na \rp_{12})\bigr\|_{L^2_t(L^2)}^2
\leq \cA_0\w{t}\Bigl(\|v_0\cdot\na X_0\|_{\cB^{s_0}}^2\\
&+\bigl\|\s^{-\f{s_1}2}\na(v\cdot\na X)\|_{L^2_t(L^2)}^2+
\bigl\|\s^{\f{1-s_0}2}\bigl(\p_t(v\cdot\na X),\na\dive(v\cdot\na X), F_1(v,\pi)\bigr)\bigr\|_{L^2_t(L^2)}^2\Bigr).
\end{split}
\eeq
It follows from the law of product in Besov spaces (see \cite{BCD} for instance) that
\beno
\|v_0\cdot\na X_0\|_{\cB^{s_0}}
\lesssim \|v_0\|_{L^2\cap \cB^{s_0}}\|\na X_0\|_{\dB^{\f2{p}}_{p,1}}
\lesssim \|v_0\|_{L^2\cap \cB^{s_0}}\|\na X_0\|_{W^{1,p}}\leq \cA_1.
\eeno
It follows from Corollary \ref{S1col1} that
\beno
\begin{split}
\bigl\|\s^{-\f{s_1}2}\na(v\cdot\na X)\|_{L^2_t(L^2)}^2
\leq &\int_0^t\s^{-1+\th_0}\bigl(\|\s^{\f{1-s_0}2}v\|_{L^{\f{2p}{p-2}}}^2
+\|\s^{\f{1-s_0}2}\na v\|_{L^2}^2\bigr)\|\na X\|_{W^{1,p}}^2\,\dt'
\\
\leq &\cA_0\int_0^t\s(t')^{-1+\th_0}\|\na X(t')\|_{W^{1,p}}^2\,\dt'.
\end{split}
\eeno
While in view of \eqref{eq:X}, we have
\beq\label{S3eq35}
\p_t(v\cdot\na X)=v_t\cdot\na X-(v\cdot\na) v\cdot\na X-v_\al v\cdot\na\p_\al X+v\cdot\na\v^1, \eeq
 and thus it comes out
 \beq\label{S3eq26}
\begin{split} \bigl\|\s^{\f{1-s_0}2}&\p_t(v\cdot\na X)\bigr\|_{L^2_t(L^2)}^2\lesssim
\int_0^t\Bigl(\|\s^{\f{1-s_0}2}v_t\|_{L^2}^2
+\s^{-(1-s_0)}\bigl(\bigl\|\s^{\f{1-s_0}2\left(\f12+\f1p\right)}v\bigr\|_{L^\infty_t(L^{\f{4p}{p-2}})}^4
\\& +\bigl\|\s^{\f{1-s_0}2}v\bigr\|_{L^\infty_t(L^\infty)}^2
\|\s^{\f{1-s_0}2}\na
v\|_{L^\infty_t(L^2)}^2\bigr)
\Bigr)\|\na X\|_{W^{1,p}}^2\,\dt'
\\
&+\int_0^t\s^{-(1-s_1)}
\bigl\|\s^{\f{1-s_0}2}v\|_{L^\infty_t(L^\infty)}^2
\bigl(\bigl\|\s^{\f{1-s_1}2}\na
\v_{11}\|_{ L^2}^2+\bigl\|\s^{\f{1-s_1}2}\na
\v_{12}\bigr\|_{L^2}^2\bigr)\,\dt',
\end{split}
\eeq
from which, Corollary \ref{S1col1} and \eqref{S3eq7}, we infer that
\beno
\begin{split}
\bigl\|\s^{\f{1-s_0}2}\p_t(v\cdot\na X)\bigr\|_{L^2_t(L^2)}^2
\leq \cA_1\w{t}
+\cA_0 \int_0^t\bigl(&\bigl(\frak{B}_0(t')+\s(t')^{-1+s_0}\bigr)\|\na X(t')\|_{W^{1,p}}^2\\
&+\s(t')^{-(1-s_1)}\|\s^{\f{1-s_1}2}\na
\v_{12}(t')\|_{L^2}^2\bigr)\,\dt' .
\end{split}
\eeno
Notice that due to $\dive X=0,$ we write \beq \label{S3eq23}
\|\nabla\div(v\cdot\nabla X)\|_{L^2} = \|\nabla(\p_\al v\cdot\nabla
X^\al)\|_{L^2} \leq
 C\bigl(\|\nabla^2 v\|_{L^2}
 +
 \|\nabla v\|_{L^{\f{2p}{p-2}}}\bigr)\|\nabla X\|_{W^{1,p}},
\eeq
which together with Corollary \ref{S1col1} ensures that
\beno
\begin{split}
\|\s^{\f{1-s_0}2}\nabla\div(v\cdot\nabla X)\|_{L^2_t(L^2)}^2\leq &\int_0^t\bigl(\|\s^{\f{1-s_0}2}\na^2v\|_{L^2}^2+\|\s^{\f{1-s_0}2}\na v\|_{L^{\f{2p}{p-2}}}^2\bigr)\|\nabla X\|_{W^{1,p}}^2\,\dt'\\
\leq &\int_0^t\frak{B}_0(t')\|\nabla X(t')\|_{W^{1,p}}^2\,\dt'.
\end{split}
\eeno
Similarly, we deduce from \eqref{brho} and \eqref{S3eq1} that
 $$ \longformule{
\bigl\|\s^{\f{1-s_0}2}F_1(v,\pi)\|_{L^2_t(L^2)}^2
\lesssim
\|\s^{\f{1-s_0}2}D_t  v\|_{L^2_t(L^2)}^2 }{{}
+\int_0^t\left(\bigl\|\s^{\f{1-s_0}2}\na
v\bigr\|_{L^{\f{2p}{p-2}}}^2+\bigl\|\s^{\f{1-s_0}2}(\na^2 v,\na
\pi)\bigr\|_{L^{2}}^2\right)\|\na
X\|_{W^{1,p}}^2\,\dt',}
$$
which together with Lemma \ref{lem:Dv} and Corollary \ref{S2col3} ensures that
\beno
\bigl\|\s^{\f{1-s_0}2}F_1(v,\pi)\|_{L^2_t(L^2)}^2
\leq
\cA_0+\int_0^t\frak{B}_0(t')\|\na
X(t')\|_{W^{1,p}}^2\,\dt'.
\eeno
Inserting the above estimates into \eqref{S3eq22} gives rise to
\beno
\begin{split}
 \bigl\|&\s^{\f{1-s_1}2} \na \v_{12}\bigr\|_{L^\infty_t(L^2)}^2
+\bigl\|\s^{\f{1-s_1}2}(\p_t\v_{12},\na^2
\v_{12},\na \rp_{12})\bigr\|_{L^2_t(L^2)}^2\leq
\cA_1\w{t}^2 \\
&+\cA_0\w{t}\int_0^t\Bigl(\bigl(\frak{B}_0(t')+\s(t')^{-1+\th_0}\bigr)\|\na X(t')\|_{W^{1,p}}^2
+\s(t')^{-(1-s_1)}\|\s^{\f{1-s_1}2}\na
\v_{12}(t')\|_{L^2}^2\Bigr)\,\dt' .
\end{split}
\eeno
Applying Gronwall's inequality leads to \eqref{S3eq21}. This completes the proof of the proposition.
\end{proof}

Combining \eqref{S3eq7} with Proposition \ref{S3prop1}, we
obtain

\begin{col}\label{S3col1}
{\sl Let $(\r, v, \na \pi, X)$ be a smooth enough solution of the coupled system
\eqref{inNS} and \eqref{eq:X}. Then one has \beq
\begin{split} \|\s^{\f{1-s_1}2}\na
\v^1\|_{{L}^\infty_t(L^2)}^2+&\int_0^t\s(t')^{1-s_1}\bigl\|(\p_t\v^1,\na^2
\v^1,\na \rm{\pi}^1)\bigr\|_{L^2}^2\,dt'\leq \frak{G}_{1,X}(t),
\end{split} \label{S3eq40} \eeq with $\frak{G}_{1,X}(t)$ being given by \eqref{S3eq21}.}
\end{col}

\subsection{The proof of Proposition \ref{prop:B1} }

We first present the $L^2$ energy estimate of $\v_{12}.$

\begin{lem}\label{S3prop2}
{\sl Let $(\v_{12}, \na \rp_{12})$ be a smooth enough solution of
\eqref{S3eq5}. Then there holds
\beq \label{S3eq8}\begin{split}
\|\v_{12}\|_{L^\infty_t(L^2)}^2+&\|\na
\v_{12}\|_{L^2_t(L^2)}^2
\leq \frak{G}_{1,X}(t), \end{split} \eeq for $\frak{G}_{1,X}(t)$  given by \eqref{S3eq21}. }
\end{lem}

\begin{proof} We  get, by using standard energy estimate to \eqref{S3eq5},
that \beno \f12\|\sqrt{\r}\v_{12}\|_{L^\infty_t(L^2)}^2+\|\na
\v_{12}\|_{L^2_t(L^2)}^2
= \int_0^t\int_{\R^2}\bigl(-\na \rp_{12}+F_1(v,\pi)\bigr) | \v_{12}\,\dx\,\dt'.
\eeno
By using
$\dive \v_{12}=\dive(v\cdot\na X),$ one has \beno
\begin{split}
\int_{\R^2}\na \rp_{12} | \v_{12}\,\dx
%=&-\int_{\R^2} \rp_{12} | \dive
%\v_{12}\,dx\\
%=&-\int_{\R^2} \rp_{12} | \dive(v\cdot\na X)\,dx
=\int_{\R^2}\na \rp_{12}
| v\cdot\na X\,\dx,
\end{split}
\eeno from which and the momentum equation of \eqref{S3eq5}, we
infer \beq\label{S3eq9}
\begin{split}
\f12\|\sqrt{\r}\v_{12}\|_{L^\infty_t(L^2)}^2+\|\na
\v_{12}&\|_{L^2_t(L^2)}^2=\int_0^t\int_{\R^2} F_1(v,\pi) | (\v_{12}- v\cdot\na
X)\,\dx\,\dt'\\
&+\int_0^t\int_{\R^2}\left(\r\p_t\v_{12}+\r v\cdot\na \v_{12}-\D \v_{12}\right)
| v\cdot\na X \dx\,\dt'.
\end{split}
\eeq
It is
easy to observe from \eqref{brho} and \eqref{S3eq1} that  \beno
\begin{split}
|\int_0^t&\int_{\R^2} F_1(v,\pi) | (\v_{12}- v\cdot\na
X)\,\dx\,\dt'|
\leq  \f14\|\sqrt{\r}\v_{12}\|_{L^\infty_t(L^2)}^2
\\
&+C\Bigl(\int_0^t\s^{-\f{1-s_0}2}\times\s^{\f{1-s_0}2}\bigl(\|\na v\|_{L^{\f{2p}{p-2}}}+
\|(D_t v, \na^2v, \na\pi)\|_{L^2}\bigr)
(1+\|\na X\|_{W^{1,p}})\,\dt'\Bigr)^2
\\
&+C\int_0^t\s^{-\f{1- s_0}2}\times\s^{\f{1-s_0}2}\bigl(\|\na v\|_{L^{\f{2p}{p-2}}}+
\|(D_t v, \na^2v, \na\pi)\|_{L^2}\bigr)\|v\|_{L^2}
(1+\|\na X\|_{W^{1,p}}^2)\,\dt'.
\end{split}
\eeno Hence by virtue of Corollary \ref{S1col1} and \eqref{S1eq25}, we
deduce \beq \label{S3eq10} \begin{split}\bigl| \int_0^t\int_{\R^2}
F_1(v,\pi) | (\v_{12}- v\cdot\na X)\,\dx\,\dt'\bigr|
\leq & \f14\|\sqrt{\r}\v_{12}\|_{L^\infty_t(L^2)}^2+\frak{G}_{1,X}(t)
\end{split} \eeq for  $\frak{G}_{1,X}(t)$  given by \eqref{S3eq21}.

Let us now turn to the estimate of the second line of
\eqref{S3eq9}. Once again, we get, by applying H\"older's
inequality, that
$$\longformule{ \bigl|\int_0^t\int_{\R^2}\r(\p_t\v_{12}-\D \v_{12}) | v\cdot\na X\,\dx\,\dt'\bigr|\leq
C\int_0^t\|(\p_t\v_{12},\na^2 \v_{12})\|_{L^2}\|v\|_{L^2}\|\na
X\|_{L^\infty}\,\dt'}{{} \leq
C\bigl\|\s^{\f{1-s_1}2}(\p_t\v_{12},\na^2 \v_{12})\|_{L^2_t(L^2)}^2+C\int_0^t\s^{-\left(1-s_1\right)}
\|v\|_{L^\infty_t(L^2)}^2\|\na
X\|_{W^{1,p}}^2\,\dt', } $$
and
\beno
\begin{split}
\bigl|\int_0^t\int_{\R^2}&\r v\cdot\na \v_{12} | v\cdot\na X\,dx\,dt'\bigr|\leq
\int_0^t\|v\|_{L^\infty}\|\na \v_{12}\|_{L^2}\|v\|_{L^2}\|\na
X\|_{L^{\infty}}\,\dt'
\\
\leq &\f12\|\na
\v_{12}\|_{L^2_t(L^2)}^2+\int_0^t\s^{-(1-s_0)}\bigl\|\s^{\f{1-s_0}2}
v\|_{L^\infty_t(L^\infty)}^2\|v\|_{L^\infty_t(L^{2})}^2\|\na
X\|_{W^{1,p}}^2\,\dt'.
\end{split}
\eeno
 Together with Corollary \ref{S1col1}, the above inequalities
ensure that \beq
\begin{split}
\bigl|&\int_0^t\int_{\R^2}\left(\r\p_t\v_{12}+\r v\cdot\na \v_{12}-\D
\v_{12}\right) | v\cdot\na X dx\,dt'\bigr|\leq  \f12\|\na \v_{12}\|_{L^2_t(L^2)}^2 \\
&\qquad+
C\bigl\|\s^{\f{1-s_1}2}(\p_t\v_{12},\na^2 \v_{12})\|_{L^2_t(L^2)}^2
+\cA_0\int_0^t\s(t')^{-1+s_1}\|\na
X(t')\|_{W^{1,p}}^2\,\dt'.
\end{split} \label{S3eq12} \eeq
Inserting the Inequalities \eqref{S3eq10} and \eqref{S3eq12} into
\eqref{S3eq9}   leads to \beno
\begin{split}
\|\sqrt{\r}\v_{12}\|_{L^\infty_t(L^2)}^2+\|\na
\v_{12}\|_{L^2_t(L^2)}^2 \leq &
C\|\s^{\f{1-s_1}2}(\p_t\v_{12},\na^2 \v_{12})\|_{L^2_t(L^2)}^2+\frak{G}_{1,X}(t)
\end{split} \eeno for  $\frak{G}_{1,X}(t)$  given by \eqref{S3eq21}. This together with Proposition \ref{S3prop1} leads to \eqref{S3eq8}. We finish the proof of the lemma.
\end{proof}

Thanks to \eqref{S3eq3},
we  conclude the proof of Proposition \ref{prop:B1} by  combining the Estimates \eqref{S3eq6} and \eqref{S3eq7} with  Proposition \ref{S3prop1} and Lemma \ref{S3prop2}. Moreover, we have the following corollary, which will be used in Section \ref{sectdtv}:

\begin{col}\label{lem:Dv1}
{\sl Let $\frak{G}_{1,X}(t)$ be given by \eqref{S3eq21}. Then under the assumptions of  Proposition \ref{prop:B1},
we have
\begin{equation}\label{S4eq10p}
\begin{split}
&\bigl\| \sigma ^{ \frac{1-s_1}{2}}
 \bigl( D_t \v^1, \v^1\otimes\nabla v, v\otimes \nabla \v^1,
\Delta D_t X, D_t\Delta X\bigr)\bigr\|_{L^2_t(L^2)}^2
\\
&+
\| \sigma ^{ \frac{1-s_1}{2}}\nabla \v^1\|_{L^4_t(L^4)}^2
+\bigl\|\sigma ^{1-\frac{s_1}{2}}
(\nabla v\otimes\nabla \v^1, v\otimes\nabla^2 \v^1,
\v^1\otimes\nabla^2 v)\bigr\|_{L^2_t(L^2)}^2
\\
&
+\|\sigma^{1-\frac{s_1}{2}} (\v^1\otimes\nabla v, v\otimes\nabla \v^1)\|_{L^\infty_t(L^2)}^2+\|\sigma ^{  \frac{3-s_1}{2}}\v^1\otimes (D_t\nabla v, \nabla v_t)\|_{L^2_t(L^2)}^2
\\
&
+\bigl\|\sigma ^{  \frac{3-s_1}{2}}
\nabla v\otimes(\nabla^2 \v^1,\nabla{\rm \pi}^1,
D_t \v^1, D_t\Delta X)\bigr\|_{L^2_t(L^2)}^2
\\
&
+\bigl\|\sigma ^{  \frac{3-s_1}{2}}
( \nabla \v^1, D_t\nabla X)
\otimes(\nabla^2 v, \nabla\pi, D_t v, v\otimes\nabla v)\bigr\|_{L^2_t(L^2)}^2\leq \frak{G}_{1,X}(t).
\end{split}
\end{equation} }
\end{col}
%%%%%%% Proof of Lemma %%%%%%%
\begin{proof}
We first deduce from  \eqref{S2eq6} and Proposition \ref{prop:B1} that
\begin{align}\label{S4eq11p}
 \|\v^1\|_{L^4_t(L^4)}^2
 +\|\sigma ^{ \frac{1-s_1}{2}}  \nabla \v^1  \|_{L^4_t(L^4)}^2
 \leq  \frak{G}_{1,X}(t).
\end{align}
Then due to
\begin{align*}
 \bigl\|\sigma^{ \frac{1-s_1}{2}}
 (v\otimes \nabla \v^1, \v^1\otimes\nabla v ) \bigr\|_{L^2_t(L^2)}
\leq&
 \| v  \|_{L^4_t(L^4)}
 \|\sigma ^{ \frac{1-s_1}{2}}
 \nabla \v^1 \|_{L^4_t(L^4)}+
 \|\sigma ^{ \frac{1-s_1}{2}}
 \nabla v\|_{L^4_t(L^4)}
 \| \v^1  \|_{L^4_t(L^4)},
\end{align*} and
\begin{align*}
 \|\sigma ^{ \frac{1-s_1}{2}}
D_t \v^1 \|_{L^2_t(L^2)}
& \leq
 \|\sigma ^{ \frac{1-s_1}{2}}
\d_t \v^1  \|_{L^2_t(L^2)}
+\|\sigma ^{ \frac{1-s_1}{2}}
v\cdot\nabla \v^1\|_{L^2_t(L^2)},
\end{align*}
we deduce from Corollary \ref{S1col1} and Proposition \ref{prop:B1} that
\beno
 \bigl\|\sigma^{ \frac{1-s_1}{2}}
 (v\otimes \nabla \v^1, \v^1\otimes\nabla v ) \bigr\|_{L^2_t(L^2)}^2+\|\sigma ^{ \frac{1-s_1}{2}}
D_t \v^1 \|_{L^2_t(L^2)}^2\leq \frak{G}_{1,X}(t).
\eeno
Note that $D_t X=\v^1,$ we have
\beq\label{S4eq11py}
D_t \Delta X=\Delta D_t X
-\Delta v\cdot\nabla X
-2\nabla v_\al\cdot\p_\al\nabla X,
\eeq from which and Proposition \ref{prop:B1}, we infer
\beq\label{S4eq13p}
\begin{split}
\|&\sigma ^{\frac{1-s_1}2}
 (\Delta D_t X, D_t \Delta X) \|_{ L^2_t(L^2)}^2\lesssim
\| \sigma ^{\frac{1-s_1}{2}}
 \Delta \v^1\|_{ L^2_t(L^2)}^2\\
 &\quad+\int_0^t\bigl(\|\sigma ^{\frac{1-s_0}2}
  \Delta v\|_{L^{2}}^2
+\|\sigma^{\frac{1-s_0}2}\nabla v\|_{L^{\f{2p}{p-2}}}^2 \bigr)
\| \na X \|_{W^{1,p}}^2\,\dt'\leq \frak{G}_{1,X}(t).
\end{split} \eeq

On the other hand,
it is easy to observe from  Corollary \ref{S1col1}, Proposition \ref{prop:B1} and \eqref{S4eq11p} that
\begin{equation}\label{v1v''}
\begin{split}
&\bigl\|\sigma ^{1-\frac{s_1}{2}}
(\nabla v\otimes\nabla \v^1, v\otimes\nabla^2 \v^1,
\v^1\otimes\nabla^2 v)\bigr\|_{L^2_t(L^2)}
 \leq
 \|\sigma ^{\frac 12}
 \nabla v \|_{L^4_t(L^4)}
 \|\sigma ^{ \frac{1-s_1}{2} }
 \nabla \v^1 \|_{L^4_t(L^4)}
 \\
&+
 \|\sigma ^{\frac 12 }
 v  \|_{L^\infty_t(L^\infty)}
 \|\sigma ^{ \frac{1-s_1}{2} }
 \nabla^2 \v^1 \|_{L^2_t(L^2)}+
 \|  \v^1  \|_{L^4_t(L^4)}
 \|\sigma^{1-\frac{s_1}{2}}   \nabla^2 v  \|_{L^4_t(L^4)}
 \leq
 \frak{G}_{1,X}^{\frac 12}(t).
 \end{split}
 \end{equation}
Similarly we deduce from Corollaries \ref{S1col1} and \ref{S2col3}, Proposition \ref{prop:B1}  and \eqref{S4eq13p} that
\begin{align*}
 \bigl\| \sigma^{1-\frac{s_1}{2}}  (\v^1\otimes\nabla v, v\otimes\nabla \v^1)\bigr\|_{L^\infty_t(L^2)}
 \leq &
 \|  \v^1 \|_{L^\infty_t(L^2)}
 \|\sigma^{1-\frac{s_1}{2}}\nabla v\|_{L^\infty_t(L^\infty)}\\
 &+
 \| \sigma^{\frac 12} v \|_{L^\infty_t(L^\infty)}
 \|\sigma^{ \frac{1-s_1}{2}}\nabla \v^1\|_{L^\infty_t(L^2)}\leq
 \frak{G}_{1,X}^{\frac 12}(t),
\end{align*}
and
\begin{align*}
&\bigl\|\sigma ^{ \frac{3-s_1}{2}}
\nabla v\otimes(\nabla^2 \v^1, \nabla{\rm \pi}^1, D_t \v^1, D_t\Delta X)\bigr\|_{L^2_t(L^2)}
\\
&\quad \leq
 \|\sigma
 \nabla v \|_{L^\infty_t(L^\infty)}
 \|\sigma ^{\frac{1-s_1}{2} }
 (\nabla^2 \v^1, \nabla{\rm \pi}^1, D_t \v^1,  D_t\Delta X) \|_{L^2_t(L^2)}\leq
 \frak{G}_{1,X}^{\frac 12}(t),
 \\
&\bigl\|\sigma ^{ \frac{3-s_1}{2}}
\nabla \v^1
\otimes(\nabla^2 v, \nabla\pi, D_t v, v\otimes\nabla v)\bigr\|_{L^2_t(L^2)}
\\
&\quad \leq
 \|\sigma ^{\frac{1-s_1}2 }
\nabla  \v^1\|_{L^4_t(L^4)}
 \|\sigma
( \nabla^2 v, \nabla\pi, D_t v, v\otimes\nabla v) \|_{L^4_t(L^4)}\leq
\frak{G}_{1,X}^{\frac 12}(t).
\end{align*}
As a result, we infer from    Corollary \ref{S1col1}  and Corollary \ref{S2col3}  that
\beno
\begin{split}
\bigl\|&\sigma ^{  \frac{3-s_1}{2}}
  D_t\na X
\otimes(\nabla^2 v, \nabla\pi, D_t v, v\otimes\nabla v)\bigr\|_{L^2_t(L^2)}^2\\
&\leq  \bigl\|\sigma ^{  \frac{3-s_1}{2}}
 \nabla \v^1
\otimes(\nabla^2 v, \nabla\pi, D_t v, v\otimes\nabla v)\bigr\|_{L^2_t(L^2)}^2\\
&\quad+\int_0^t\|\s^{\f{1-s_0}2}\na v\|_{L^\infty}^2\bigl\|\s^{1-\f{s_0}2}(\nabla^2 v, \nabla\pi, D_t v, v\otimes\nabla v)\bigr\|_{L^\infty_t(L^2)}^2\|\na X\|_{W^{1,p}}^2\,dt'
\leq \frak{G}_{1,X}(t).
\end{split}
\eeno
Finally it follows from  Corollary \ref{S2col3} and Proposition \ref{prop:B1} that for any $r\in ]2,+\infty[$,
\begin{align*}
\|\sigma ^{  \frac{3-s_1}{2}}&\v^1\otimes D_t\nabla v\|_{L^2_t(L^2)}
\leq
\|\sigma^{\frac{1-s_1}{r}} \v^1 \|_{L^\infty_t(L^{\frac{2r}{r-2}})}
\|\sigma ^{ \left( \frac{3}{2}-\frac 1r-\frac{s_1}{2}\right)}  D_t\nabla v \|_{L^2_t(L^r)}
\\
&\qquad\leq
C\|  \v^1 \|_{L^\infty_t(L^2)}^{1-\frac 2r}
\|\sigma^{\frac{1-s_1}{2}}\nabla \v^1\|_{L^\infty_t(L^2)}^{\frac 2r}
\|\sigma ^{  \left(\frac{3}{2}-\frac 1r-\frac{s_0}{2}\right)}  D_t\nabla v \|_{L^2_t(L^r)}\leq
\frak{G}_{1,X}^{\frac 12}(t),
\end{align*}
which together with the above estimate, \eqref{v1v''} and
Corollary \ref{S1col1} ensures
\begin{align*}
\|\sigma ^{  \frac{3-s_1}{2}}\v^1\otimes  \nabla v_t\|_{L^2_t(L^2)}
\leq &
\|\sigma ^{  \frac{3-s_1}{2}}\v^1\otimes  D_t\nabla v\|_{L^2_t(L^2)}
+\|\sigma ^{  \frac{3-s_1}{2}}\v^1\otimes  (v\cdot\nabla^2 v)\|_{L^2_t(L^2)}
\\
\leq &
\frak{G}_{1,X}^{\frac 12}(t)
+\|\sigma^{\frac 12}v\|_{L^\infty_t(L^\infty)}
\|\sigma^{1-\frac{s_1}{2}}\v^1\otimes\nabla^2 v\|_{L^2_t(L^2)}
\leq
\frak{G}_{1,X}^{\frac 12}(t).
\end{align*}
This completes the proof of the corollary.
\end{proof}
%
%%%%%%%%%%%%%%%%%%%%%%%%%%%%%%%%%%%%%%%%%
%%%%%%%%%%%%%%%%%%%%%%%%%%%%%%%%%%%%%%%%

%%%%%%%%%%%%%%%%%%%%%%%%%%%%%%%%%%%%%%%%%
%%%%%%%%%%%%%%%%%%%%%%%%%%%%%%%%%%%%%%%%
\setcounter{equation}{0}
\section{Propagation of the first-order striated regularity (continued)}\label{sectdtv}
%%%%%%%%%%%%%%%%%%%%%%%%%%%%%%%%%

The aim of this section is to prove Proposition \ref{prop:Dv1}. In order to do it, we shall first  present the time-weighted
$H^1$-energy estimate for $D_t \v^1,$  and then we derive the global {\it a priori} estimate of $\|\na X\|_{L^\infty_t(W^{1,p})}. $
Finally we study the propagation of $\cB^{s_1}$ regularity of $\v^1.$

%%%%%%%%%%%%%%%%%%%%%%%
\subsection{$H^1$ energy estimates for $D_t \v^1$}
%%%%%%%%%%%%%
In view of  of Lemma \ref{S4lem01}, to close the time-weighted $H^1$ estimate of $D_t\v^1,$ it remains to deal with the estimates of $\dive D_t\v^1$ and $\dive D_t^2\v^1,$
which are the lemmas as follows:

\begin{lem}\label{S4lem1w}
 {\sl
Let $\fa_1\eqdefa 2\v^1\cdot\nabla v+  v_t\cdot\nabla X+v\p_\al X\cdot\na v_\al
%-(v\cdot\nabla v)\cdot\nabla X
-(v\cdot\nabla X)\cdot\nabla v.$
Then we have
$\div D_t \v^1=\div \frak{a}_1,$ and for $\frak{G}_{1,X}(t)$ given by \eqref{S3eq21}, there holds
\beq \label{S8eq4} \|\sigma^{ \frac{1-s_1}{2}}\fa_1\|_{L^2_t(L^2)}^2+ \|\sigma^{1-\frac{s_1}{2}}\nabla \frak{a}_1\|_{L^2_t(L^2)}^2+\|\s^{\f{3-s_1}2}\na\dive\frak{a}_1\|_{L^2_t(L^2)}^2
\leq \frak{G}_{1,X}(t). \eeq  }
\end{lem}

\begin{proof} Due to $\dive X=\dive v=0$ and $[D_t;\p_X]=0,$ we observe that
\beq\label{divDv1}\begin{split}
\div(D_t \v^1)
=\div(\d_X D_t v)
&=\div\left(D_t v\cdot\nabla X+X\div D_t v\right)
\\
&=\div\left(D_t v\cdot\nabla X+X\div(v\cdot\nabla v)\right)
\\
&=\div\left(D_t v\cdot\nabla X+\d_X(v\cdot\nabla v)-(v\cdot\nabla v)\cdot\nabla X\right) ,
\end{split}\eeq
and
\begin{align*}
\div(\d_X(v\cdot\nabla v))
&  =\div\left(\v^1\cdot\nabla v+v\,\cdot\nabla \v^1-(v\cdot\nabla X)\cdot\nabla v\right)
\\
&=\div\left(2\v^1\cdot\nabla v +v\,\div \v^1-(v\cdot\nabla X)\cdot\nabla v\right)
\\
&=\div\left(2\v^1\cdot\nabla v +v\,\p_\al X\cdot\na v_\al-(v\cdot\nabla X)\cdot\nabla v\right).
\end{align*} This shows that $\div D_t \v^1=\div \frak{a}_1.$
Moreover,
it is easy to observe  that
\beno
\begin{split}
&\|\sigma^{ \frac{1-s_1}{2}}\v^1\cdot\na v\|_{L^2_t(L^2)}^2\leq \|\sigma^{ \frac{1-s_0}{2}}\na v\|_{L^2_t(L^\infty)}^2\|\v^1\|_{L^\infty_t(L^2)}^2,\\
&\|\sigma^{ \frac{1-s_1}{2}}
\bigl(v_t\cdot\nabla X, v\otimes\nabla v\otimes\na X\bigr)\|_{L^2_t(L^2)}^2
\leq \int_0^t\|\s^{\f{1-s_0}2}(v_t, v\otimes\na v)\|_{L^2}^2 \|\nabla X\|_{L^\infty}^2\,\dt',
\end{split}
\eeno
 from which,  Proposition  \ref{prop:B1} and \eqref{S1eq25}, we infer
\beno\begin{split} \|\sigma^{ \frac{1-s_1}{2}}\frak{a}_1\|_{L^2_t(L^2)}^2
\leq \frak{G}_{1,X}(t).
\end{split}\eeno
While we deduce from \eqref{S1eq25} and Corollary \ref{lem:Dv1} that
\beno\begin{split}
 \|\sigma^{1-\frac{s_1}{2}}\nabla \frak{a}_1\|_{L^2_t(L^2)}^2
&\lesssim \| \sigma^{ 1-\frac{s_1}{2}} (\nabla \v^1\cdot\nabla v, \v^1\cdot\nabla^2 v)\|_{L^2_t(L^2)}^2
\\
&\quad +\int_0^t\|\sigma^{\f12+\f1p-\frac{s_0}{2}}
( v_t,   v\otimes\nabla v)\|_{L^{\f{2p}{p-2}}}^2
\|\nabla^2 X\|_{L^p}^2\,\dt'
\\
&\quad
+\int_0^t\|\sigma^{1-\frac{s_1}{2}}
(\nabla  v_t, \nabla(v\otimes\nabla v))\|_{L^2}^2
\|\nabla X\|_{L^\infty}^2\,\dt'
\leq
\frak{G}_{1,X}(t).
\end{split}\eeno

On the other hand, it follows from \eqref{divDv1} that
$$
\div D_t \v^1
=\div \d_X D_t v
=\p_\al X\cdot\na D_t v^\alpha+2 \mbox{tr}(\nabla \v^1\nabla v)
-2\mbox{tr}\bigl((\nabla X^\al\p_\al v)\nabla v\bigr),
$$
from which, Corollaries \ref{S1col1}, \ref{S2col3} and \ref{lem:Dv1}, we infer
\beno
\begin{split}
\|\s^{\f{3-s_1}2}\na\dive\frak{a}_1\|_{L^2_t(L^2)}^2\lesssim & \|\s^{\f{3-s_1}2}\na(\na v\otimes\na\v^1)\|_{L^2_t(L^2)}^2
\\
&+\int_0^t\Bigl(\|\s^{\f{3-s_1}2}\na^2 D_tv\|_{L^2}^2+\|\s^{\f{3-s_1}2}(\na D_tv, \na v\otimes\na v)\|_{L^{\f{2p}{p-2}}}^2\\
&\quad+\|\s^{1-\f{s_0}2}\na^2v\|_{L^\infty_t(L^2)}^2\|\s^{\f{1-s_0}2}\na v\|_{L^\infty}^2\Bigr)\|\na X\|_{W^{1,p}}^2\,dt'
\leq
\frak{G}_{1,X}(t).
\end{split}\eeno
This completes the proof of  \eqref{S8eq4}.
 \end{proof}

\begin{lem}\label{S4lem3}
{\sl   Let $\frak{b}_0(t)$ and
$\frak{G}_{1,X}(t)$ be given by \eqref{S4eq7p} and \eqref{S3eq21} respectively. Let $\frak{b}_1\eqdefa D_t^2 v\cdot\nabla X- \frak{b}_0\cdot\nabla X
+  \d_X \frak{b}_0.$
Then we have
\beq \label{divD2v}
% \div D^2_t v=\div \frak{b}_0\andf
\div D_t^2 \v^1
=\div\frak{b}_1,
\eeq
and for any $\e>0,$ there exists a positive constant $C_\e>0$ so that
\beq \label{S4eq3w}
\begin{split}
\|\sigma^{ \frac{3-s_1}{2}} \frak{b}_1\|_{L^2_t(L^2)}^2 \leq &\e\|\s^{\f{3-s_1}2}\sqrt{\r}D_t^2\v^1\|_{L^2_t(L^2)}^2+ C_\e\frak{G}_{1,X}(t).
%&\qquad\qquad\qquad+\int_0^t\|\sigma^{ \frac{3-s_1}{2}} \frak{b}_0\|_{L^2}^2\|\na X\|_{L^\infty}^2\,dt'.
\end{split}
\eeq}
\end{lem}

\begin{proof} It follows from \eqref{S4eq7p} and $[\p_X; D_t]=0$ that
\begin{align*}
\div D_t^2 \v^1
=\div (\d_X D_t^2 v)
&=\div(D_t^2 v\cdot\nabla X)
+\div\bigl(X\div (D_t^2 v)\bigr)
\\
&=\div(D_t^2 v\cdot\nabla X)
+\div\bigl(X\div \frak{b}_0\bigr).
\end{align*}
This together with $\dive X=0$ ensures \eqref{divD2v}.

Observing from \eqref{S4eq7p} that
\begin{align*}
\d_X \frak{b}_0
&=\v^1\cdot(\nabla v_t+D_t\nabla v)+D_t \v^1\cdot\nabla v+D_t v\cdot\nabla \v^1-(D_t v\cdot\nabla X)\cdot\nabla v
\\
&\quad
+v\cdot\bigl(\nabla\d_t \v^1+D_t\nabla \v^1
-\nabla X^\al\p_\al v_t-\nabla(X_t\cdot\nabla v)-D_t(\nabla X^\al\p_\al v)\bigr),
\end{align*}
we have
\beq\label{S4eq5wr}
\begin{split}
\|\sigma^{ \frac{3-s_1}{2}}& \d_X \frak{b}_0\|_{L^2_t(L^2)}^2\lesssim  \bigl\|\s^{\f{3-s_1}2}\v^1\otimes (\na v_t+D_t \na v)\bigr\|_{L^2_t(L^2)}^2\\
&+
\bigl\|\s^{\f{3-s_1}2}(D_t \v^1\otimes \na v, D_t v\otimes\na \v^1)\bigr\|_{L^2_t(L^2)}^2\\
&+\|\s^{\f{1-s_0}2}\na v\|_{L^\infty_t(L^2)}^2\int_0^t\|\s^{1-\f{s_0}2}D_tv(t')\|_{L^\infty}^2\|\na X(t')\|_{W^{1,p}}^2\,dt'\\
&+\|\s^{\f{1-s_0}2}v\|_{L^\infty_t(L^\infty)}^2\int_0^t\|\s^{1-\f{s_0}2}\na v_t(t')\|_{L^2}^2\|\na X(t')\|_{W^{1,p}}^2\,dt'\\
&+\|\s^{\f{1-s_0}2}v\|_{L^\infty_t(L^\infty)}^2\bigl\|\s\bigl(\na \p_t\v^1
+D_t \na \v^1, \na(X_t\cdot\na v), D_t(\na X\cdot\na v)\bigr)\bigr\|_{L^2_t(L^2)}^2.
\end{split}
\eeq
To handle the last line of \eqref{S4eq5wr},
we first deduce from \eqref{eq:Dv1} and \eqref{S8eq3} that
\beq
\label{S8eq5-1}
\begin{split}
\|\sigma^{1-\frac{s_1}2} \nabla& D_t \v^1\|_{L^2_t(L^2)}^2\leq C\|\sigma^{\frac {1-s_1}2}
\sqrt{\rho}D_t \v^1\|_{L^2_t(L^2)}^2+\e\|\sigma^{\frac{3-s_1}{2}} \sqrt{\r} D_t^2\v^1\|_{L^2_t(L^2)}^2\\
&\quad+C_\e\Bigl(\|\sigma^{ \frac{1-s_1}{2}}  \frak{a}_1\|_{L^2_t(L^2)}^2
+     \|\sigma^{1-\frac{s_1}{2}}\nabla \frak{a}_1\|_{L^2_t(L^2)}^2
+\|\sigma^{\frac{3-s_1}{2}}{F_{1D}}\|_{L^2_t(L^2)}^2\Bigr).
\end{split}\eeq
While in view of \eqref{brho} and \eqref{S4eq14p}, we have
\begin{align*}
&\|\sigma^{\frac{3-s_1}{2}}
 F_{1D}   \|_{L^2_t(L^2)}^2
\lesssim
\|\sigma^{\frac{3-s_1}{2}}
\nabla v\otimes(  \nabla^2 \v^1, \nabla{\rm \pi}^1, D_t\Delta X)\|_{L^2_t(L^2)}^2
\\
&\qquad
+\|\sigma^{\frac{3-s_1}{2}}
(\nabla \v^1, D_t\nabla X)\otimes (\nabla^2 v,\nabla\pi)\|_{L^2_t(L^2)}^2
+\|\sigma^{\frac{3-s_0}{2}}D_t^2 v\|_{L^2_t(L^2)}^2
\\
&\qquad
+\int^t_0
\bigl(\|\sigma^{\frac{3-s_0}{2}}D_t \nabla v\|_{L^{\frac{2p}{p-2}}}^2
+\|\sigma^{\frac{3-s_0}{2}}(D_t\nabla^2 v, D_t\nabla\pi)\|_{L^2}^2\bigr)
\|\nabla X\|_{W^{1,p}}^2 \,\dt',
\end{align*}
from which and Proposition \ref{S1prop2} and Corollary \ref{lem:Dv1},   we infer
\beq\label{Dv1:F}\begin{split}
\|\sigma^{\frac{3-s_1}{2}}
 F_{1D} \|_{L^2_t(L^2)}^2
\leq \frak{G}_{1,X}(t).
\end{split}\eeq
Inserting the Inequality \eqref{Dv1:F} into \eqref{S8eq5-1} and using Corollary \ref{lem:Dv1} and Lemma \ref{S4lem1w},
we achieve
\beno
\|\sigma^{1-\frac{s_1}2} \nabla D_t \v^1\|_{L^2_t(L^2)}^2\leq \e\|\s^{\f{3-s_1}2}\sqrt{\r}D_t^2\v^1\|_{L^2_t(L^2)}^2+C_\e\frak{G}_{1,X}(t),\eeno
which together with Corollary \ref{lem:Dv1} implies
\beno
\begin{split}
\bigl\|\s(\na \p_t\v^1
+D_t \na \v^1)\bigr\|_{L^2_t(L^2)}^2\leq &\|\s^{1-\f{s_1}2}\na D_t\v^1\|_{L^2_t(L^2)}^2+\bigl\|\s^{1-\f{s_1}2}(v\otimes \na^2\v^1,\na v\otimes\na \v^1)\bigr\|_{L^2_t(L^2)}^2\\
\leq &\e\|\s^{\f{3-s_1}2}\sqrt{\r}D_t^2\v^1\|_{L^2_t(L^2)}^2+C_\e\frak{G}_{1,X}(t).
\end{split}
\eeno
Wheres due to $D_tX=\v^1,$ we have
$$ \longformule{
\|\s\na (X_t\cdot\na v)\|_{L^2_t(L^2)}^2\leq\bigl\|\s^{1-\f{s_1}2}\bigl(\na v\otimes\na \v^1, \v^1\otimes\na^2v\bigr)\bigr\|_{L^2_t(L^2)}^2}
{{}+\int_0^t\Bigl(\|\s^{\f{1-s_0}p}v\|_{L^\infty_t(L^{\f{2p}{p-2}})}^2\|\s^{\f{1-s_0}2}\na v\|_{L^\infty}^2
+\bigl\|\s^{1-\f{s_0}2}\bigl(\na v\otimes \na v, v\otimes \na^2 v\bigr)\bigr\|_{L^2}^2\Bigr)\|\na X\|_{W^{1,p}}^2\,\dt',}$$
and
\beno
\begin{split}
\|\s D_t(\na X)\cdot\na v)\|_{L^2_t(L^2)}^2\lesssim & \bigl\|\s^{1-\f{s_1}2}\na \v^1\otimes\na v\|_{L^2_t(L^2)}^2\\
& +\int_0^t\bigl\|\s^{1-\f{s_0}2}\na v\otimes\na v\|_{L^2}^2
\|\na X\|_{W^{1,p}}^2\,\dt',\\
\|\s \na X\otimes D_t\na v\|_{L^2_t(L^2)}^2\lesssim &\int_0^t\|\s^{1-\f{s_0}2}D_t\na v\|_{L^2}^2\|\na X\|_{W^{1,p}}^2\,\dt',
\end{split}
\eeno
we thus deduce from Corollaries  \ref{S1col1} and \ref{lem:Dv1}   that
\beno
\bigl\|\s\bigl(
%\na \p_t\v^1+D_t \na \v^1,
\na(X_t\cdot\na v), D_t(\na X\cdot\na v)\bigr)\bigr\|_{L^2_t(L^2)}^2
\leq
%\e\|\s^{\f{3-s_1}2}\sqrt{\r}D_t^2\v^1\|_{L^2_t(L^2)}^2
\frak{G}_{1,X}(t).
\eeno
Inserting the above inequalities into \eqref{S4eq5wr} and using Corollaries \ref{S1col1} and \ref{lem:Dv1}, we find
\beq\label{S8eq5}
\|\sigma^{ \frac{3-s_1}{2}} \p_X\frak{b}_0\|_{L^2_t(L^2)}^2 \leq \e\|\s^{\f{3-s_1}2}\sqrt{\r}D_t^2\v^1\|_{L^2_t(L^2)}^2+ \frak{G}_{1,X}(t),
\eeq
which together with the fact that
\beno
\begin{split}
\|\sigma^{ \frac{3-s_1}{2}} \frak{b}_1\|_{L^2_t(L^2)}^2 \leq
& \|\sigma^{ \frac{3-s_1}{2}} \p_X\frak{b}_0\|_{L^2_t(L^2)}^2\\
&+\int_0^t\bigl(\|\sigma^{ \frac{3-s_1}{2}} D_t^2v\|_{L^2}^2+\|\sigma^{ \frac{3-s_1}{2}} \frak{b}_0\|_{L^2}^2\bigr)\|\na X\|_{L^\infty}^2\,\dt',
\end{split}
\eeno
ensures \eqref{S4eq3w}.
\end{proof}

With the above preparations, we can close the time-weighted $H^1$ energy estimate of $D_t\v^1.$

\begin{prop}\label{S4prop1q}
{\sl Under the assumptions of Proposition \ref{prop:Dv1}, we have
\beq \label{S4prop1eq1}
A_{11}(t)+A_{12}(t)+\|X\|_{L^\infty_t(W^{2,p})}\leq \cH_1(t),
\eeq
where $A_{11}(t), A_{12}(t)$ are given  by \eqref{Jell} for $\ell=1.$} \end{prop}

\begin{proof}
We get, by applying Lemma \ref{S4lem01} to \eqref{eq:Dv1}, that
\beno
\begin{split}
 \|&\sigma^{1-\frac{s_1}2}D_t \v^1 \|_{L^\infty_t(L^2)\cap L^2_t(\dH^1)}^2+\|\sigma^{\frac{3-s_1}2} \na D_t \v^1 \|_{L^\infty_t(L^2)}^2
+  \|\sigma^{\frac{3-s_1}2}(D_t^2 \v^1,\na^2 D_t\v^1, \na D_t{\rm \pi}^1)\|_{L^2_t(L^2)}^2
\\
&\quad\leq \cA_0\bigl(\|\s^{\f{1-s_1}2}(D_t\v^1,\fa_1)\|_{L^2_t(L^2)}^2+\|\s^{1-\f{s_1}2}\na \fa_1\|_{L^2_t(L^2)}^2
+\|\s^{\f{3-s_1}2}(\na\dive\fa_1,\frak{b}_1,F_{1D})\|_{L^2_t(L^2)}^2\bigr) .\end{split}
\eeno
Taking into account \eqref{S4eq10p}, by substituting the Estimates \eqref{S8eq4}, \eqref{S4eq3w} and \eqref{Dv1:F} into the above inequality and then choosing $\varepsilon$ sufficiently small, we achieve
\beq \label{S4eq1t}
\begin{split}
 \|\sigma^{1-\frac{s_1}2}D_t \v^1 \|_{L^\infty_t(L^2)\cap L^2_t(\dH^1)}^2&+\|\sigma^{\frac{3-s_1}2} \na D_t \v^1 \|_{L^\infty_t(L^2)}^2
\\
&+  \|\sigma^{\frac{3-s_1}2} (D_t^2 \v^1,\na^2D_t \v^1, \na D_t{\rm \pi}^1)\|_{L^2_t(L^2)}^2
 \leq \frak{G}_{1,X}(t). \end{split} \eeq
 Moreover, it follows from  the 2-D interpolation inequality \eqref{S2eq6} that for any   $ r\in [2,\infty[$
 \begin{equation}\label{bound:Dv1,q}
 \begin{split}
 \| \sigma^{1-\frac{s_1}{2}} D_t \v^1 \|_{L^{\frac{2r}{r-2}}_t(L^r)}^2
&\leq
C \| \sigma^{1-\frac{s_1}{2}} D_t \v^1 \|_{L^{\infty}_t(L^2)}^{2\cdot\frac 2r}
 \| \sigma^{1-\frac{s_1}{2}} \nabla D_t \v^1 \|_{L^2_t(L^2)}^{2(1-\frac 2r)}
\leq
\frak{G}_{1,X}(t).
 \end{split}
 \end{equation}
 Along the same line, we deduce from Proposition \ref{S1prop2} that
 \begin{equation}\label{bound:Dv,q}
 \begin{split}
 \| \sigma ^{1-\frac{s_0}{2}} D_t v \|_{L^{\frac{2r}{r-2}}_t(L^r)}
&\leq
\cC(v_0,s_0)
\quad\forall\ r\in [2,\infty[.
 \end{split}
 \end{equation}

Let us now turn to the estimate of $\|X(t)\|_{W^{2,p}}.$
We first get, by using $L^p$ energy estimate to \eqref{eq:X} and \eqref{S4eq11py},  that
\begin{equation}\label{S4eq15p}
\begin{split}
 \frac{d}{dt}\| X\|_{W^{2,p}}
&\leq
C\bigl(\|\nabla v\|_{L^\infty}\|X\|_{W^{2,p}}
+\|\Delta v\|_{L^p}\|\nabla X\|_{L^\infty}\bigr)
+\|\Delta \v^1\|_{L^p}.
  \end{split}
\end{equation}
To deal with the estimate of $\|\Delta \v^1\|_{L^p},$  we rewrite \eqref{S1eq3} as
$$
-\Delta\bigl(\v^1-\nabla\Delta^{-1}(\p_\al v\cdot\nabla X^\al)\bigr)
+\nabla{\rm \pi}^1
=-\rho D_t \v^1+\na(\p_\al v\cdot\nabla X^\alpha)+F_1(v,\pi),
$$
for $F_1(v,\pi)$ given by \eqref{S3eq1}. Then due to $\dive \v^1=\p_\al v\cdot\na X^\alpha,$   we deduce from
 classical estimate for the Stokes operator and \eqref{S3eq1}  that
 that for any $r\in [2,p]$
\beq\label{Stokes:v1}\begin{split}
\|(\nabla^2 \v^1,& \nabla{\rm \pi}^1)\|_{L^r}
\leq
C\bigl(\|\nabla(\nabla v\otimes\nabla X)\|_{L^r}
+\|\rho D_t \v^1\|_{L^r}
+\|F_1(v, \pi)\|_{L^r}\bigr)
\\
&\leq C\Bigl(
 \|(D_t \v^1, D_t v)\|_{L^r}
+ \bigl(\|\nabla v\|_{  L^{\frac{rp}{p-r}}}
  +\|(\nabla^2 v, \nabla\pi)\|_{L^r}\bigr)\|\na X\|_{W^{1,p} }\Bigr) .
\end{split}\eeq
Inserting \eqref{Stokes:v1} for $r=p$ into \eqref{S4eq15p}, and  integrating the resulting inequality over $[0,t],$ we obtain
\beq \label{S4eq15p1}
\begin{split}
\|X(t)\|_{W^{2,p}}\leq &\|X_0\|_{W^{2,p}}+C\|(D_t \v^1, D_t v)\|_{L^1_t(L^p)}\\
&+C\int_0^t \bigl(\|\nabla v\|_{L^\infty}
  +\|(\nabla^2 v, \nabla\pi)\|_{L^p}\bigr)\| X\|_{W^{2,p} }\,\dt'.\end{split}
  \eeq
Note from the hypothesis that $p\in \bigl]2, 2/(1-s_1)\bigr[,$ which implies
$$
\frac{2p}{p+2}(1-\frac{s_1}{2})<1.
$$
This together with \eqref{bound:Dv1,q} and \eqref{bound:Dv,q} ensures that
\beno
\begin{split}
\|(D_t \v^1, D_t v) \|_{L^1_t(L^p)}^2
\leq &
 \|\sigma ^{1-\frac{s_1}{2}}(D_t \v^1, D_t v)  \|_{L^{\frac{2p}{p-2}}_t(L^p)}^2
 \| \sigma ^{-(1-\frac{s_1}{2})} \|_{L^{\frac{2p}{p+2}}([0,t])}^2\leq \frak{G}_{1,X}(t).
\end{split}\eeno
And it is obvious to observe from $2/p+s_1>1$ that
$$\longformule{
\Bigl(\int_0^t \bigl(\|\nabla v\|_{L^\infty}
  +\|(\nabla^2 v, \nabla\pi)\|_{L^p}\bigr)\| X\|_{W^{2,p} }\,\dt'\Bigr)^2\lesssim \w{t}\int_0^t \|\s^{\f{1-s_0}2}\nabla v\|_{L^\infty}^2
  \| X\|_{W^{2,p} }^2\,\dt'}{{}+\int_0^t\|\s^{\left(1-\f1p-\f{s_0}2\right)}(\nabla^2 v, \nabla\pi)\|_{L^p}^2\| X\|_{W^{2,p} }^2\,\dt'\int_0^t\s^{-2+\f2p+s_0}\,\dt'\leq \frak{G}_{1,X}(t).}
$$
Substituting the above estimates into \eqref{S4eq15p1} and using the definition of $\frak{G}_{1,X}(t)$ given by \eqref{S3eq21} leads  to
\begin{align*}
 \| X(t)\|_{W^{2,p}}^2
&\leq
\|X_0\|_{W^{2,p}}^2
+ \exp\bigl(\cA_0\w{t}^2\bigr)\Bigl(\cA_1
+\int^t_0 (\frak{B}_0(t')+\s(t')^{-1+\frac{\theta_0}2})\|X(t')\|_{W^{2,p}}^2\,\dt'\Bigr).
\end{align*}
Then by virtue of \eqref{S2col3eq1}, we get, by applying Gronwall's inequality, that
\beq\label{S4eq16p}
 \| X \|_{L^\infty_t(W^{2,p})}
 \leq \cA_1 \bigl(1+\|X_0\|_{W^{2,p}}^2\bigr)\exp\left(\exp(\cA_0\w{t}^2)\right)\leq\cH_1(t),\eeq which together with Proposition \ref{prop:B1} and \eqref{S2col3eq1}
 implies that
 \beq \label{S4eq16py}
 A_{11}(t) \leq \cH_1(t).
 \eeq
Similarly, by inserting \eqref{S4eq16p} into \eqref{S4eq1t}, we arrive at
\beq\label{S4eq17p}
\begin{split}
&\|\sigma^{1-\frac{s_1}2}D_t \v^1 \|_{L^\infty_t(L^2)}^2+\|\sigma^{\frac{3-s_1}2} \na D_t \v^1 \|_{L^\infty_t(L^2)}^2\\
&\qquad\qquad+  \|\sigma^{1-\frac{s_1}2} \nabla D_t \v^1\|_{L^2_t(L^2)}^2+  \|\sigma^{\frac{3-s_1}2} (D_t^2 \v^1, \nabla^2 D_t v^1, \nabla D_t\pi^1)\|_{L^2_t(L^2)}^2
 \leq \cH_1(t).
\end{split}
\eeq
On the other side, it follows from   \eqref{Stokes:v1} for $r=2$  that
\begin{align*}
 \|\sigma^{1-\frac{s_1}{2}}
&(\nabla^2 \v^1, \nabla{\rm \pi}^1) \|_{L^\infty_t(L^2)}
\leq
C \|\sigma ^{1-\frac{s_1}{2}} (D_t \v^1, D_t v) \|_{L^\infty_t(L^2)}
\\
&+
C\bigl(
 \|\sigma ^{1-\frac{s_1}{2}}\nabla v \|_{L^\infty_t(L^{\frac{2p}{p-2}})}
 + \|\sigma^{1-\frac{s_1}{2}}
  (\nabla^2 v, \nabla\pi) \|_{L^\infty_t(L^2)} \bigr)
  \|X\|_{L^\infty_t(W^{2,p} )},
\end{align*}
from which, Proposition \ref{S1prop1}, \eqref{bound:Dv,q}, and \eqref{S4eq17p}, we infer
\begin{align*} \label{S4eq18p}
\bigl\|\sigma ^{1-\frac{s_1}{2}}
(\nabla^2 \v^1, \nabla{\rm \pi}^1) \bigr\|_{L^\infty_t(L^2)}
\leq \cH_1(t),
\end{align*}
which together with \eqref{S4eq17p} implies
\beno
A_{12}(t)\leq \cH_1(t).
\eeno
Along with \eqref{S4eq16p},  \eqref{S4eq16py}, we conclude the proof of the proposition.
\end{proof}

%%%%%%%%%%%%%%%%%%%%%%%
\subsection{Some implications }
%%%%%%%%%%%%%

In this subsection, we shall derive some useful estimates implied by Proposition \ref{S4prop1q}.

\begin{col}\label{S4col1}
{\sl Under the assumptions of of Proposition \ref{prop:Dv1}, for any $\e_0\in ]0,1[$ and
$r\in \{2,p\},$ one has
\beno
%\label{S1eq6}
\begin{split}
\dot{\frak{A}}_{1}(t) &\eqdefa
%\|\sigma ^{1-\frac{s_1}{2}} D_t \Delta X \|_{L^\infty_t(L^2)}
%+
\|\v^1\|_{L^\infty_t(L^2)}
+\|\sigma^{\frac{1-s_1}{2}}
\v^1
%, X_t)
\|_{L^\infty_t(  L^{\frac{2p}{p-2}})}
+\|\sigma^{\frac{1+\varepsilon_0-s_1}{2}}
(\v^1,  X_t) \|_{L^\infty_t(L^\infty )}
\\
&+ \bigl\| \sigma ^{ \frac{1-s_1}{2} }
(\d_X \nabla v, \nabla \v^1, \nabla  X_t ) \bigr\|_{L^\infty_t(L^2)}
+\bigl\|\sigma ^{\left(1-\frac{s_1}{2}\right)}
(   \d_X  v_t, \d_t \v^1,
  D_t  \nabla^2 X) \bigr\|_{L^\infty_t(L^2)}
\\
&
+ \bigl\| \sigma ^{\left(\frac 12+\frac 1p-\frac{s_1}{2}\right)}
( \d_X\nabla v, \nabla \v^1,
  D_t\nabla  X) \|_{L^\infty_t( L^{\frac{2p}{p-2}})}
  + \| \sigma ^{1-\frac{s_1}{2}}
(\d_X\nabla v, \nabla \v^1,
  D_t\nabla  X) \|_{L^\infty_t(L^\infty )}
\\
&+\sum_{i+j+l= 1}
\bigl\|\sigma ^{\left(\frac 32-\frac 1r-\frac{s_1}{2}\right)}
(  \d_X^i \nabla\d_X^j\nabla  \v^l,
%\d_X^i \d_t\v^j,
  D_t \v^i, \d_X^i \nabla {\rm \pi}^j
 % D_t  \nabla^2 X
 ) \bigr\|_{L^\infty_t(L^r)}
  \leq  \cH_1(t).
\end{split}
\eeno}\end{col}

\begin{proof}
In view of
Proposition \ref{S4prop1q}, we have
  for any $ r\in [2,+\infty[$
 \begin{equation}\label{S5:v1,r}
 \begin{split}
 &\|\sigma^{ \frac{1-s_1}{2}(1-\frac 2r) }
   \v^1 \|_{L^\infty_t(L^r)}
  \leq
  C \|  \v^1 \|_{L^\infty_t(L^2)}^{\frac 2r}
   \|\sigma^{1 -\frac{s_1}{2}}
  \nabla  \v^1 \|_{L^\infty_t(L^2)}^{1-\frac 2r}
  \leq\cH_1(t),
  \\
& \|\sigma^{\left(1-\frac 1r-\frac{s_1}{2}\right)}
  \nabla \v^1 \|_{L^\infty_t(L^r)}
  \leq
  C \|\sigma^{ \frac{1-s_1}{2}}
  \nabla \v^1 \|_{L^\infty_t(L^2)}^{\frac 2r}
   \|\sigma^{1 -\frac{s_1}{2}}
  \nabla^2 \v^1 \|_{L^\infty_t(L^2)}^{1-\frac 2r}
  \leq\cH_1(t),
 \end{split}
 \end{equation}
 which together with 2-D the interpolation inequality, $\|a\|_{L^\infty}\lesssim \|a\|_{L^2}^{\f{\e_0}{1+\e_0}}\|\na a\|_{L^{\f2{1-\e_0}}}^{\f1{1+\e_0}},$
 ensures that   for any $\e_0\in ]0,1[ $
\begin{align*}
\|\sigma^{\frac{1+\varepsilon_0-s_1}{2 }}\v^1\|_{L^\infty_t(L^\infty)}
&\leq
C\|\v^1\|_{L^\infty_t(L^2)}^{\frac{\varepsilon_0}{1+\varepsilon_0}}
\|\sigma^{\frac{1+\varepsilon_0-s_1}{2}}
\nabla \v^1\|_{L^\infty_t(L^{\frac{2}{1-\varepsilon_0}})}^{\frac{1}{1+\varepsilon_0}}
\leq\cH_1(t).
\end{align*}
Furthermore, in view of \eqref{eq:X} and Corollary \ref{S1col1}, for any $ \varepsilon_0\in ]0,1[,$ we infer
$$\longformule{\|\sigma^{\frac{1+\varepsilon_0-s_1}{2}}  X_t \|_{L^\infty_t(L^\infty )}
+\|\sigma^{\frac{1-s_1}{2}}  X_t\|_{L^\infty_t(  L^{\frac{2p}{p-2}})}
}{{}\leq
\cH_1(t)+
 C\|\sigma^{\frac{1-s_1}{2}}  v\|_{L^\infty_t(L^2\cap L^\infty)}
\|\nabla X\|_{L^\infty_t(L^\infty)}
\leq \cH_1(t),} $$
and
 \begin{align*}
\bigl\|&\sigma^{  \frac{1-s_1}{2} }
\bigl(\d_X\nabla v, \nabla \v^1,  \nabla X_t \bigr)\bigr\|_{L^\infty_t(L^2)}\leq
\bigl\|\sigma^{  \frac{1-s_1}{2} }
\bigl( \nabla \v^1, \nabla v\otimes\nabla X, v\otimes\nabla^2 X\bigr)\bigr\|_{L^\infty_t(L^2)}
\\
&\leq
\Bigl(\|\sigma^{  \frac{1-s_1}{2} }
(\nabla \v^1, \nabla v)\|_{L^\infty_t(L^2)}
+\|\sigma^{\frac{1-s_1}{2}}v\|_{L^\infty_t(   L^{\frac{2p}{p-2}})}\Bigr)
\bigl(1+\|X\|_{L^\infty_t(W^{2,p})}\bigr)
\leq\cH_1(t).
\end{align*}
As a result, we obtain
\beq\label{S4eq20p}
\begin{split}
\|\sigma^{\frac{1+\varepsilon_0-s_1}{2}}
(\v^1,& X_t) \|_{L^\infty_t(L^\infty )}
+\|\sigma^{\frac{1-s_1}{2}}
(\v^1, X_t)\|_{L^\infty_t(  L^{\frac{2p}{p-2}})}
\\
&+\sum_{i+j= 1}\| \sigma^{  \frac{1-s_1}{2} }
(\d_X^i \nabla \v^j,    \nabla X_t)\|_{L^\infty_t(L^2)}
  \leq \cH_1(t).
\end{split} \eeq

Whereas it follows from Proposition \ref{S4prop1q} and the 2-D interpolation inequality \eqref{S2eq6} that for any $r\in [2,+\infty[,$
 $$
 \|\sigma ^{\left(\frac 32-\frac 1r-\frac{s_1}{2}\right)}
 D_t\v^1 \|_{L^\infty_t(L^r)}\leq \|\sigma ^{1-\frac{s_1}{2}}
 D_t\v^1 \|_{L^\infty_t(L^2)}^{\f2r}\|\sigma ^{\frac{3-s_1}{2}}\na
 D_t\v^1 \|_{L^\infty_t(L^2)}^{1-\f2r}
 \leq \cH_1(t),
 $$
 which together with  \eqref{Stokes:v1}, Corollaries \ref{S1col1} and \ref{S2col3}   ensures that for   $r\in \{2,p\}$
 \beq\label{V1:v''}\begin{split}
& \bigl\|\sigma^{\left(\frac 32-\frac 1r-\frac{s_1}{2}\right)}
  (D_t\v^1,\nabla^2 \v^1, \nabla{\rm \pi}^1) \bigr\|_{L^\infty_t(L^r)}
\leq
C \bigl(1+\|X\|_{L^\infty_t(W^{2,p} )} \bigr)
\\
&\ \times\Bigl(\bigl\|\sigma^{\left(\frac 32-\frac 1r-\frac{s_1}{2}\right)}
  (D_t \v^1, D_t v, \nabla^2 v, \nabla\pi)\bigr\|_{L^\infty_t(L^r)}
+  \|\sigma ^{1-\frac{s_1}{2}}
  \nabla v\|_{L^\infty_t(L^\infty \cap L^{\frac{2p}{p-2}})} \Bigr)
\leq \cH_1(t).
 \end{split}\eeq
Then we deduce from 2-D interpolation inequality that  \beq\label{S4eq20po}
\|\sigma^{\left(1 -\frac {s_1}{2}\right)}\nabla \v^1\|_{L^\infty}
\leq C\|\sigma^{ \frac {1-s_1}{2} }\nabla \v^1\|_{L^2}^{\f{1/2-1/p}{1-1/p}}
\|\sigma^{\left(\frac 32-\frac 1p-\frac {s_1}{2}\right)}\na^2 \v^1\|_{L^p}^{\f{1/2}{1-1/p}}
\leq\cH_1(t).
\eeq

Observing that
$$\d_X\nabla v=\nabla \v^1-\nabla X^\al\p_\al v \andf
 \quad D_t\nabla X=\nabla \v^1-\nabla v_\al\p_\al X.
 $$
 Then in view of Corollaries \ref{S1col1}, \ref{S2col3} and, \eqref{S5:v1,r}, \eqref{S4eq20po}, we infer
 \beq\label{V1:v'}\begin{split}
 & \bigl\|\sigma ^{\left(\frac 12+\frac 1p-\frac {s_1}{2}\right)} ( \nabla \v^1, \d_X\nabla v, D_t\nabla X)
 \bigr\|_{L^\infty_t(  L^{\frac{2p}{p-2}})}+\bigl\|\sigma ^{1-\frac {s_1}{2}} ( \nabla \v^1, \d_X\nabla v, D_t\nabla X)
 \bigr\|_{L^\infty_t(L^\infty )}
 \\
& \leq
\Bigl(\|\sigma ^{\left(\frac 12+\frac 1p-\frac {s_1}{2}\right)}
(\nabla v, \nabla \v^1) \|_{L^\infty_t(  L^{\frac{2p}{p-2}})}
 +\|\sigma ^{1-\frac {s_1}{2}} (\nabla v, \nabla \v^1)
 \|_{L^\infty_t(L^\infty )}\Bigr)\\
&\qquad\qquad\qquad\qquad\qquad\qquad\qquad\qquad\qquad\quad\times \bigl(1 + \| \nabla X  \|_{L^\infty_t(L^\infty )}\bigr)
 \leq \cH_1(t).
 \end{split}\eeq

 Finally it is easy to observe that
 \begin{align*}
&\sum_{i+j+l=1}\bigl\|\sigma ^{\left(\frac 32-\frac 1r-\frac {s_1}{2}\right)}
 \bigl(\d_X^i\nabla\d_X^j\nabla \v^l, \d_X^i\d_t \v^j, D_t \v^i, \d_X^i\nabla\rm{\pi}^j ,
 D_t\nabla^2 X\bigr)\bigr\|_{L^\infty_t(L^r)}
 \\
   &=\bigl\|\sigma ^{\left(\frac 32-\frac 1r-\frac {s_1}{2}\right)}
 \bigl(
 \d_X\nabla^2 v, \nabla\d_X\nabla v,  \nabla^2 \v^1,
  \d_X v_t,
  \d_t \v^1, D_t \v^1,
 \d_X \nabla \pi, \nabla{\rm \pi}^1,  D_t\nabla^2 X\bigr)\bigr\|_{L^\infty_t(L^r)}
 \\
 &\leq
 \bigl\|\sigma ^{\left(\frac 32-\frac 1r-\frac {s_1}{2}\right)}
 \bigl(\nabla^2 \v^1,  D_t \v^1,   \nabla{\rm \pi}^1\bigr)\bigr\|_{L^\infty_t(L^r)}
 \\
&\qquad +
 \bigl\|\sigma ^{\left(\frac 32-\frac 1r-\frac {s_1}{2}\right)}
 \bigl(
  \nabla^2 X\otimes\nabla v, \nabla X\otimes\nabla^2 v,  v\cdot\nabla \v^1, \nabla X\otimes\nabla\pi,
 X_t\otimes\nabla v\bigr)\bigr\|_{L^\infty_t(L^r)}.
 \end{align*}
While it follows from Corollary \ref{S1col1} and \eqref{S5:v1,r} that for any $r\in [2,\infty[$
\begin{align*}
 \|\sigma ^{\left(\frac 32-\frac 1r-\frac {s_1}{2}\right)}
 ( \v^1\otimes\nabla v, & v\otimes\nabla \v^1) \|_{L^\infty_t(L^r)}\\
& \leq  \|\sigma ^{\frac 12  }  (v, \v^1 ) \|_{L^\infty_t(L^\infty)}
 \|\sigma ^{\left(1-\frac 1r-\frac {s_1}{2}\right)}
( \nabla v , \nabla \v^1) \|_{L^\infty_t(L^r)}
  \leq \cH_1(t),
 \end{align*}
 which together with \eqref{eq:X}  and \eqref{S2col3eq1} ensures that for   $ r\in \{2,p\},$
  \begin{align*}
& \bigl\|\sigma ^{\left(\frac 32-\frac 1r-\frac {s_1}{2}\right)}
 \bigl(\nabla^2 X\otimes\nabla v, \nabla X\otimes\nabla^2 v, \nabla X\otimes\nabla\pi,
 X_t\otimes\nabla v, v\otimes\nabla \v^1 \bigr)\bigr\|_{L^\infty_t(L^r)}
 \\
 &\leq
 C\bigl(1+\|X\|_{L^\infty_t(W^{2,p})}\bigr)
 \Bigl( \|\sigma ^{1-\frac {s_1}{2}}   \nabla v \|_{L^\infty_t(L^\infty\cap L^{\frac{2p}{p-2}})}\\
 &\qquad\qquad +
 \|\sigma ^{\left(\frac 32-\frac 1r-\frac {s_1}{2}\right)}
 (\nabla^2 v, \nabla \pi,
 v\otimes\nabla v, \v^1\otimes\nabla v, v\otimes\nabla \v^1) \|_{L^\infty_t(L^r)} \Bigr)
 \leq
 \cH_1(t),\quad
 \end{align*}
 from which
 and \eqref{V1:v''}, we infer that for  $ r\in \{2,p\},$
 \begin{align*}
&\sum_{i+j+l=1}\bigl\|\sigma ^{\left(\frac 32-\frac 1r-\frac {s_1}{2}\right)}
 \bigl(\d_X^i\nabla\d_X^j\nabla \v^l,  \d_X^i\d_t \v^j, D_t\v^i, \d_X^i\nabla{\rm \pi}^j,
 D_t\nabla^2 X\bigr)\bigr\|_{L^\infty_t(L^r)}
  \leq \cH_1(t).
 \end{align*}
 This together with
  \eqref{S4eq20p} and \eqref{V1:v'} completes the proof of the corollary
 \end{proof}

 \begin{col}\label{S4col2}
{\sl Under the assumptions Corollary \ref{S4col1}, one has
\beq \label{S1eq7}
\begin{split}
\int_0^t\dot{\frak{B}}_1(t')\,dt' \leq \cH_1(t),\end{split}\eeq
where
\beno
\begin{split}
\dot{\frak{B}}_1(t) \eqdefa
 \sum_{i+j = 1}
\Bigl(& \bigl\|\sigma^{\left(\frac 12+\frac 1p-\frac{s_1}{2}\right)} D_t \v^1(t)\|_{L^{\frac{2p}{p-2}}}^2
+\bigl\|\sigma^{1-\frac{s_1}{2}}
(\d_X^i D_t \nabla \v^j, \d_X^i \nabla D_t \v^j)(t)\bigr\|_{L^2}^2\\
&+\bigl\|\sigma^{\left(1+\frac 1p-\frac{s_1}{2}\right)}
(\d_X^i D_t \nabla \v^j, \d_X^i \nabla D_t \v^j)(t)\bigr\|_{L^{\f{2p}{p-2}}}^2
 \\
 &+\bigl\|\sigma^{\frac{3-s_1}2}\bigl( D_t^2 \v^i, \d_X^i D_t \nabla^2\v^j,
\nabla\d_X^i \nabla D_t\v^j, \d_X^i D_t \nabla{{\rm \pi}}^j\bigr)(t)\bigr\|_{L^2}^2 \Bigr).
\end{split}\eeno}
\end{col}

\begin{proof}
%We first deduce from \eqref{S4eq17p} that
%\beno
%\bigl\|\sigma^{1-\frac{s_1}2}( D_t \nabla \v^1,\na\p_t\v^1)\bigr\|_{L^2_t(L^2)}
%\leq C\bigl\|\sigma^{1-\frac{s_1}2} \bigl(\nabla D_t \v^1, \na(v\cdot\na \v^1)\bigr)\bigr\|_{L^2_t(L^2)}\leq \cH_1(t).
%\eeno
%While
It follows from
Corollary \ref{lem:Dv1}
%Proposition \ref{S1prop2}
and
Proposition \ref{S4prop1q} that
  \beno\label{cV1:Dv}\begin{split}
  \|\sigma^{\left(\frac 12+\frac 1p-\frac{s_1}{2}\right)}
  %(D_t v,
  D_t \v^1\|_{L^2_t(L^{\frac{2p}{p-2}})}
 \leq &
C  \|\sigma^{\frac{1-s_1}{2}}
 % ( D_t v,
 D_t \v^1\|_{L^2_t(L^2)}^{1-\frac 2p}
\|\sigma^{1- \frac{s_1}{2}}
 % (\nabla D_t v,
 \nabla D_t \v^1 \|_{L^2_t(L^2)}^{\frac 2p}
   \leq \cH_1(t),
  \end{split}\eeno
and  for any $r\in [2,+\infty[$
\begin{align*}
 \|\sigma^{\frac 32-\frac 1r- \frac{s_1}{2}}
   \nabla D_t \v^1\|_{L^2_t(L^r)}
 \leq
  C\|\sigma^{\left(1- \frac{s_1}{2}\right)}
  \nabla D_t \v^1 \|_{L^2_t(L^2)}^{\frac2r}
  \|\sigma^{\left(\frac 32 - \frac{s_1}{2}\right)}
   \nabla^2 D_t \v^1 \|_{L^2_t(L^2)}^{1-\frac 2r}
  \leq\cH_1(t).
\end{align*}
As a result, for  $r=2, \frac{2p}{p-2},$ we deduce from  Corollary \ref{S2col3}
  that
$$ \longformule{\sum_{i+j=1}\|\sigma^{\left(\frac 32-\frac 1r- \frac{s_1}{2}\right)}
\d_X^i \nabla D_t \v^j\|_{L^2_t(L^r)}
 \leq
  \bigl\|\sigma^{\left(\frac 32-\frac 1r- \frac{s_1}{2}\right)}
  \bigl(\nabla D_t \v^1, [\d_X;\nabla]   D_t v\bigr)\bigr\|_{L^2_t(L^r)}
 }{{}\leq
  \|\sigma^{\left(\frac 32-\frac 1r- \frac{s_1}{2}\right)}
  (\nabla D_t v, \nabla D_t \v^1 )\|_{L^2_t(L^r)}
  (1+\|\nabla X\|_{L^\infty_t(L^\infty)})
  \leq \cH_1(t).} $$
Moreover, for  $r=2$ or $r=\f{2p}{p-2},$ by virtue of Corollary \ref{S2col3} and Corollary \ref{S4col1}, we infer
\beno \begin{split}
&\sum_{i+j= 1}\|\sigma^{\left(\frac 32-\frac 1r- \frac{s_1}{2}\right)}
\d_X^i D_t\nabla \v^j\|_{L^2_t(L^r)}\leq
  \bigl\|\sigma^{\bigl(\frac 32-\frac 1r- \frac{s_1}{2}\bigr)}
  \bigl( D_t\nabla \v^1, D_t [\d_X; \nabla] v\bigr)\bigr\|_{L^2_t(L^r)}
  \\
  &\leq
  \bigl\|\sigma^{\left(\frac 32-\frac 1r- \frac{s_1}{2}\right)}
  (\nabla D_t \v^1, D_t \nabla v)\|_{L^2_t(L^r)}
  (1+\|\nabla X\|_{L^\infty_t(L^\infty)})
  \\
&\qquad\qquad\qquad\qquad
  +
  \|\sigma^{\left(\frac 32-\frac 1r- \frac{s_1}{2}\right)}
  ( \nabla v\otimes\nabla \v^1,
  D_t\nabla X\otimes\nabla v )\|_{L^2_t(L^r)}
  \\
  &\leq
 \cH_1(t)
 +
 \bigl( \|\sigma^{\frac 12-\frac 1p }
  \nabla v \|_{L^2_t(L^p)}
  +\|\sigma^{\frac{1 }2}
  \nabla v \|_{L^2_t(L^\infty)} \bigr)
 \|\sigma^{\left(\frac 12+\frac 1p-\frac{s_1}{2}\right)}
 (\nabla \v^1,     D_t\nabla X )\|_{L^\infty_t(L^{\frac{2p}{p-2}})}
 \\
&  \leq \cH_1(t).
\end{split}\eeno
It is easy to observe that
 \beno
 \begin{split}
 & [D_t; \nabla^2 ]\v^1=-\nabla^2 v_\al\p_\al \v^1-2\na v^\alpha\p_\alpha\nabla \v^1,
  \qquad\qquad
  [D_t; \nabla]{\rm \pi}^1=-\nabla v_\al\p_\al{\rm \pi}^1,
  \\
 & D_t [\d_X; \nabla^2 ]v=-D_t(\nabla^2 X^\al\p_\al v+2\nabla X^\al\p_\al\nabla v),
 \quad \  D_t[\d_X; \nabla]\pi=-D_t(\nabla X^\al\p_\al\pi),
  \\
&\nabla[\d_X; \nabla]D_t v=-\nabla(\nabla X^\al\p_\al D_t v),
\end{split}
  \eeno
we deduce from Corollaries \ref{S2col3}, \ref{lem:Dv1} and \ref{S4col1},  Proposition \ref{S4prop1q} that
  \begin{align*}
  &\sum_{i+j =1}  \bigl\|\sigma ^{ \frac{3-s_1}{2}}
  \bigl(D_t^2 \v^i, \d_X^i D_t\nabla^2 \v^j, \nabla\d_X^i\nabla D_t \v^j,
  \d_X^i D_t\nabla{\rm \pi}^j\bigr)\|_{L^2_t(L^2)}
  \\
  &\lesssim\bigl\|\sigma ^{ \frac{3-s_1}{2}}
 \bigl (  D_t^2 \v^1,
  \nabla^2 D_t \v^1,
  \nabla D_t{\rm \pi}^1, \nabla^2 v\otimes\nabla \v^1,
   \nabla v\otimes\nabla^2 \v^1, \nabla v\otimes\nabla{\rm \pi}^1 \bigr)\bigr\|_{L^2_t(L^2)}
  \\
  &\quad+\bigl\|\sigma ^{ \frac{3-s_1}{2}}
 D_t\bigl (\nabla^2 X\otimes \nabla v,
  \nabla X\otimes(\nabla^2 v, \nabla\pi) \bigr)\bigr\|_{L^2_t(L^2)}
  +\bigl\|\sigma ^{ \frac{3-s_1}{2}}
  \nabla(\nabla X\otimes\nabla D_t v) \bigr\|_{L^2_t(L^2)}
\\
&\lesssim \cH_1(t)
+\|\sigma ^{\frac 12}(\nabla v, D_t\nabla X )\|_{L^2_t(L^\infty)}
\|\sigma ^{1-\frac{s_1}{2}}(D_t\nabla^2 X, \nabla^2 v,
 \nabla\pi)\|_{L^\infty_t(L^2)}
\\
&\quad +\|X\|_{L^\infty_t(W^{2,p})}
\Bigl(\|\sigma ^{ \frac{3-s_1}{2}}
(D_t\nabla v, \nabla D_t v)\|_{L^2_t(L^{\frac{2p}{p-2}})}\\
&\qquad\qquad\qquad\qquad\qquad\qquad\qquad+\|\sigma ^{ \frac{3-s_1}{2}}
(D_t\nabla^2 v, \nabla^2 D_t v, D_t\nabla\pi)\|_{L^2_t(L^2)}\Bigr)
\leq \cH_1(t).
  \end{align*}
  Summing up the above estimates leads to \eqref{S1eq7}.
\end{proof}

%%%%%%%%%%%
\subsection{Propagation of the $\cB^{s_1}$-regularity for $\v^1$}\label{sectpxv}

We shall use the same  strategy as that used in the proof of Proposition
\ref{S2prop1} to study the Besov  regularity of $(\v^1,\na{\rm \pi}^1).$
Indeed let $(v_q,\na \pi_q)$ be the unique solution of
\eqref{S1eq15}, we shall first
investigate the following system for the unknown $(\u_q,\nabla\mathfrak{p}_q)$:
\begin{equation}\label{S4eq1}
 \left\{\begin{array}{l}
\displaystyle\r\d_t\u_q+\r v\cdot\na\u_q-\Delta \u_q+\nabla\mathfrak{p}_q=F_{1,q}\with\\
\ F_{1,q}\eqdefa -\rr^1 D_t v_q -(\Delta
X\cdot\nabla v_q+2\p_\al X\cdot\nabla\p_\al v_q) +\nabla
X^\al\p_\al \pi_q,\\
\displaystyle \div \u_q=\dive(v_q\cdot\na X),\\
\displaystyle  \u_q|_{t=0}=0.
\end{array}\right.
\end{equation}
Let $(\v_{12},\na \rp_{12})$ be a smooth enough solution of \eqref{S3eq5}.
Then by virtue of \eqref{S1eq3}, \eqref{S3eq5} and \eqref{S4eq1},
we have \beq \label{S4eq2} \v_{12}=\sum_{q\in\Z}\u_q \andf \na
\rp_{12}=\sum_{q\in\Z}\na \frak{p}_q. \eeq

\subsubsection{The $L^2$ energy estimate}

The following lemma concerning the norms of solution to
\eqref{S1eq15} will be very useful in the estimates that follows.

\begin{lem}\label{S4lem1}
{\sl Let $(v_q,\na \pi_q)$ be the unique  solution determined
by the System \eqref{S1eq15}. Let
$\th_0$ and $\cC(v_0,s_0)$ be given by \eqref{S1eq9}, then the following inequality is valid \beq\label{S4eq3}
\begin{split}
\|&v_q\|_{L^\infty_t(L^{\f2{1-\th_0}})}+ \bigl\|\s^{\f{1-\th_0}2}
v_q\bigr\|_{L^\infty_t(L^\infty)}+\|\na
v_q\|_{L^2_t(L^{\f2{1-\th_0}})}+ \bigl\|\s^{\f{1-\th_0}2}\na
v_q\bigr\|_{L^\infty_t(L^2)}\\
&+\bigl\|\s^{\f{1-\th_0}2}\na
v_q\bigr\|_{L^2_t(L^\infty)}
+\bigl\|\s^{\f{1-\th_0}2}(D_t v_q, \p_tv_q,\na^2v_q,\na\pi_q)\bigr\|_{L^2_t(L^2)}\\
&+\bigl\|\s^{\f{2-\th_0}4}(D_t v_q, \p_tv_q,\na^2v_q,\na\pi_q)\bigr\|_{L^2_t(L^{\f4{2-\th_0}})}\leq
\cC(v_0,s_0)\w{t}^{\f12} d_q2^{-qs_1}.
\end{split}
\eeq}
\end{lem}

\begin{proof}
It follows from  \eqref{S2eq6}, \eqref{S1eq15} and \eqref{S2eq21a},
\eqref{S2eq22}, \eqref{S2eq23}  that for any $\de\in ]0,1[$ \beno
\begin{split}
&\|v_q\|_{L^\infty_t(L^{\f2{1-\th_0}})}\lesssim
\|v_q\|_{L^\infty_t(L^2)}^{1-\th_0}\|\na
v_q\|_{L^\infty_t(L^2)}^{\th_0}\leq \cC(v_0,s_0) d_q2^{-qs_1},\\
&\|\na v_q\|_{L^2_t(L^{\f2{1-\th_0}})}\lesssim \|\na
v_q\|_{L^2_t(L^2)}^{1-\th_0}\|\na^2
v_q\|_{L^2_t(L^2)}^{\th_0}\leq \cC(v_0,s_0) d_q2^{-qs_1},\\
&\bigl\|\s^{\f{1-\de}2}v_q\bigr\|_{L^\infty_t(L^{\f2\de})}\lesssim
\|v_q\|_{L^\infty_t(L^2)}^\de\bigl\|\s^{\f12}\na
v_q\bigr\|_{L^\infty_t(L^2)}^{1-\de}\leq \cC(v_0,s_0) d_q
2^{-qs_0},\end{split} \eeno
which together with the 2-D
interpolation inequality  that
\beno \|a\|_{L^\infty}\lesssim
\|a\|_{L^{\f2\de}}^{\f1{1+\de}}\|\na^2 a\|_{L^2}^{\f\de{1+\de}},
\eeno
for $\de=\f{\th_0}{1-\th_0}$ and \eqref{S2eq23a} ensures that
\beno
\bigl\|\s^{\f{1-\th_0}2}
v_q\bigr\|_{L^\infty_t(L^\infty)}
\lesssim
\bigl\|\s^{\f{1-2\th_0}{2(1-\th_0)}}
 v_q\|_{L^\infty_t(L^{\f{2(1-\th_0)}{\th_0}})}^{1-\th_0}
\bigl\|\s^{\f12}\na^2v_q\|_{L^\infty_t(L^2)}^{\th_0} \leq \cC(v_0,s_0)
d_q2^{-qs_1}. \eeno
It is easy to observe that
\beno \begin{split}
&\bigl\|\s^{\f{1-\th_0}2}\na v_q\bigr\|_{L^\infty_t(L^2)}\leq \|\na
v_q\|_{L^\infty_t(L^2)}^{\th_0}\bigl\|\s^{\f12}\na
v_q\bigr\|_{L^\infty_t(L^2)}^{1-\th_0},\\
&\bigl\|\s^{\f{1-\th_0}2}\p_tv_q\bigr\|_{L^2_t(L^2)}
\leq
\|\p_tv_q\|_{L^2_t(L^2)}^{\th_0}
\bigl\|\s^{\f12}\p_tv_q\bigr\|_{L^2_t(L^2)}^{1-\th_0},
\end{split} \eeno
which together with \eqref{S2eq22} and \eqref{S2eq23} implies
that
\beno \bigl\|\s^{\f{1-\th_0}2}\na
v_q\bigr\|_{L^\infty_t(L^2)}
+\bigl\|\s^{\f{1-\th_0}2}\p_tv_q\bigr\|_{L^2_t(L^2)}
\leq
\cC(v_0,s_0) d_q2^{-qs_1}. \eeno
While note from  \eqref{S2eq6},
\eqref{S2eq23} and \eqref{S2eq23a} that for $\de\in [0,1[,$ there
holds
\beq\label{S4eq5} \begin{split}
&\bigl\|\s^{\f12}\na
v_q\bigr\|_{L^\infty_t(L^{\f2{1-\de}})}
\leq C\bigl\|\s^{\f12}\na
v_q\bigr\|_{L^\infty_t(L^2)}^{1-\de}
\bigl\|\s^{\f12}\na^2v_q\bigr\|_{L^\infty_t(L^2)}^{\de}
\leq  \cC(v_0,s_0) d_q2^{-q(s_0-\de)},
\\
&\bigl\|\s^{\f12}\p_tv_q\bigr\|_{L^2_t(L^{\f2{1-\de}})}\leq
C\bigl\|\s^{\f12}\p_tv_q\bigr\|_{L^2_t(L^2)}^{1-\de}
\bigl\|\s^{\f12}\na\p_tv_q\bigr\|_{L^2_t(L^2)}^{\de}
\leq \cC(v_0,s_0) d_q2^{-q(s_0-\de)}.
\end{split}
\eeq
Then it follows from the classical estimate on Stokes operator
and  Corollary \ref{S1col1}, \eqref{S1eq15} that
 \beq \label{S4eq15}
\begin{split}
\bigl\|\s^{\f12}(&\na^2v_q,\na\pi_q)\bigr\|_{L^2_t(L^{\f2{1-\de}})}
\lesssim
\bigl\|\s^{\f12}\p_tv_q\bigr\|_{L^2_t(L^{\f2{1-\de}})}\\
&+\w{t}^{\f{1}2}\|\s^{\f{1-s_0}2}v\|_{L^\infty_t(L^\infty)}
\bigl\|\s^{\f{1}2}\na v_q\bigr\|_{L^\infty_t(L^{\f2{1-\de}})}
\leq  \cC(v_0,s_0)\w{t}^{\f12}d_q2^{-q(s_0-\de)}.
\end{split}
\eeq
Taking $\de=\f{\th_0}{2-\th_0}$ in the second inequality of
\eqref{S4eq5}  and using \eqref{S2eq22} gives rise to \beno
\begin{split}
\bigl\|\s^{\f{2-\th_0}4}\p_tv_q\bigr\|_{L^2_t(L^{\f4{2-\th_0}})}\leq
&
\|\p_tv_q\|_{L^2_t(L^2)}^{\f{\th_0}2}
\bigl\|\s^{\f12}\p_tv_q\bigr\|_{L^2_t(L^{\f{2-\th_0}{1-\th_0}})}^{1-\f{\th_0}2}
\leq \cC(v_0,s_0) d_q2^{-qs_1},
\end{split}
\eeno
which together with a similar derivation of \eqref{S4eq15}
ensures that \beno
\bigl\|\s^{\f{2-\th_0}4}(\na^2v_q,\na\pi_q)\bigr\|_{L^2_t(L^{\f4{2-\th_0}})}
+\bigl\|\s^{\f{1-\th_0}2}(\na^2v_q,\na\pi_q)\bigr\|_{L^2_t(L^2)}\leq
\cC(v_0,s_0)\w{t}^{\f12}d_q2^{-qs_1}. \eeno
Furthermore, thanks to the
following interpolation inequality
\beno \|a\|_{L^\infty}\lesssim
\|a\|_{L^2}^{\f{\de}{1+\de}}\|\na a\|_{L^{\f2{1-\de}}}^{\f1{1+\de}},
\eeno
and \eqref{S4eq15}, we infer \beno
\begin{split}
\bigl\|\s^{\f{1-\th_0}2}\na v_q\bigr\|_{L^2_t(L^\infty)}\lesssim
&\|\na v_q\|_{L^2_t(L^2)}^{\th_0}\bigl\|\s^{\f{1}2}\na^2
v_q\bigr\|_{L^2_t(L^{\frac{2(1-\theta_0)}{1-2\th_0}})}^{1-\th_0}\leq \cC(v_0,s_0)\w{t}^{\f12}d_q2^{-qs_1}.
\end{split}
\eeno By summing up the above estimates and making use of the fact that $\rho D_t v_q=\Delta v_q-\nabla\pi_q$, we complete the proof of
\eqref{S4eq3}.
\end{proof}

Let us turn to the $L^2$ energy estimate of $\u_q.$

\begin{lem}\label{S4prop1}
{\sl Let $(\u_q, \na\frak{p}_q)$ be the unique solution of \eqref{S4eq1}.
Then there holds
\beq\label{S4eq9} \|\u_q\|_{L^\infty_t(L^2)}+\|\na
\u_q\|_{L^2_t(L^2)}\leq
 \cH_1(t)d_q2^{-qs_1}. \eeq}
\end{lem}

\begin{proof} We first get, by using $L^2$ energy estimate to the $\u_q$
equation of \eqref{S4eq1}, that \beq \label{S4eq10}
\f12\|\sqrt{\r}\u_q(t)\|_{L^2}^2
+\|\na \u_q\|_{L^2_t(L^2)}^2
=-\int_0^t\int_{\R^2}\na \frak{p}_q | \u_q\,\dx\,\dt'
+\int_0^t\int_{\R^2} F_{1,q} | \u_q\,\dx\,\dt'. \eeq
It follows from the definition of $F_{1,q}$ and
\eqref{brho} that
\beno
\begin{split}
\|\sigma^{\frac{2-\theta_0}{4}} F_{1,q}\|_{L^2_t(L^2)}
 \lesssim &
 \|\sigma^{\frac{1-\theta_0}{2}} D_tv_q\|_{L^2_t(L^2)}
 +\|\D X\|_{L^\infty_t(L^p)}
 \|\sigma^{\frac{1-\theta_0}{2}}\na v_q\|_{L^2_t(L^{\f{2p}{p-2}})}
 \\
&\qquad
+\|\na X\|_{L^\infty_t(L^{\f4{\th_0}})}
 \|\sigma^{\frac{2-\theta_0}{4}}
 (\na^2 v_q,\nabla{\pi}_q)\|_{L^2_t(L^{\f4{2-\th_0}})} ,
\end{split}
\eeno
from which,  Proposition \ref{S4prop1q} and Lemma \ref{S4lem1},
we infer for $\beta=\f{p-2}{(1-\th_0)p}\in ]0,1[$
\beq\label{S4eq11-}
\begin{split}
\|\sigma^{\frac{2-\theta_0}{4}} F_{1,q}\|_{L^2_t(L^2)}
&\lesssim
\cH_1(t)d_q 2^{-qs_1}
+\Bigl(\|\na v_q\|_{L^2_t(L^{\f2{1-\th_0}})}^{ \beta}
\bigl\|\s^{\f{1-\th_0}2}\na
v_q\bigr\|_{L^2_t(L^\infty)}^{ 1-\beta }\\
&+
 \bigl\|\s^{\f{2-\th_0}4}(\na^2 v_q,\na
\pi_q)\bigr\|_{L^2_t(L^{\f4{2-\th_0}})}
\Bigr)\|\na
X\|_{L^\infty_t(W^{1,p})}
\lesssim
 \cH_1(t) d_q 2^{- qs_1}.
\end{split}
\eeq
We thus obtain
\beq\label{S4eq11}
\begin{split}
\bigl|\int_0^t&\int_{\R^2} F_{1,q} | \u_q\,\dx\,\dt'\bigr|
\leq \int^t_0\s^{-1+ \f{ \th_0}2}\|\u_q\|_{L^2}^2\,\dt'
+\cH_1(t) d_q^22^{-2qs_1}.
\end{split}
\eeq
 While we get, by using integration by parts and \eqref{S4eq1}, that
\beq\label{S4eq14}
\begin{split}
\int_{\R^2}\na\frak{p}_q | \u_q\,\dx
=&\int_{\R^2} \na\frak{p}_q | v_q\cdot\na X\,\dx
\\
=&\int_{\R^2}\left(F_{1,q}-\p_t(\r \u_q)-\dive(\r
v\otimes\u_q)+\D\u_q\right) |  v_q\cdot\na X\,\dx.
\end{split}
\eeq
In view of Proposition \ref{S4prop1q} and   Lemma \ref{S4lem1}, one has \beq \label{S4eq13}
\|v_q\cdot\na X\|_{L^\infty_t(L^2)}\leq
\|v_q\|_{L^\infty_t(L^{\f2{1-\th_0}})}\|\na
X\|_{L^\infty_t(L^{\f2{\th_0}})}
\leq
 \cH_1 (t)d_q2^{-qs_1}, \eeq
so that a similar derivation as \eqref{S4eq11} ensures
\beno
\bigl|\int_0^t\int_{\R^2}F_{1,q} |  v_q\cdot\na X\,dx\,\dt'\bigr|
\leq
 \cH_1 (t)d_q^22^{-2qs_1}.
 \eeno
Again by  integration by parts, we write
\beno \int_{\R^2}\p_t(\r \u_q) | v_q\cdot\na X\,\dx
=\f{d}{dt}\int_{\R^2}\r \u_q | v_q\cdot\na X\,\dx
-\int_{\R^2}\r \u_q|  \p_t(v_q\cdot\na X) \,\dx.
\eeno
Yet it follows from   Proposition \ref{S4prop1q}, Corollary \ref{S4col1} and Lemma
\ref{S4lem1} that
 \beno
\begin{split}
\bigl\|\sigma^{\frac{2-\theta_0}{4}}\d_t(v_q\cdot\nabla X)\bigr\|_{L^2_t(L^2)}
&\leq
 \|\na X\|_{L^\infty_t(W^{1,p})}
 \bigl\|\s^{\f{2-\th_0}4}\p_tv_q\bigr\|_{L^2_t(L^{\f4{2-\th_0}})}
 \\
&\quad
 +C\w{t}\|\s^{\f{1-\th_0}2} v_q\|_{L^\infty_t(L^\infty)}
 \|\s^{\f{1-s_1}2}\nabla X_t\|_{L^\infty_t(L^2)}
 \leq
  \cH_1 (t)d_q 2^{- qs_1},
\end{split}
\eeno
which together with   \eqref{S4eq13} implies
 \beno
\begin{split}
 \bigl|\int_0^t\int_{\R^2}\p_t(\r \u_q) |
&v_q\cdot\na X\,\dx\,\dt'\bigr|
\leq
\f14\|\sqrt{\r}\u_q(t)\|_{L^2}^2+
C\|v_q\cdot\na X\|_{L^\infty_t(L^2)}^2
\\
&\qquad\qquad
+C\int_0^t\s^{-1+\f{\th_0}2}\|\u_q(t')\|_{L^2}^2\,\dt'
+C\|\sigma^{\frac{2-\theta_0}{4}}\d_t(v_q\cdot\nabla X)\|_{L^2_t(L^2)}^2
\\
 &\qquad\quad\leq
\f14\|\sqrt{\r}\u_q\|_{L^2}^2
+C\int_0^t\s^{-1+\f{\th_0}2}\|\u_q(t')\|_{L^2}^2\,dt'
+  \cH_1 (t)d_q^22^{-2qs_1}.
\end{split}
\eeno
Along the same line to the derivation of \eqref{S4eq11-}, one has
\beno
\begin{split}
\bigl\|\nabla( v_q\cdot\na X)\bigr\|_{L^2_t(L^2)}
&\leq
\w{t}\Bigl(\|\na v_q\|_{L^2_t(L^{\f2{1-\th_0}})}
+\|v_q\|_{L^\infty_t(L^{\f2{1-\th_0}})}^{ \beta}
\|\s^{\f{1-\th_0}2}v_q\|_{L^\infty_t(L^\infty)}^{ 1-\beta }\Bigr)
\|\na X\|_{L^\infty_t(W^{1,p})}
\\
&\leq \cH_1(t)d_q2^{-qs_1},
\end{split}
\eeno
which ensures that
 \beno \begin{split}
\bigl|\int_0^t\int_{\R^2}\dive(\r v\otimes\u_q) |&
 v_q\cdot\na X\,\dx\,\dt'\bigr|
 \leq
 C\int_0^t\|v\|_{L^4}\|\u_q\|_{L^4}
 \|\nabla(v_q\cdot\nabla X)\|_{L^2}\,\dt'
\\
&\leq \f16\|\na \u_q\|_{L^2_t(L^2)}^2
+\int_0^t\|v\|_{L^4}^4\|\u_q\|_{L^2}^2\,\dt'
+ \cH_1 (t)d_q^22^{-2qs_1},
\end{split} \eeno
and \beno
\bigl|\int_0^t\int_{\R^2}\D \u_q | v_q\cdot\na X\,\dx\,\dt'\bigr|\leq
\f16\|\na\u_q\|_{L^2_t(L^2)}^2
+\cH_1 (t)d_q^22^{-2qs_1}. \eeno
Inserting the above estimates into \eqref{S4eq14} gives rise to \beq \label{S4eq4}
\begin{split}
\bigl|\int_0^t\int_{\R^2}\na\frak{p}_q |& \u_q\,\dx\,\dt'\bigr|\leq
\f13\bigl(\|\sqrt{\r}\u_q(t)\|_{L^2}^2+\|\na\u_q\|_{L^2_t(L^2)}^2\bigr)
\\
&+C\int_0^t\bigl(\|v\|_{L^4}^4+\s^{-1+\f{\th_0}2}\bigr)\|\u_q(t')\|_{L^2}^2\,\dt'
+\cH_1 (t)d_q^22^{-2qs_1}.
\end{split}
\eeq
Substituting the
Inequalities \eqref{S4eq11} and \eqref{S4eq4} into \eqref{S4eq10}, we arrive at \beno
\|\sqrt{\r}\u_q(t)\|_{L^2}^2+\|\na\u_q\|_{L^2_t(L^2)}^2\leq
C\int_0^t\bigl(\|v\|_{L^4}^4+\s^{-1+\f{\th_0}2}\bigr)\|\u_q(t')\|_{L^2}^2\,dt'
+  \cH_1(t)d_q^22^{-2qs_1}.
\eeno
 Then applying Gronwall's inequality  leads to \eqref{S4eq9}.
This completes the proof of the lemma.
\end{proof}

\subsubsection{The $\dH^1$ energy estimate}

\begin{lem}\label{S4prop2}
{\sl Under the asme assumptions of Lemma \ref{S4prop1}, one
has \beq\label{S4eq6}
\|\na\u_q\|_{L^\infty_t(L^2)}
+\bigl\|(\p_t\u_q,\na^2 \u_q,\na\frak{p}_q)\bigr\|_{L^2_t(L^2)}
%\leq  \exp(C\cH_1(t))d_q2^{q(1-s_1)}
\leq \cH_1(t)d_q2^{q(1-s_1)}.
\eeq
}
\end{lem}
%%%%%%%%%
\begin{proof} In view of
\eqref{S4eq1}, we get, by applying Lemma \ref{S2lem1}, that
$$\longformule{
\f{d}{dt}\|\na\u_q(t)\|_{L^2}^2+\|(\p_t\u_q,\na^2\u_q,\na\frak{p}_q)\|_{L^2}^2}{{}\leq
C\bigl(\|v\|_{L^4}^4\|\na\u_q\|_{L^2}^2
+\|\p_t(v_q\cdot\na X)\|_{L^2}^2
+\|(F_{1,q}, \na\dive(v_q\cdot\na X)) \|_{L^2}^2 \bigr).
} $$
 Applying Gronwall's inequality
yields
 \beq \label{S4eq8}
\begin{split}
\|\na&\u_q\|_{L^\infty_t(L^2)}^2+\|(\p_t\u_q,\na^2\u_q,\na\frak{p}_q)\|_{L^2_t(L^2)}^2
\\
&\leq
\exp\bigl(C\|v_0\|_{L^2}^4\bigr)
\left(\|\p_t(v_q\cdot\na X)\|_{L^2}^2
+\|(F_{1,q}, \na\dive(v_q\cdot\na X))\|_{L^2_t(L^2)}^2
\right).
\end{split} \eeq

Note that
since $v_0\in L^2\cap\cB^{s_0},$ we observe
that the Inequalities \eqref{S2eq21a} to \eqref{S2eq23a}  hold for
any $s\in ]0,s_0[.$
In particular, it follows
from \eqref{S2eq22}, \eqref{S2eq23a} and $\rho D_t v_q=\Delta v_q-\nabla\pi_q$ that
 \beq\label{S4eq17}
\begin{split}
\|(\na v_q, \sigma^{\frac 12}\nabla^2 v_q)\|_{L^\infty_t(L^2)}
+\bigl\|(D_t v_q, \na^2 v_q,\p_tv_q,\na \pi_q)\bigr\|_{L^2_t(L^2)}
\lesssim
&\cC(v_0,s_1)d_q 2^{ q(1-s_1)},
\end{split} \eeq
  which together with \eqref{S2eq6} ensures that for any $\delta\in ]0,1[$,
 \beq \label{S4eq20} \begin{split}
 \|\na v_q\|_{L^{\frac{2}{1-\delta}}_t(L^{\f{2}{\delta}})}
\lesssim
\|\na v_q\|_{L^\infty_t(L^{2})}^{\delta}
 \|\na^2 v_q\|_{L^2_t(L^{2})}^{1-\delta}
 \lesssim
&\cC(v_0,s_1)d_q2^{q(1-s_1)}.
\end{split} \eeq
Similarly for any $\delta\in ]0,1[$, we have
\beno\label{S4eq17a}
\begin{split}
\|v_q\|_{L^\infty_t(L^{\f{2}{\delta}})}
\lesssim
&\|v_q\|_{L^\infty_t(L^2)}^{\delta}
\|\na v_q\|_{L^\infty_t(L^2)}^{1-\delta}
\lesssim
\cC(v_0,s_1-\delta)d_q2^{q(1-s_1)},
\end{split} \eeno
and
\beq\label{S4eq17b}
\begin{split}
\|\sigma^{\frac{\delta}{2(1+\delta)}}v_q\|_{L^\infty_t(L^\infty)}
\lesssim
&\|v_q\|_{L^\infty_t(L^{\frac{2}{\delta}})}^{\frac 1{1+\delta}}
\|\sigma^{\frac 12}\na^2 v_q\|_{L^\infty_t(L^2)}^{\frac{\delta}{1+\delta}}
\leq \cA_0 d_q2^{q(1-s_1)}.
\end{split} \eeq

Taking $\de\in ]0, s_1/(1-s_1)[,$ we deduce from  \eqref{S4eq17}, \eqref{S4eq17b} and Corollary \ref{S4col1} that
\beq\label{S4eq18} \begin{split}
\|\p_t(v_q\cdot\na
X)\|_{L^2_t(L^2)}\leq &
\|\p_tv_q\|_{L^2_t(L^2)}\|\nabla X\|_{L^\infty_t(L^\infty)}\\
&+\|\sigma^{\frac{\delta}{2(1+\delta)}}v_q\|_{L^\infty_t(L^\infty)}
\|\sigma^{\frac{1-s_1}{2}}\nabla X_t\|_{L^\infty_t(L^2)}
\|\sigma^{-\f12\left(\frac{1}{1+\d}-s_1\right)}\|_{L^2_t}\\
\leq &
 \cH_1(t)d_q2^{q(1-s_1)}.
\end{split}
\eeq

While due to $\dive X=0,$ we have
\beno \na \dive(v_q\cdot\na X)
=\na\p_\al v_q\cdot\na X^\al+\p_\al v_q\cdot\na \na
X^\al,
\eeno
from which, \eqref{brho} and  \eqref{S4eq1},  we infer
\begin{align*}
\|(F_{1q}, \na \dive(v_q\cdot\na X))\|_{L^2_t(L^2)}
\lesssim &
(1+\|\nabla X\|_{L^\infty_t(W^{1,p})})\\
&\times \bigl(\|(D_t v_q, \nabla^2 v_q, \nabla\pi_q)\|_{L^2_t(L^2)}
+t^{\left(\f{1}2-\f1p\right)}\|\nabla v_q\|_{L^p_t(L^{\frac{2p}{p-2}})}\bigr).
\end{align*}
As a result, by virtue of  \eqref{S4eq17} and \eqref{S4eq20} (for $\delta=(p-2)/p$), it comes out
\beq \label{S4eq21}
\|(F_{1q}, \na \dive(v_q\cdot\na X))\|_{L^2_t(L^2)}
\leq  \cH_1(t)d_q2^{q(1-s_1)}.
\eeq

Substituting the Inequalities  \eqref{S4eq18} and \eqref{S4eq21} into \eqref{S4eq8} gives rise to \eqref{S4eq6}.  This finishes the proof of Lemma
\ref{S4prop2}.
\end{proof}

The main result of this subsection is as follows

\begin{prop}\label{S4prop3}
{\sl Under the assumptions of Proposition \ref{prop:Dv1}, we have
\beq \label{S4eq22} \|{\v}^1\|_{\wt{L}^\infty_t(\cB^{s_1})}+\|\na
\v^1\|_{\wt{L}^2_t(\cB^{s_1})}\leq
\cH_1(t). \eeq}
\end{prop}

\begin{proof} Indeed in view of \eqref{S4eq2}, Lemmas
\ref{S4prop1} and \ref{S4prop2}, we deduce by a similar proof of
Proposition \ref{S2prop1} that \beno
\|\v_{12}\|_{\wt{L}^\infty_t(\cB^{s_1})}+\|\na
\v_{12}\|_{\wt{L}^2_t(\cB^{s_1})}\leq \cH_1(t), \eeno
  which together
with \eqref{S3eq3} and \eqref{S3eq6} leads to \eqref{S4eq22}.
\end{proof}

Combing Proposition \ref{S4prop1q} with Proposition \ref{S4prop3}, we achieve Proposition \ref{prop:Dv1}.

%%%%%%%%%%%%%%%%%%%%%%%%%%%%%%%%%%%%%%%%%%%%%%%%%%%%%%%%
%%%%%%%%%%%%%%%%%%%%%%%%%%%%%%%%%%%%%%%%%%%%%%%%%%%%%%%%%%%
\setcounter{equation}{0}
\section{Propagation of higher order striated regularity}\label{Sect6}

The goal of this section is  to present the proof of Proposition \ref{prop:Jell}. To this end,
for the functional $A_\ell(t)$  given by \eqref{estimate:Jell},
 we  inductively assume that
\begin{equation}\label{bound:Jell-1}
A_{l}(t)\leq \cH_{l}(t), \quad\mbox{for any}\  l\leq \ell-1\andf \ell\leq k.
\end{equation}
We shall always assume \eqref{bound:Jell-1} throughout this section. We aim at establishing the Estimate \eqref{bound:Jell-1} for $l=\ell$.

\subsection{Deductive estimates from \eqref{bound:Jell-1}}%\label{subsec:ell1}
%%%%%%%%%%%%%%%%%%%%%%%%%%%%%%%%%%%%%%%%%%%%%%%%%%%%%%%%%%%%%%
In this subsection we shall derive some   estimates from the inductive assumption \eqref{bound:Jell-1}, which will be used constantly in the following context.
For $\ell\geq 1,$ let
\begin{equation}\label{Rell}
\begin{split}
R_{\ell}(t)
\eqdefa &
\sum_{  m+n+\kappa\leq\ell-1}
 \bigl(\|\d_X^m\nabla \rx^n\|_{L^\infty_t(W^{1,p} )}
+\|\d_X^m\nabla\d_X^n \nabla\rx^\kappa\|_{L^\infty_t(L^p )}\bigr).
\end{split}
\end{equation}
Let
\begin{equation}\label{r}
\begin{split}
  &r_1\in \left\{2, \,{2p}/{(p-2)},\,+\infty\right\},
  \quad
  r\in \{2,\,p \},
  \\
&   r_2\in\left\{2p/(p-2), \,+\infty\right\},
\andf r_3\in \left\{2,\,2p/(p-2) \right\},
  \end{split}
  \end{equation}
and
\begin{equation}\label{epsilon0}
0<\e_0<\min\left(\,{s_k}/2,
%\,1- 2/p,
\, \bigl({p}/2-1\bigr)(1-s_0)\,\right),
\end{equation}
 we denote
\begin{equation}\label{Vell}
\begin{split}
\dot{\frak{A}}_{\ell}(t)
 \eqdefa
\mathop{ \sum_{ i+j+l=\ell}}\limits_{ m+n=\ell-1}
\Bigl(& \|\v^i\|_{L^\infty_t(L^2)}
+\| \sigma ^{\frac{1-s_{\ell}}{2 }}
 \v^i  \|_{L^\infty_t(  L^{\frac{2p}{p-2}})}
+\| \sigma ^{ \frac{1+\varepsilon_0-s_{\ell}}{2 }}
( \v^i, \d_X^m \d_t X) \|_{L^\infty_t(L^\infty )}
\\
&+ \| \sigma ^{ \frac{1-s_{\ell}}{2} }
\d_X^m  \nabla X_t \|_{L^\infty_t(L^2)}
+\bigl\|\sigma ^{ 1-\frac{s_{\ell}}{2}}
(  \d_X^i \d_t \v^j,
   \d_X^m D_t  \nabla ^2 \rx^n)\|_{L^\infty_t(L^2)}
\\
&
+ \| \sigma ^{ \left(1-\frac{1}{r_1}-\frac{s_{\ell}}{2}\right) }
  \d_X^i\nabla \v^j  \|_{L^\infty_t(L^{r_1})}
+ \| \sigma ^{\left(1-\frac{1}{r_2}-\frac{s_{\ell}}{2}\right)}
\d_X^m D_t \nabla \rx^n\|_{L^\infty_t(L^{r_2})}
\\
&+ \bigl\|\sigma ^{ \left(\frac 32-\frac 1r-\frac{s_{\ell}}{2}\right) }
(\d_X^i \nabla\d_X^j \nabla \v^l,
  D_t \v^i, \d_X^i \nabla{\rm \pi}^j)\|_{L^\infty_t(L^r)} \Bigr);
\end{split}
\end{equation}
and
\begin{equation}\label{cVell}
\begin{split}
\dot{\frak{B}}_{\ell}(t)
\eqdefa
\sum_{ i+j= \ell}
  \Bigl(&\|\sigma^{\left(\frac 32-\frac{1}{r_3}-\frac{s_{\ell}}{2}\right)}
  ( \d_X^i\nabla D_t \v^j, \d_X^i D_t\nabla \v^j)(t)\|_{L^{r_3}}^2
+\|\sigma^{ \left(\frac 12+\frac 1p-\frac{s_{\ell}}{2}\right) }
   D_t \v^i(t)\|_{L^{\frac{2p}{p-2}}}^2
  \\
  &\quad
  + \bigl\|\sigma^{\frac{3-s_{\ell}}{2}}
 ( D_t^2\v^i, \d_X^i D_t \nabla^2\v^j,
\nabla\d_X^i \nabla D_t\v^j, \d_X^i D_t \nabla{{\rm \pi}}^j)(t)\bigr\|_{L^2}^2 \Bigr);
 \end{split}
\end{equation}
and
\beq\label{eq:AB}
\frak{A}_\ell(t)\eqdefa
 \sum_{ l\leq \ell}\dot{\frak{A}}_{l}(t)
\andf \frak{B}_\ell(t)\eqdefa
\frak{B}_0(t)+\sum_{  l\leq \ell}\dot{\frak{B}}_{l}(t), \eeq
where   $\frak{B}_0(t)$ is given by  \eqref{S1eq25}.

\smallbreak

Under the inductive assumption \eqref{bound:Jell-1},
the  estimates of $R_{\ell-1}(t), \frak{A}_{\ell-1}(t), \frak{B}_{\ell-1}(t)$ relies on the   following lemma concerning the  commutative  estimates,
 the proof of which will be postponed in the appendix.

%%%%%% Commutator X%%%%%%
\begin{lem}\label{lem:comm,X}
{\sl Let $\ell\in \{1,\cdots, k\}$ and $(i,j)$ be   any pair of nonnegative integers  with
$i+j\leq \ell.$ Then for $r_1, r, r_3$  satisfying \eqref{r}, there exists a positive constant  $C$  such that
\begin{equation}\label{comm:X}
\begin{array}{l}
\|\d_X^i\nabla\rx^{\ell-i}
   -\nabla\rx^\ell\|_{L^\infty_t(W^{1,p})}
 +\|\d_X^i\nabla\d_X^j\nabla\rx^{\ell-i-j}
   -\nabla^2\rx^{\ell}\|_{L^\infty_t(L^p)}
   \leq
C R_{\ell}^{2}(t) ,
\end{array}\end{equation}
and
\begin{equation}\label{comm:v}
\begin{split}
& \bigl\|\sigma ^{\left(1-\frac{1}{r_1}-\frac{s_{\ell-1}}{2}\right)}
(\d_X^i \nabla\v^{\ell-i} -\nabla \v^\ell) \bigr\|_{L^\infty_t(L^{r_1})}
+ \bigl\|\sigma^{\left(\frac 32-\frac 1r-\frac{s_{\ell-1}}{2}\right)}
 ( \d_X^i \nabla{\rm \pi}^{\ell-i}  -\nabla{\rm \pi}^\ell) \bigr\|_{L^\infty_t(L^r)}
\\
&\qquad\quad
  + \bigl\|\sigma ^{\left(\frac 32-\frac 1r-\frac{s_{\ell-1}}{2}\right)}
 ( \d_X^i\nabla\d_X^j\nabla \v^{\ell-i-j}
  -\nabla^2 \v^\ell) \bigr\|_{L^\infty_t(L^r)}
   \leq
C  R_\ell(t) \frak{A}_{\ell-1}(t),
\end{split}
\end{equation}
and
\begin{equation}\label{comm:v''}
\begin{split}
 & \bigl\|\sigma(t)^{\left(\frac 32-\frac{1}{r_3}- \frac{s_{\ell-1}}2 \right)}\bigl(\d_X^i \nabla D_t \v^{\ell-i}
-\nabla D_t \v^\ell,\,
\d_X^i D_t\nabla \v^{\ell-i}- \nabla D_t \v^\ell\bigr)(t) \bigr\|_{L^{r_3} }^2
\\
&
+\bigl\| \sigma(t)^{ \frac{3-s_{\ell-1}}2 }
(\d_X^i D_t\nabla^2\v^{\ell-i} - \nabla^2 D_t \v^\ell,\,
   \nabla\d_X^i \nabla D_t \v^{\ell-i}
  -\nabla^2 D_t \v^\ell)(t) \bigr\|_{L^2}^2
    \\&
+   \bigl\|\sigma(t)^{ \frac{3-s_{\ell-1}}2 }
 (\d_X^i D_t\nabla{\rm{\pi}}^{\ell-i}- \nabla D_t{\rm{\pi}}^\ell)(t)  \bigr\|_{L^2}^2
   \leq
C R_{\ell}^2(t)\frak{B}_{\ell-1}(t)
+C\frak{A}_{\ell}^4(t)\s^{-(1-s_\ell)}.
\end{split}
\end{equation}
 }
\end{lem}

\begin{lem}\label{lem:R,V}
{\sl Under the assumptions of  Proposition \ref{prop:Jell} and the inductive assumption \eqref{bound:Jell-1},
one has
\begin{equation}\label{bound:R,V}
\begin{array}{l}
R_{l}(t)+\frak{A}_{l}(t)
+\int_0^t\frak{B}_{l}(t')\,dt'
\leq \cH_{l}(t),\quad \mbox{for}\quad 1\leq l\leq \ell-1.
\end{array}\end{equation}}
\end{lem}
%%%%%%%%
\begin{proof}
We first deduce from    Corollaries \ref{S1col1} and \ref{S2col3},
Proposition \ref{prop:Dv1},
Corollaries \ref{S4col1} and \ref{S4col2}  that \eqref{bound:R,V} holds for
 $l=1$.
 Inductively, we assume that
 \beq \label{S6eq1q}
 R_{\kappa}(t)+\frak{A}_{\kappa}(t)
+\int_0^t\frak{B}_{\kappa}(t')\,dt'
\leq \cH_{\kappa}(t),\quad 2\leq \kappa\leq l-1.
\eeq
We intend to show the Estimate \eqref{S6eq1q} for $\kappa+1.$
Indeed it follows from  \eqref{bound:Jell-1} that
$$
\|\nabla\rx^{\beta}\|_{L^\infty_t(W^{1,p})}  \leq \cH_{\beta+1}(t)\quad \forall\ \beta\leq l-1,
$$
which together with the commutator estimate \eqref{comm:X} for  $\ell=\kappa$ and \eqref{S6eq1q} ensures that
\begin{equation}\label{RV:R}
R_{\kappa+1}(t)\leq \cH_{\kappa+1}(t).
\end{equation}
While we deduce from \eqref{S2eq6} and \eqref{bound:Jell-1}   that for any $ \beta\leq l$ and $ r\in [2,\infty[,$
\begin{equation}\label{RV:J}
\begin{split}
&\|\v^{\beta}\|_{L^\infty_t(L^2)\cap L^2_t(\dot H^1)}
+ \|\sigma ^{ \frac{1-s_{\beta}}{2}}
\nabla \v^{\beta} \|_{L^\infty_t(L^2)}
+\|\sigma^{\frac{1-s_\beta}{2}} \d_t \v^\beta\|_{L^2_t(L^2)}
\\
&+\bigl\|\sigma ^{1-\frac{s_{\beta}}{2}}
(D_t \v^{\beta}, \nabla^2 \v^{\beta}, \nabla{\rm{\pi}}^{\beta}) \bigr\|_{L^\infty_t(L^2)}
 + \|\sigma ^{ \frac{3-s_{\beta}}{2}}
\nabla D_t \v^{\beta} \|_{L^\infty_t(L^2)}
\\
&
+\|\sigma^{\frac{1-s_\beta}{2}}\nabla D_t \v^\beta\|_{L^2_t(L^2)}
+ \|\sigma ^{ \frac{3-s_{\beta}}{2}}
(D_t^2 \v^\beta, \nabla^2 D_t \v^{\beta}, \nabla D_t\pi^\beta) \|_{L^2_t(L^2)}
\\
& +\bigl\|\sigma^{ \left(\frac 32-\frac 1{r}-\frac{s_{\beta}}{2}\right) }
D_t \v^{\beta} \|_{L^\infty_t(L^r)}
 +\bigl\|\sigma^{\left( \frac 32-\frac 1{r}-\frac{s_{\beta}}{2}\right) }
\nabla D_t \v^{\beta} \|_{L^2_t(L^r)}\leq \cH_{\beta}(t),
\end{split}
\end{equation}
from which, the second inequality of  \eqref{est:Fell,p} below and
%\eqref{S6eq1q},
\eqref{RV:R}, we infer
\begin{equation}\label{RV:v''}
\bigl\| \sigma ^{\left(\frac 32-\frac 1r-\frac{s_{\kappa+1}}{2}\right)}
(D_t \v^{\kappa+1}, \nabla^2 \v^{\kappa+1}, \nabla{\rm{\pi}}^{\kappa+1})
\bigr\|_{L^\infty_t(L^r)}\leq \cH_{\kappa+1}(t),
\quad r\in \{2,p\}.
\end{equation}
Furthermore,  by Corollary \ref{S1col1} and \eqref{RV:J} one has
$$
\|\sigma^{\frac{1-s_\beta}{2}} D_t\v^\beta\|_{L^2_t(L^2)}
\leq
\|\sigma^{\frac{1-s_\beta}{2}} \d_t\v^\beta\|_{L^2_t(L^2)}
+\|\sigma^{\frac{1-s_\beta}{2}} v\|_{L^\infty_t(L^\infty)}
\|  \nabla\v^\beta\|_{L^2_t(L^2)}
\leq\cH_\beta(t),
\,\forall\beta\leq l,
$$
from which, \eqref{RV:J} and \eqref{S2eq6} we deduce
\begin{equation}\label{RV:Dv}
\begin{split}
\|\sigma^{\left(\frac 12+\frac 1p-\frac{s_{\kappa+1}}{2}\right)} D_t\v^{\kappa+1}\|_{L^2_t(L^{\frac{2p}{p-2}})}
\leq\cH_{\kappa+1}(t) .
\end{split}
\end{equation}

On the other side, in view of \eqref{RV:J} and \eqref{RV:v''},  we get, by applying \eqref{S2eq6}
and
\begin{equation*}
\|\sigma^{\left(1-\f{s}{2}\right)}
 a\|_{L^\infty}
\leq
C\|\sigma^{\frac{1-s}{2}}a\|_{L^2}^{\frac{1-2/p}{2(1-1/p)}}
\|\sigma^{\left(\frac 32-\f1p-\f{s}{2}\right)}
\nabla a\|_{L^p}^{\frac{1}{2(1-1/p)}},
\end{equation*}
that
\begin{equation}\label{RV:v'}
 \bigl\|\sigma ^{\left(1-\frac{1}{r_1}-\frac{s_{\kappa+1}}{2}\right)}
\nabla \v^{\kappa+1} \|_{L^\infty_t(L^{r_1})}\leq \cH_{\kappa+1}(t),
\quad \forall r_1\in [2,+\infty].
\end{equation}
Noticing that for any $q\in ]2,\infty[,$
\begin{equation*}
\|\sigma^{\left(\f12-\f{s}{4(1-1/q)}\right)}
 a\|_{L^\infty}
\leq
C\|a\|_{L^2}^{\frac{1-2/q}{2(1-1/q)}}
\|\sigma^{\left(1-\f1q-\f{s}{2}\right)}
\nabla a\|_{L^q}^{\frac{1}{2(1-1/q)}},
\end{equation*}
and \eqref{S2eq6}
$$
\|\sigma^{\frac{1 -s_\beta}{2}}
 a\|_{L^{\frac{2p}{p-2}}}
\leq
C\|a\|_{L^2}^{ 1-\frac 2p }
\|\sigma^{\frac{1 -s_\beta}{2}}
 \nabla a\|_{ L^2}^{\frac 2p},
$$
we deduce  from \eqref{RV:J} and  \eqref{RV:v'} (by taking $q=\f{2(s_{\kappa+1}-\e_0)}{s_{\kappa+1}-2\e_0}$)
 that
 % for any  $\varepsilon_0$ satisfying \eqref{epsilon0}
\begin{equation}\label{RV:v}
 \|\sigma^{ \frac{1+\varepsilon_0-s_{\kappa+1}}{2}}
  \v^{\kappa+1} \|_{L^\infty_t(L^\infty )}
  + \|\sigma^{ \frac{1 -s_{\kappa+1}}{2}}
  \v^{\kappa+1} \|_{L^\infty_t( L^{\frac{2p}{p-2}})}
 \leq \cH_{\kappa+1}(t).
\end{equation}

To deal with the estimates pertaining to $X$ in $\dot{\frak{A}}_{\kappa+1}(t),$ we first get from \eqref{eq:X} that
\begin{equation}\label{S6eqX'}
\d_X^\kappa\d_t X
=\d_X^\kappa(\v^{1}-v\cdot\nabla X)
=\v^{\kappa+1}-\sum_{n=0}^\kappa C_\kappa^n\v^n\cdot\d_X^{\kappa-n}\nabla X,
\end{equation}
and
\begin{equation}\label{S6eqX''}
\d_X^\kappa\nabla X_t
=\d_X^\kappa\nabla(\v^1-v\cdot\nabla X)
=\d_X^\kappa\nabla\v^1
-\sum_{n=0}^\kappa C_\kappa^n \bigl(\d_X^n\nabla v\cdot\d_X^{\kappa-n}\nabla X
+\v^n\cdot\d_X^{\kappa-n}\nabla^2 X\bigr).
\end{equation}
We then
apply the operator $\d_X^{\kappa-1}$ to  \eqref{eq:X} to get
\begin{equation}\label{S9eq10}
D_t\rx^{\kappa-1}
=\d_t\rx^{\kappa-1} +v\cdot\nabla\rx^{\kappa-1} =\v^\kappa,
\end{equation}
which together with \eqref{S3eq0} ensures that for any nonnegative integer $m\leq\kappa$,
\begin{equation}\label{S6eqDX'}
\begin{split}
\d_X^m D_t \nabla \rx^{\kappa-m}
&=\d_X^m\nabla D_t\rx^{\kappa-m}
-\d_X^m(\nabla v\cdot\nabla\rx^{\kappa-m})
\\
&=\d_X^m\nabla\v^{\kappa+1-m}
-\sum_{n=0}^m C_m^n \d_X^n \nabla v_\alpha\,\d_X^{m-n}\d_\alpha\rx^{\kappa-m},
\end{split}
\end{equation}
and
\begin{equation}\label{S6eqDX''}
\begin{split}
&\d_X^m D_t \nabla^2 \rx^{\kappa-m}
=\d_X^m\nabla^2 D_t\rx^{\kappa-m}
-\d_X^m(\nabla^2 v_\alpha \d_\alpha\rx^{\kappa-m}
+2\nabla v_\alpha\d_\alpha\nabla\rx^{\kappa-m})
\\
&\quad
=\d_X^m\nabla^2 \v^{\kappa+1-m}
-\sum_{n=0}^m C_m^n
\left( \d_X^n\nabla^2 v_\alpha\d_X^{m-n}\d_\alpha\rx^{\kappa-m}
+2\d_X^n\nabla v_\alpha\d_X^{m-n}\d_\alpha\nabla\rx^{\kappa-m}\right).
\end{split}
\end{equation}
By virtue of \eqref{Rell}, we deduce  from \eqref{S6eqX'}-\eqref{S6eqX''} that
\begin{align*}
&\|\sigma^{\frac{1+\e_0-s_{\kappa+1}}{2}}\d_X^\kappa X_t\|_{L^\infty_t(L^\infty)}
+\|\sigma^{\frac{1 -s_{\kappa+1}}{2}} \d_X^\kappa\nabla X_t\|_{L^\infty_t(L^2)}
\leq C(1+R_{\kappa+1}(t))
\\
& \quad\times\sum_{i+j\leq \kappa+1}
\Bigl(\|\sigma^{\frac{1+\e_0-s_{\kappa+1}}{2}}\v^i\|_{L^\infty_t(L^\infty)}
+\|\sigma^{\frac{1 -s_{\kappa+1}}{2}} \d_X^i\nabla \v^j\|_{L^\infty_t(L^2)}
+\|\sigma^{\frac{1 -s_{\kappa+1}}{2}} \v^i\|_{L^\infty_t(L^{\frac{2p}{p-2}})}\Bigr),
\end{align*}
and it follows from \eqref{S6eqDX'}-\eqref{S6eqDX''} that for $r_2=\frac{2p}{p-2}$ or $r_2=\infty$ and for any $m\leq\kappa$,
\begin{align*}
&\|\sigma^{1-\frac{1}{r_2}-\frac{s_{\kappa+1}}{2}}
\d_X^m D_t\nabla \rx^{\kappa-m}\|_{L^\infty_t(L^{r_2})}
+\|\sigma^{1-\frac{s_{\kappa+1}}{2}}
\d_X^m D_t\nabla^2 \rx^{\kappa-m}\|_{L^\infty_t(L^2)}
\\
& \quad
\leq C(1+R_{\kappa+1}(t)) \sum_{i+j\leq \kappa+1}
\Bigl(\|\sigma^{1-\frac{1}{r_2}-\frac{s_{\kappa+1}}{2}}
\d_X^i  \nabla \v^{j}\|_{L^\infty_t(L^{r_2})}
\\
&\qquad\quad\qquad\quad
+\|\sigma^{1-\frac{s_{\kappa+1}}{2}} \d_X^i\nabla^2 \v^j\|_{L^\infty_t(L^2)}
+\|\sigma^{1-\frac{ s_{\kappa+1}}{2}} \d_X^i\nabla v\|_{L^\infty_t(L^{\frac{2p}{p-2}})}\Bigr).
\end{align*}
The Estimates   \eqref{RV:R} to \eqref{RV:v}, along with \eqref{bound:Jell-1},  \eqref{comm:v}, \eqref{S6eq1q},
the above two inequalities and the following fact
  for any nonnegative integer $i\leq\kappa+1$,
\begin{equation*} \begin{split}
& \|\sigma ^{ 1-\frac{s_{\kappa+1}}{2} }
\d_X^i \d_t \v^{\kappa+1-i} \|_{L^\infty_t(L^2)}
=
 \|\sigma ^{ 1-\frac{s_{\kappa+1}}{2} }
(\d_X^i D_t \v^{\kappa+1-i}
-\d_X^i (v\cdot\nabla \v^{\kappa+1-i})) \|_{L^\infty_t(L^2)}
\\
&\leq
\|\sigma ^{ 1-\frac{s_{\kappa+1}}{2} }
D_t \v^{\kappa+1} \|_{L^\infty_t(L^2)}
+C\sum_{n=0}^i \|\sigma^{\frac 12}\v^n\|_{L^\infty_t(L^\infty)}
\|\sigma^{\frac{1-s_{\kappa+1}}{2}}\d_X^{i-n}\nabla\v^{\kappa+1-i}\|_{L^\infty_t(L^2)},
\end{split}\end{equation*}
 implies that
\begin{equation}\label{RV:Vell}
\begin{split}
&\dot{\frak{A}}_{\kappa+1}(t) \leq \cH_{\kappa+1}(t).
\end{split}
\end{equation}

Finally under the assumption \eqref{bound:Jell-1},  we deduce from the
commutative estimates \eqref{comm:v''}
and  \eqref{S6eq1q}, \eqref{RV:R},  \eqref{RV:J}, \eqref{RV:Dv}, \eqref{RV:Vell}, that
 \beq\label{RV:B}
 \int_0^t\dot{\frak{B}}_{\kappa+1}(t')\,dt'\leq \cH_{\kappa+1}(t). \eeq
Whence in view of  \eqref{eq:AB},  by summing up \eqref{S6eq1q}, \eqref{RV:R},  \eqref{RV:Vell}
 and \eqref{RV:B}, we conclude that  \eqref{S6eq1q}
 is valid for $\kappa+1$.
Hence Lemma \ref{lem:R,V} follows.
\end{proof}

%%%%%%%%%%%%%%%%%%
\subsection{Some preliminary estimates}
\label{sec:prepa,vell}
%%%%%%%%%%%%%%
According to Lemma  \ref{S3lem2}   and  \eqref{eq:vell},
the time-weighted $H^1$ energy estimates of $\v^\ell$ relies on  the estimates of  $F_\ell(v,\pi)$, $\dive\v^{\ell}$ and $\nabla \div \v^\ell,$
which is the goal of this subsection.

Let us first calculate $\dive \v^\ell.$
Notice that $\div \v^1=\div(v\cdot\nabla X)$.
Suppose inductively that
$
\div \v^{\ell-1}=\div g_{\ell-1},
$
then due to $\div X=0$ we have
\beq\label{S1eq6}
\begin{split}
\div \v^{\ell}
=\div(\d_X \v^{\ell-1})
&=\div\left(\v^{\ell-1}\cdot\nabla X
+ X\div g_{\ell-1}\right)
\\
&=\div\left(\v^{\ell-1}\cdot\nabla X
+\d_Xg_{\ell-1}
-g_{\ell-1}\cdot\nabla X\right)\eqdefa \dive g_{\ell}.
\end{split}\eeq
We denote
\beq\label{S1eq6a}
\cE_X f\eqdefa \p_Xf-f\cdot\na X.\eeq
It is easy to observe that
\begin{align*}
\div \cE_X f
=\d_X\div f.
\end{align*}
We thus get, by using \eqref{S1eq6} and an inductive argument, that
\begin{equation}\label{tildevell}
\begin{split}
g_{\ell}
&=\sum_{i=0}^{\ell-1} \cE_X^{i}(\v^{\ell-1-i}\cdot\nabla X),
\end{split}
\end{equation}
and
\beq\label{S1eq8}
\begin{split}
\div \v^\ell
=\sum_{i=0}^{\ell-1}\div \cE_X^i(\v^{\ell-1-i}\cdot\nabla X)
&=\sum_{i=0}^{\ell-1}\d_X^i \div(\v^{\ell-1-i}\cdot\nabla X)
\\
&=\sum_{i=0}^{\ell-1}\d_X^i (\p_\al \v^{\ell-1-i}\cdot\nabla X^\al).
\end{split} \eeq

\begin{lem}\label{lem:Fell}
{\sl For $\ell=2,\cdots,k$, and $r\in\{2,p\},$   there hold
\begin{equation}\label{est:Fell,p}
\begin{split}
 & \|\sigma ^{\left(\frac 32-\frac 1r-\frac{s_{\ell-1}}{2}\right)}
 (F_\ell(v,\pi),\nabla \div \v^{\ell})(t) \|_{L^r}
 \leq
\cH_{\ell-1}(t)(1+\|\na\rx^{\ell-1}\|_{W^{1,p}})\andf\\
&\bigl\|\sigma ^{\left(\frac 32-\frac 1r-\frac{s_{\ell}}{2}\right)}
( \nabla^2 \v^\ell, \nabla{\rm{\pi}}^\ell)(t)\bigr\|_{L^r}
\leq
C\|\sigma ^{\left(\frac 32-\frac 1r-\frac{s_{\ell}}{2}\right)}
 D_t \v^\ell(t)\|_{L^r}\\
&\qquad\qquad\qquad\qquad\qquad\qquad\qquad\qquad+\cH_{\ell-1}(t)\bigl(1
+\|\na\rx^{\ell-1}(t)\|_{W^{1,p}}\bigr).
\end{split}
\end{equation} }
\end{lem}
\begin{proof} For $r=2$ or $p,$ we  deduce from \eqref{S1eq5} that
\begin{align*}
\|F_\ell(v,\pi)\|_{L^r }
&\leq
C\Bigl(R_{\ell-1}(t)+\sum_{j=1}^\ell\|\rr^j\|_{L^\infty }
+\|\d_X^{\ell-1}\nabla X\|_{L^\infty }
+\|\d_X^{\ell-1}\Delta X\|_{L^p } \Bigr)
\\
& \qquad
\times\sum_{ i+j\leq \ell-1}
\Bigl( \|(D_t \v^j, \d_X^i\nabla^2 \v^j, \d_X^i\nabla{\rm{\pi}}^j)\|_{L^r }
+\|\d_X^i \nabla \v^j\|_{L^{\frac{pr}{p-r}}} \Bigr),
\end{align*}
which together with  \eqref{Vell} and \eqref{comm:X}, ensures that for $r=2$ or $r=p$
\beno
\begin{split}
 \| F_{\ell }(v,\pi) \|_{L^r}
&\leq
C\Bigl(R_{\ell-1}(t)+R_{\ell-1}^2(t)
+\sum_{j=1}^\ell\|\rr^\ell\|_{L^\infty }
 +\|\nabla\rx^{\ell-1}\|_{W^{1,p} } \Bigr)
 \sigma^{-\left(\f32-\f1r-\frac{s_{\ell-1}}{2}\right)}
 \frak{A}_{\ell-1}(t).
\end{split}
\eeno
As a result, we deduce from Lemma \ref{lem:R,V} and \eqref{brho} that for $r=2$ or $r=p$
\beq\label{S1eq8af}
\begin{split}
& \|  F_\ell(v,\pi) \|_{L^r}
 \leq
\cH_{\ell-1}(t)\sigma(t)^{-\left(\frac 32-\frac 1r-\frac{s_{\ell-1}}{2}\right)}
\bigl( 1 +\|\na\rx^{\ell-1}\|_{W^{1,p}}\bigr).
\end{split}
\eeq

While it follows from \eqref{S1eq8} that
\begin{align*}\label{divvell}
\nabla\div \v^\ell
&=\sum_{i=0}^{\ell-1}\sum_{j=0}^i C_i^j
\nabla\Bigl(\d_X^j  \p_\al \v^{\ell-1-i} \cdot\d_X^{i-j}\nabla X^\al\Bigr),
\end{align*}
from which, we infer for   $r=2$ or $r=p$
\begin{align*}
\|\nabla\div \v^\ell\|_{L^r}
\leq &
C\bigl( R_{\ell-1} (t)+\|\d_X^{\ell-1}\nabla X\|_{L^\infty}
+\| \nabla\d_X^{\ell-1}\nabla X\|_{L^p}\bigr)
\\
&\qquad\times
\sum_{  i+j\leq\ell-1}
\Bigl(\|\nabla\d_X^i\nabla \v^j\|_{L^r}
+\|\d_X^i\nabla \v^j\|_{L^{\f{pr}{p-r}}}\Bigr) .
\end{align*}
Then we deduce  from  \eqref{Vell} and the commutative estimate \eqref{comm:X} that for $r=2,p,$
 \begin{align*}
 \|\nabla\div v_\ell\|_{L^r}
\leq
C\bigl(R_{\ell-1}(t)+R_{\ell-1}^2(t)+\|\na\rx^{\ell-1}\|_{W^{1,p}}\bigr)
\sigma^{-\left(\frac 32-\frac 1r-\frac{s_{\ell-1}}{2}\right)}\frak{A}_{\ell-1}(t).
 \end{align*}
By virtue of Lemma \ref{lem:R,V} and \eqref{S1eq8af}, we obtain the first  inequality of  \eqref{est:Fell,p}.

On the other hand, in view of \eqref{eq:vell}, we write
\begin{align*}
&-\Delta (\v^\ell-\nabla\Delta^{-1}\div   \v^{\ell} )
+\nabla{\rm{\pi}}^\ell
 =\nabla\div  \v^{\ell}
-\rho D_t \v^\ell
+F_\ell(v,\pi ),
\end{align*}
so that for any $r\in ]1,\infty[$, it follows from classical estimates for   Stokes operator  that  \begin{equation}\label{Stokes:vell}
\begin{split}
 \| ( \nabla^2 \v^\ell, \nabla{\rm{\pi}}^\ell ) \|_{L^r }
\leq
C\bigl(\|\nabla\div   \v^{\ell}\|_{L^r }
+\|\rho D_t \v^\ell \|_{L^r }
+\|F_\ell(v,\pi )\|_{L^r }\bigr),
\end{split}
\end{equation}
which together with the first inequality of \eqref{est:Fell,p} gives rise to the second one of \eqref{est:Fell,p}.
This completes the proof of the lemma.
\end{proof}

\begin{lem}\label{lem:vell}
{\sl Let $g_{\ell}$ be given by \eqref{tildevell}. Then one has
\beq\label{S6eq1}
 \begin{split}
 &\|g_{\ell}(t)\|_{L^2}+\sigma(t)^{\frac{1-s_{\ell-1}}{2}}\|\nabla g_{\ell}(t)\|_{L^2 }
\leq \cH_{\ell-1}(t)\bigl(1+\|\na\rx^{\ell-1}(t)\|_{W^{1,p}}\bigr),
\\
&\|\d_t g_{\ell}(t)\|_{L^2 }
 \leq
\|v\|_{L^\infty}\|\nabla \v^\ell\|_{L^2}
+\cH_{\ell-1}(t)\sigma(t)^{-\left(1-\frac{s_{\ell-1}}{2}\right)}\bigl(1+\|\na\rx^{\ell-1}(t)\|_{W^{1,p}}\bigr).
 \end{split}\eeq
 }
\end{lem}
\begin{proof} It follows from  \eqref{tildevell} that
\beq\label{tildevell1}\begin{split}
g_{\ell}
&=\sum_{i=0}^{\ell-1}\bigl (\d_X-(\nabla X)^T\cdot\bigr) ^i (\v^{\ell-1-i}\cdot\nabla X)
\\
&=\sum_{i=0}^{\ell-1}\sum_{i_1+\cdots+i_r=i}
\d_X^{i_1}(\v^{\ell-1-i}\cdot\nabla X)
\cdot\d_X^{i_2}  (-\nabla X) ^{i_3}
\cdots \d_X^{i_{r-1}} (-\nabla X)^{i_r}.
\end{split}\eeq
It is obvious to observe from \eqref{tildevell1}, Lemmas \ref{lem:comm,X} and \ref{lem:R,V}  that
\beq\label{S6eq7}
\begin{split}
\|g_{\ell}\|_{L^2}
\leq &
\|v\cdot\d_X^{\ell-1}\nabla X\|_{L^2}
+\sum_{l\leq\ell-2, i\leq\ell-1}
\|\v^i\|_{L^2}\|\d_X^l\nabla X\|_{L^\infty}^{\ell-1}
\\
\leq &
\Bigl(\sum_{i\leq\ell-1}\|\v^i\|_{L^2} \Bigr)
\bigl(\|\d_X^{\ell-1}\nabla X\|_{L^\infty}+R^{\ell-1}_{\ell-1}(t)\bigr)\\
\leq&
\cH_{\ell-1}(t)(1+\|\na\rx^{\ell-1}\|_{W^{1,p}}).
\end{split}\eeq

While by taking $\d_{\kappa}$ with $\kappa=0,1,2$ (here $\p_0\eqdefa \p_t$) to \eqref{tildevell1} and using \eqref{S3eq0}, we write
\begin{equation}\label{vell'}
\begin{split}
\d_{\kappa} g_{\ell}
=g_{\ell,\kappa}^1+g_{\ell,\kappa}^2,
\end{split}
\end{equation}
where
\begin{align*}
g_{\ell,\kappa}^1
\eqdefa & \sum_{i=1}^{\ell-1}\sum_{i_1+\cdots+i_r=i-1}\Bigl(
 \d_X^{i_1}\bigl(\d_\kappa X\cdot\nabla\d_X^{i_2}(\v^{\ell-1-i}\cdot\nabla X)\bigr)
 \cdots\d_X^{i_{r-1}} (-\nabla X)^{i_r}
\\
&
+\d_X^{i_1}(\v^{\ell-1-i}\cdot\nabla X)\cdots
 \d_X^{i_{j-1}}\bigl( \d_\kappa X\cdot\nabla \d_X^{i_j}(-\nabla X)^{i_{j+1}}\bigr)
\cdots\d_X^{i_{r-1}} (-\nabla X)^{i_r}
\\
&
+\d_X^{i_1} (  \v^{\ell-1-i} \cdot \nabla X )
\cdot\d_X^{i_2}(-\nabla X) ^{i_3}
\cdots\d_X^{i_j} \d_\kappa(-\nabla X) \cdots\d_X^{i_{r-1}} (-\nabla X)^{i_r}\Bigr) ,
\end{align*}
and
\begin{align*}
g_{\ell,\kappa}^2
\eqdefa&\sum_{i=0}^{\ell-1}\sum_{i_1+\cdots+i_r=i}
\d_X^{i_1} \d_\kappa (\v^{\ell-1-i}\cdot\nabla X)
\cdot\d_X^{i_2}(-\nabla X) ^{i_3}\cdots\d_X^{i_{r-1}} (-\nabla X)^{i_r}  .
\end{align*}
Notice that the indices $i_1,i_2,\cdots,i_r$ satisfy $i_1+\cdots i_r\leq \ell-2$  in
 $g_{\ell,\kappa}^1$. We thus get, by applying Lemma \ref{lem:R,V}, that for $\kappa=1,2,$
\begin{align*}%%%%%%%% V11 %%%%%%%%
\|g_{\ell,\kappa}^1\|_{L^2}
&\leq
C\sum_{ i+j\leq \ell-2} \cH_{\ell-1}(t)
\bigl(\|\d_X^{i}\nabla  \v^{j}\|_{L^2}
+
\|  \v^{i}\|_{L^{\frac{2p}{p-2}}}
\bigr)
\leq
C\cH_{\ell-1}(t)\sigma^{-\frac{1-s_{\ell-1}}{2}}\frak{A}_{\ell-1}(t) .
\end{align*}
To handle $g_{\ell,\kappa}^2$, we  separate the case when $i_1=i=\ell-1$ from others  to get
\begin{align*}
\|g_{\ell,\kappa}^2\|_{L^2}
&\leq
 \|\nabla v\|_{L^2}\|\d_X^{\ell-1}\nabla X\|_{L^\infty}
 + \|v\|_{L^{\frac{2p}{p-2}}}\|\d_X^{\ell-1}\nabla^2 X\|_{L^p}
\\
&\qquad\qquad
+C\sum_{  i+j\leq \ell-1 }
\cH_{\ell -1}(t)
\bigl(\|\d_X^{i}\nabla \v^j  \|_{L^2}
+\| \v^i\|_{L^{\frac{2p}{p-2}}} \bigr),
\end{align*}
from which and the commutative estimate \eqref{comm:X}, we infer
 \begin{align*}
 \|g_{\ell,\kappa}^2\|_{L^2 }
& \leq
C \bigl(\cH_{\ell-1}(t) +\|\na\rx^{\ell-1}\|_{W^{1,p} }\bigr)
\sigma^{-\frac{1-s_{\ell-1}}{2}}\frak{A}_{\ell-1}(t).
 \end{align*}
 Combining the above estimates of $g_{\ell,\kappa}^1$ and
 $g_{\ell,\kappa}^2$  with \eqref{S6eq7} and Lemma \ref{lem:R,V}, we obtain the first estimate of \eqref{S6eq1}.

Along the same line, we deduce from \eqref{vell'} that
\beno\begin{split}
\|g_{\ell,0}^1
 \|_{L^2 }
\lesssim &
\cH_{\ell-1}(t) \sum_{\substack{i +j\leq \ell-2\\
 l\leq \ell-2}}
\Bigl(  \|\d_X^{l} X_t\|_{L^\infty}
\bigl(\|\d_X^{i}\nabla  \v^{j}\|_{L^2}
+
\|\v^i\|_{L^{\frac{2p}{p-2}}}
\bigr)
+\|\d_X^l\nabla X_t\|_{L^2}\|\v^{i}\|_{L^\infty}\Bigr)
\\
\lesssim&
 \cH_{\ell-1}(t)\sigma^{-\left(1-\frac{s_{\ell-1}}{2}\right)}
 \sum_{\substack{i +j\leq\ell-2\\ l\leq \ell-2}}
\Bigl( \|\sigma^{\frac 12}\d_X^{l}X_t\|_{L^\infty}
\bigl(\|\sigma^{\frac{1-s_{\ell-1}}{2}}\d_X^{i}\nabla  \v^{j}\|_{L^2}
 \\
 &\qquad\qquad\qquad\qquad\qquad+
\|\sigma^{\frac{1-s_{\ell-1}}{2}} \v^i\|_{L^{\frac{2p}{p-2}}}
\bigr)
+\|\sigma^{\frac{1-s_{\ell-1}}{2}}\d_X^l\nabla X_t\|_{L^2}
\|\sigma^{\frac 12} \v^{i}\|_{L^\infty}\Bigr),
\end{split}\eeno
from which and Lemma \ref{lem:R,V}, we infer that
\beq\label{tildevell2}\begin{split}
\|g_{\ell,0}^1\|_{L^2 }
\leq
 \cH_{\ell-1}(t)
\sigma^{-\left(1-\frac{s_{\ell-1}}{2}\right)} \frak{A}_{\ell-1}^2(t)
\leq \cH_{\ell-1}(t)
\sigma^{-\left(1-\frac{s_{\ell-1}}{2}\right)}.
\end{split}\eeq
While by separating the case when $i_1=i=\ell-1$ from others and taking into account of the fact that:
$X_t=-v\cdot\nabla X+\v^1$ one has
\beno
\begin{split}
 \|g_{\ell,0}^2\|_{L^2}\lesssim&
 \|v_t\cdot\d_X^{\ell-1}\nabla X \|_{L^2}
  + \|v\otimes \d_X^{\ell-1}\nabla (v\cdot\nabla X) \|_{L^2}
  + \|v\cdot \d_X^{\ell-1}\nabla \v^1 \|_{L^{2}}
\\
&
+C\sum_{\substack{ i+j\leq \ell-1\\ l+m\leq\ell-1, m\leq \ell-2}}
\cH_{\ell-1}(t)
\bigl(    \|\d_X^{i}\d_t \v^j \|_{L^2}
+   \| \v^l\|_{L^\infty}
\| \d_X^{m}\nabla X_t\|_{L^2}\bigr),
\end{split}\eeno
 which yields
\begin{align*}\|g_{\ell,0}^2\|_{L^2 }
\leq&
\bigl( \| (v_t, v\otimes \nabla v) \|_{L^2}
+\|v\otimes v\|_{L^{\frac{2p}{p-2}}}\bigr)
 \bigl(\|\d_X^{\ell-1}\nabla X\|_{L^\infty}
+ \|\d_X^{\ell-1}\nabla^2 X\|_{L^p} \bigr)
\\
&+\| v\cdot\nabla \v^\ell\|_{L^2 }
+ \sum_{  i+j\leq \ell-1}
\bigl(\|v\otimes\d_X^i\nabla \v^j\|_{L^2}
+\|v\otimes \v^i\|_{L^{\frac{2p}{p-2}}}\bigr)R_{\ell-1}(t)\\
&+\| v\otimes(\d_X^{\ell-1}\nabla \v^1- \nabla \v^\ell)\|_{L^2}
+C   \cH_{\ell-1}(t) \sigma^{-\left(1-\frac{s_{\ell-1}}{2}\right)}
\bigl(\frak{A}_{\ell-1}(t)+\frak{A}_{\ell-1}^2(t)\bigr) .
\end{align*}
However, due to
\begin{align*}
\d_X^{\ell-1}\nabla \v^1-\nabla \v^\ell
=\sum_{i=0}^{\ell-2}\d_X^i[\d_X,\nabla]\v^{\ell-i-1}
&=-\sum_{i=0}^{\ell-2}\d_X^i (\nabla X\cdot\nabla \v^{\ell-i-1})
\\
&=-\sum_{i=0}^{\ell-2}\sum_{m=0}^i C_i^m
\d_X^m(\nabla X)\cdot\d_X^{i-m}\nabla \v^{\ell-i-1},
\end{align*}
we deduce from the commutative estimate \eqref{comm:X} and Lemma \ref{lem:R,V} that
\begin{equation*}
\begin{split}
 \|g_{\ell,0}^2\|_{L^2 }
\leq &
\|v\|_{L^\infty}\|\nabla \v^\ell\|_{L^2}
 +\cH_{\ell-1}(t) \sigma^{-\left(1-\frac{s_{\ell-1}}{2}\right)}
+C\bigl(\cH_{\ell-1}(t) +\|\na\rx^{\ell-1}\|_{W^{1,p}}\bigr)\\
&\times
\sum_{ i+j\leq\ell-1 }
 \Bigl(\|v_t\|_{L^2}
 +\|v\|_{L^\infty}\bigl( \|\d_X^i\nabla \v^j\|_{L^2}
 +\|\v^i\|_{L^{\frac{2p}{p-2}}}\bigr)\Bigr)\\
\leq &
\|v\|_{L^\infty}\|\nabla \v^\ell\|_{L^2}
 +\cH_{\ell-1}(t) \sigma^{-\left(1-\frac{s_{\ell-1}}{2}\right)}
 \bigl(1+\|\na\rx^{\ell-1}\|_{W^{1,p}}\bigr),\end{split}
\end{equation*}
which together with the Estimate \eqref{tildevell2} ensures the second estimate of \eqref{S6eq1}.
This finishes the proof of Lemma \ref{lem:vell}.
\end{proof}

%%%%%%%%%%%%%%%%%%%%%%%%%%%%%%%%%%
\subsection{Time-weighted $H^1$ energy estimate of $ \v^{\ell}$}
%%%%%%%%%%%%%%%%%%%%%%%%%%%%%
In this subsection,  we  follow the same lines as that in Section \ref{sec:J1} to derive  the time-weighted $H^1$ estimate of $\v^{\ell}$.
Similar to the beginning of Section \ref{sec:J1}, due to \eqref{eq:vell},
we first decompose $(\v^\ell, \nabla{\rm{\pi}}^\ell)$ as
\beq \label{S6eq2}
\v^\ell=\v_{\ell 1}+\v_{\ell 2},
\andf \nabla{\rm{\pi}}^\ell=\nabla \rm{p}_{\ell 1}+\nabla \rp_{\ell 2},
\eeq
with $(\v_{\ell 1}, \nabla \rp_{\ell 1})$ and $(\v_{\ell 2}, \nabla \rp_{\ell 2})$ solving the following systems respectively
\beq\label{vell1}
\left\{\begin{array}{l}
\rho \d_t \v_{\ell 1}+\rho v\cdot\nabla \v_{\ell 1}
-\Delta \v_{\ell 1}+\nabla \rp_{\ell 1}=0,
\\
\div \v_{\ell 1}=0,
\\
\v_{\ell 1}|_{t=0}=\d_{X_0}^\ell v_0,
\end{array}\right.
\eeq
and
\beq\label{vell2}
\left\{\begin{array}{l}
\rho \d_t \v_{\ell 2}+\rho v\cdot\nabla \v_{\ell 2}
-\Delta \v_{\ell 2}+\nabla {\rm{p}}_{\ell 2}=F_\ell(v,\pi),
\\
\div \v_{\ell 2}=\div g_\ell,
\\
\v_{\ell 2}|_{t=0}=0,
\end{array}\right.
\eeq
where $F_\ell(v,\pi)$ and $g_{\ell}$ are given by \eqref{S1eq5} and \eqref{tildevell} respectively.

It follows from Proposition \ref{S2prop1} that
\beq
\|\v_{\ell 1}\|_{\wt{L}^\infty_t(\cB^{s_\ell})}
+\|\nabla \v_{\ell 1}\|_{\wt{L}^2_t(\cB^{s_\ell })}
\leq
C\|\p_{X_0}^\ell v_0\|_{\cB^{s_\ell }}
\exp\bigl(C\|v_0\|_{L^2}^4\bigr),
\label{vell1:L2}
 \eeq
and
 \beq
\begin{split}
\|\s^{\f{1-s_\ell }2}
\na \v_{\ell 1}\|_{L^\infty_t(L^2)}
&+
\bigl\|\s^{\f{1-s_\ell }2}(\p_t\v_{\ell 1},\na^2
\v_{\ell 1},\na \rp_{\ell 1})\bigr\|_{L_t^2(L^2)}\leq \cC(v_0,\p_{X_0}^\ell v_0,s_\ell).
\end{split} \label{vell1:H1} \eeq

\begin{prop}\label{S6prop1}
{\sl Let $A_{\ell1}(t)$ be given by \eqref{Jell}. Then under the assumptions of Proposition \ref{prop:Jell}, we have
\beq\label{S6eq3}
A_{\ell1}^2(t)\leq \cA_\ell\exp\left(\cA_0\w{t}^2\right)
+\cH_{\ell-1}(t)
\Bigl(1+\int_0^t\s^{-\left(1-\f{\th_0}2\right)} \|\na\rx^{\ell-1}(t')\|_{W^{1,p}}^2\,dt'\Bigr)
\eqdefa\wt{\frak{G}}_{\ell,X}(t).
\eeq
}
\end{prop}

\begin{proof} We first deduce from Lemma \ref{S3lem2} and \eqref{vell2} that
\beq\label{S6eq4}
\begin{split}
\|&\s^{\f{1-s_\ell }2}
\na \v_{\ell 2}\|_{\wt{L}^\infty_t(L^2)}^2
+
\bigl\|\s^{\f{1-s_\ell }2}(\p_t\v_{\ell 2},\na^2
\v_{\ell 2},\na {\rm{p}}_{\ell 2})\bigr\|_{L_t^2(L^2)}\\
&\leq \cA_0\w{t}\Bigl(\|g_\ell(0)\|_{\cB^{s_{\ell-1}}}^2+\bigl\|\s^{-\f{s_\ell}2}(\p_tg_\ell, F_\ell(v,\pi))\bigr\|_{L^1_t(L^2)}^2+\|\s^{-\f{1-(s_0-s_\ell)}2}\na g_\ell\|_{L^1_t(L^2)}^2\\
&\qquad\qquad\qquad\quad
+\|\s^{-\f{s_\ell}2}\na g_\ell\|_{L^2_t(L^2)}^2+\bigl\|\s^{\f{1-s_\ell}2}(\p_tg_\ell,\na\dive g_\ell, F_\ell(v,\pi))\bigr\|_{L^2_t(L^2)}^2\Bigr).
\end{split}
\eeq
In view of \eqref{tildevell1}, we get, by applying the law of product  in Besov spaces (see \cite{BCD} for instance), that
\beno\begin{split}
\|g_{\ell}(0)\|_{\cB^{s_{\ell-1}}}
&\leq
C\sum_{i,j,l\leq\ell-1}
\|\v^i(0)\|_{\cB^{s_{\ell-1}}}
 \|\d_{X_0}^j\nabla X_0\|_{\dot B^{\f2p}_{p,1}}^{l}
 \\
 &\leq C\sum_{i,j\leq\ell-1}\|\v^i(0)\|_{\cB^{s_{\ell-1}}}
 \|\d_{X_0}^j\nabla X_0\|_{W^{1,p}}^l
 \\
& \leq
 C\bigl(\|v\|_{L^2\cap \cB^{s_0}},\cdots,
 \|\d_{X_0}^{\ell-1}v_0\|_{\cB^{s_{\ell-1}}},
 \|  X_0\|_{W^{2,p}}, \cdots,
 \|\d_{X_0}^{\ell-1} X_0\|_{W^{2,p}}\bigr)\leq \cA_\ell.
\end{split}\eeno
While it follows from Lemma \ref{lem:Fell} and  $\th_0=s_{\ell-1}-s_\ell$ that
\beno \begin{split}
\bigl\|\s^{-\f{s_\ell}2}F_\ell(v,\pi)\bigr\|_{L^1_t(L^2)}^2+&
\bigl\|\s^{\f{1-s_\ell}2}(\na\dive g_\ell, F_\ell(v,\pi))\bigr\|_{L^2_t(L^2)}^2\\
\lesssim & \cH_{\ell-1}^2(t)\int_0^t\s^{-\left(1-\f{\th_0}2\right)}\bigl(1+\|\na\rx^{\ell-1}\|^2_{W^{1,p}}\bigr)\,dt'.
\end{split}
\eeno
And Lemma \ref{lem:vell} ensures that
\beno
\begin{split}
\bigl\|&\s^{-\f{s_\ell}2}\p_tg_\ell\bigr\|_{L^1_t(L^2)}^2+\|\s^{-\f{1-(s_0-s_\ell)}2}\na g_\ell\|_{L^1_t(L^2)}^2+\|\s^{-\f{s_\ell}2}\na g_\ell\|_{L^2_t(L^2)}^2+\|\s^{\f{1-s_\ell}2}\p_tg_\ell\|_{L^2_t(L^2)}^2\\
&\lesssim\w{t}\|\s^{\f{1-s_0}2}v\|_{L^\infty_t(L^\infty)}^2\int_0^t\s^{-\left(1-\f{s_0}2\right)}\|\s^{\f{1-s_\ell}2}\na\v^\ell\|_{L^2}^2\,\dt'\\
&\qquad\qquad\qquad\qquad+\cH_{\ell-1}^2(t)\int_0^t\s^{-\left(1-\f{\th_0}2\right)}\bigl(1+\|\na\rx^{\ell-1}\|^2_{W^{1,p}}\bigr)\,\dt'.
\end{split}
\eeno
Yet it follows from \eqref{S6eq2} that
\beno
\begin{split} \int_0^t\s^{-\left(1-\f{s_0}2\right)}\|\s^{\f{1-s_\ell}2}\na\v^\ell\|_{L^2}^2\,\dt'
\lesssim& \w{t}\|\s^{\f{1-s_\ell}2}\na \v_{\ell 1}\|_{L^\infty_t(L^2)}^2
+\int_0^t\s^{-\left(1-\f{s_0}2\right)}\|\s^{\f{1-s_\ell}2}\na \v_{\ell 2}\|_{L^2}^2\,\dt'.
\end{split}
\eeno
This together with  \eqref{vell1:H1} implies
\beno
\begin{split}
\bigl\|&\s^{-\f{s_\ell}2}\p_tg_\ell\bigr\|_{L^1_t(L^2)}^2+\|\s^{-\f{1-(s_0-s_\ell)}2}\na g_\ell\|_{L^1_t(L^2)}^2+\|\s^{-\f{s_\ell}2}\na g_\ell\|_{L^2_t(L^2)}^2
+\|\s^{\f{1-s_\ell}2}\p_tg_\ell\|_{L^2_t(L^2)}^2\\
&\lesssim \cA_\ell\w{t}^2+\cA_0\w{t}\int_0^t\s^{-\left(1-\f{s_0}2\right)}\|\s^{\f{1-s_\ell}2}\na \v_{\ell 2}\|_{L^2}^2\,\dt'\\
&\qquad\qquad\qquad\qquad+\cH_{\ell-1}^2(t)\int_0^t\s^{-\left(1-\f{\th_0}2\right)}\bigl(1+\|\na\rx^{\ell-1}\|^2_{W^{1,p}}\bigr)\,\dt'.
\end{split}
\eeno
Inserting the above estimates into \eqref{S6eq4} and applying Gronwall's inequality gives rise to
\beq\label{S6eq5}
\|\s^{\f{1-s_\ell }2}
\na \v_{\ell 2}\|_{\wt{L}^\infty_t(L^2)}^2
+
\bigl\|\s^{\f{1-s_\ell }2}(\p_t\v_{\ell 2},\na^2
\v_{\ell 2},\na {\rm{p}}_{\ell 2})\bigr\|_{L_t^2(L^2)}^2\leq \wt{\frak{G}}_{\ell,X}(t),
\eeq
for $\wt{\frak{G}}_{\ell,X}(t)$ given by \eqref{S6eq3}.

Let us now turn to the $L^2$ energy estimate of $\v^\ell.$
Similar to the derivation of \eqref{S3eq9},  we   get, by taking $L^2(\R^2)$ inner product of \eqref{vell2} with $\v_{\ell2}$ and 
making use of the fact: $\div \v_{\ell2}=\div g_{\ell},$ that
%\begin{equation*}\label{estimate:vell}
%\begin{split}
% \frac12\frac{d}{dt}\int_{\R^2} \rho(t)|\v_{\ell2}(t)|^2 \dx
%+\int_{\R^2} |\nabla \v_{\ell2} |^2\dx
% &=\int_{\R^2}  F_\ell  (v,\pi ) | \v_{\ell2}\,\dx
%- \int_{\R^2} \nabla{\rm{p}}_{\ell2}|\v_{\ell2}\,\dx.
%\end{split}
%\end{equation*}
%By using integration by parts and  we write
\begin{equation}\label{estimate:vell}
\begin{split}
\frac12\frac{d}{dt}\|\sqrt{\rho}\v_{\ell2}\|_{L^2}^2
+\|\nabla \v_{\ell2}\|_{L^2}^2
 =&\int_{\R^2}  F_\ell  (v,\pi ) | (\v_{\ell2}-g_{\ell})\,\dx\\
&+\int_{\R^2} (\rho \d_t \v_{\ell2}
+\rho v\cdot\nabla \v_{\ell2}
 -\Delta \v_{\ell2}) | g_{\ell} \,\dx .
\end{split}
\end{equation}
By virtue of Lemmas \ref{lem:Fell} and \ref{lem:vell}, we find
\begin{equation*}\label{estimate:Fell}
\begin{split}
|\int_{\R^2}  F_\ell  (v,\pi ) | (\v_{\ell2}-g_{\ell})\,\dx|
 &\leq
\sigma^{1- {s_{\ell}} }\|F_\ell(v,\pi)\|_{L^2} ^2
+\sigma^{-(1- {s_{\ell}} )}\|\v_{\ell2}\|_{L^2}^2
+\sigma^{-(1- {s_{\ell}} )}\|g_{\ell}\|_{L^2}^2
 \\
&\leq
\sigma^{-(1- {s_{\ell}} )}
\|\v_{\ell2}\|_{L^2}^2
+\cH_{\ell-1}(t)\s^{-(1-\th_0)}\bigl(1+\|\na\rx^{\ell-1}\|^2_{W^{1,p}}\bigr),
\end{split}
\end{equation*}
and
$$\longformule{|\int_{\R^2} (\rho \d_t \v_{\ell2}
+\rho v\cdot\nabla \v_{\ell2}
 -\Delta \v_{\ell2}) | g_{\ell} \,\dx|
   \lesssim
 \|\sigma^{\frac{1-s_\ell}{2}}
   (\d_t \v_{\ell2}, \Delta \v_{\ell2})\|_{L^2}^2}
 {{}  +C\|v\|_{L^\infty}^2\|\sigma^{\frac{1-s_\ell}{2}}
   \nabla \v_{\ell2}\|_{L^2}^2
   +\cH_{\ell-1}(t)\s^{-(1-s_\ell)}\bigl(1+\|\na\rx^{\ell-1}\|^2_{W^{1,p}}\bigr).
}$$
Inserting the above inequalities into   \eqref{estimate:vell}   and then applying Gronwall's inequality to the resulting inequality, we achieve
 \begin{equation*}
 \begin{split}
 \|&\v_{\ell2}\|_{L^\infty_t(L^2)}^2
+\|\nabla \v_{\ell2}\|_{L^2_t(L^2)}^2
 \leq  C\exp\left(\w{t}\right)\Bigl(
 \|\sigma^{\frac{1-s_\ell}{2}}
   (\d_t \v_{\ell2}, \Delta \v_{\ell2})\|_{L^2_t(L^2)}^2\\
  &\quad +\|v\|_{L^2_t(L^\infty)}^2\|\sigma^{\frac{1-s_\ell}{2}}
   \nabla \v_{\ell2}\|_{L^\infty_t(L^2)}^2
  +\cH_{\ell-1}(t)\int_0^t\s^{-(1-\th_0)}\bigl(1+\|\na\rx^{\ell-1}\|^2_{W^{1,p}}\bigr)\,dt'\Bigr),
 \end{split}
 \end{equation*}
 which together with Corollary \ref{S1col1} and \eqref{S6eq5} ensures that
 \begin{equation}\label{S6eq6}
 \begin{split}
&\|\v_{\ell2}\|_{L^\infty_t(L^2)}^2
+\|\nabla \v_{\ell2}\|_{L^2_t(L^2)}^2
 \leq \wt{\frak{G}}_{\ell,X}(t).
 \end{split}
 \end{equation}

 By summing up \eqref{vell1:L2}, \eqref{vell1:H1}, \eqref{S6eq5} and \eqref{S6eq6},
 we conclude the proof of \eqref{S6eq3}.
\end{proof}

  \begin{col}\label{lem:vellH1}
{\sl Under the assumptions of Proposition \ref{S6prop1}, one has
  \beq\label{vell:L4}\begin{split}
  \|\sigma^{\frac {1-s_\ell}p}\v^\ell\|_{L^\infty_t(L^{\frac{2p}{p-2}})}^2
 & + \|\sigma^{\frac{1-s_{\ell}}{2}}
  ( D_t \v^\ell,
  \nabla\d_X^\ell\nabla v,
  %\d_X^\ell\nabla^2 v,
  D_t\d_X^{\ell-1}\Delta X)\|_{L^2_t(L^2)}^2
  \\
&  +
\|\sigma^{\frac{1- s_{\ell} }{p}}
(\nabla \v^\ell, \d_X^\ell\nabla v,
D_t\d_X^{\ell-1}\nabla X)\|_{L^2_t(L^{\frac{2p}{p-2}})}^2
   \leq
 \wt{\frak{G}}_{\ell,X}(t).
\end{split}  \eeq}
\end{col}
%%%%%%%%%%%
\begin{proof}
We  first deduce from  Proposition \ref{S6prop1} and the 2-D interpolation inequality \eqref{S2eq6} that
$$
\|\sigma^{\frac {1-s_\ell}p}\v^\ell\|_{L^\infty_t(L^{\frac{2p}{p-2}})}^2
\leq C\Bigl(\|\v^\ell\|_{L^\infty_t(L^2)}^{1-\frac 2p}
\|\sigma^{\frac{1-s_\ell}{2}}\nabla\v^\ell\|_{L^\infty_t(L^2)}^{\frac 2p}\Bigr)^2
\leq \wt{\frak{G}}_{\ell,X}(t).
$$
It follows  from Corollary \ref{S1col1} and Proposition \ref{S6prop1} that
\begin{align*}
 \|\sigma^{\frac{1-s_\ell}{2}}D_t \v^\ell\|_{L^2_t(L^2)}^2
 \leq
\|\sigma^{\frac{1-s_\ell}{2}}\d_t \v^\ell\|_{L^2_t(L^2)}^2
 +\|v\|_{L^2_t(L^\infty)}^2
 \|\sigma^{\frac{1-s_\ell}{2}}\nabla \v^\ell\|_{L^\infty_t(L^2)}^2\leq \wt{\frak{G}}_{\ell,X}(t).
\end{align*}
While it is easy to observe from \eqref{S3eq0} that
\begin{align*}
\nabla\d_X^\ell\nabla v
&=\nabla\sum_{i=0}^{\ell-1}\d_X^i[\d_X; \nabla]\v^{\ell-1-i}
+\nabla^2\v^\ell
=-\sum_{i=0}^{\ell-1}\nabla\d_X^i(\nabla X\cdot\nabla\v^{\ell-1-i})+\nabla^2 \v^\ell,
\end{align*}
from which, we infer
\begin{align*}
 \|\sigma^{\frac{1-s_\ell}{2}}
 (\nabla\d_X^{\ell}\nabla v&-\nabla^2 \v^\ell)\|_{L^2_t(L^2)}^2
 \lesssim \bigl(
 \|\sigma^{1-\frac{s_0}{2}}\nabla v\|_{L^\infty_t(L^{\frac{2p}{p-2}})} ^2+
 \|\sigma^{1-\frac{s_0}{2}}\nabla^2 v\|_{L^\infty_t(L^2)}^2\bigr)
\\
&\times \int^t_0 \sigma^{-(1-\theta_0)} \bigl(\|\nabla\d_X^{\ell-1}\nabla X\|_{L^p}^2+\|\d_X^{\ell-1}\nabla X\|_{L^\infty}^2\bigr)
\,\dt'
 \\
+\w{t}&\sum_{l\leq\ell-2, i+j\leq\ell-1}
\Bigl(\|\nabla\d_X^l\nabla X\|_{L^\infty_t(L^p)}^2
\|\sigma^{1-\frac{s_{\ell-1}}{2}}
 \d_X^i\nabla\v^j\|_{L^\infty_t(L^{\frac{2p}{p-2}})}^2
\\
&\qquad\qquad\qquad
+\|\d_X^l\nabla X\|_{L^\infty_t(L^\infty)}^2
\|\sigma^{1-\frac{s_{\ell-1}}{2}}
 \nabla\d_X^i\nabla\v^j\|_{L^\infty_t(L^2)}^2\Bigr).
\end{align*}
As a result, we deduce from   Lemmas \ref{lem:comm,X} and \ref{lem:R,V}, and Proposition \ref{S6prop1}  that
\begin{align}\label{vellL4:v''}
\|\sigma^{\frac{1-s_\ell}{2}}\nabla\d_X^{\ell}\nabla v\|_{L^2_t(L^2)}^2
\leq \wt{\frak{G}}_{\ell,X}(t).
\end{align}
%Note that
%$$
%\d_X^\ell\nabla^2 v-\nabla\d_X^\ell\nabla v
%=\sum_{i=0}^{\ell-1}\d_X^i[\d_X,\nabla]\d_X^{\ell-1-i}\nabla v
%=-\sum_{i=0}^{\ell-1}\d_X^i(\nabla X\cdot\nabla\d_X^{\ell-1-i}\nabla v),
%$$
%we can follow  the same proof of \eqref{vellL4:v''} to achieve  the same estimate   for $\d_X^\ell\nabla^2 v$.

Next due to $D_t\rx^{\ell-1}= \v^\ell,$ we have
\begin{align*}
D_t\d_X^{\ell-1}\Delta X
&=D_t\sum_{i=0}^{\ell-2}\d_X^i[\d_X;\Delta]\rx^{\ell-2-i}
+D_t\Delta\rx^{\ell-1}
\\
&=-\sum_{i=0}^{\ell-2}D_t\d_X^i \bigl(\Delta X\cdot\nabla\rx^{\ell-2-i}
+2\p_\al X\cdot\nabla\p_\al\rx^{\ell-2-i}\bigr)\\
&\qquad
+\Delta \v^\ell
-\bigl(\Delta v\cdot\nabla\rx^{\ell-1}+2\p_\al v\cdot\nabla\p_\al \rx^{\ell-1}\bigr),
\end{align*}
and it thus comes out
\begin{align*}
&\|\sigma^{\frac{1-s_\ell}{2}}
(D_t\d_X^{\ell-1}\Delta X-\Delta \v^\ell)\|_{L^2_t(L^2)} ^2
\\
&\ \lesssim \w{t}
\sum_{m+n+\kappa+l\leq\ell-2}\Bigl(
\|\sigma^{1-\frac{s_{\ell-1}}{2}} D_t\d_X^m \nabla^2 \rx^n \|_{L^\infty_t(L^2)}^2
\| \d_X^\kappa \nabla\rx^l\|_{L^\infty_t(L^\infty)}^2
\\
&\qquad\qquad\qquad\qquad
+
\|\d_X^m\nabla^2\rx^n\|_{L^\infty_t(L^p)}^2
\|\sigma^{1-\frac{s_{\ell-1}}{2}}
 D_t\d_X^\kappa\nabla\rx^l\|_{L^\infty_t(L^{\frac{2p}{p-2}})}^2\Bigr)
\\
&\qquad
+ \Bigl(\|\sigma^{1-\frac{ s_{\ell-1}}{2}}
\nabla^2 v\|_{L^\infty_t(L^2)} ^2
+\|\sigma^{1-\frac{ s_{\ell-1}}{2}}
 \nabla v\|_{L^\infty_t(L^{\frac{2p}{p-2}})}^2\Bigr)
\int^t_0\sigma^{-(1-\theta_0)}
  \|\na\rx^{\ell-1}\|_{W^{1,p}} ^2\,\dt',
\end{align*}
from which  and Lemma \ref{lem:R,V}, Proposition \ref{S6prop1}, we infer that
\begin{align*}
 \|\sigma^{\frac{1-s_\ell}{2}}D_t\d_X^{\ell-1}\Delta X\|_{L^2_t(L^2)} ^2
 \leq \wt{\frak{G}}_{\ell,X}(t).
\end{align*}

On the other hand, by virtue of \eqref{S6eq3}, we get from \eqref{S2eq6} that
\beq\label{vellL4:v'}\begin{split}
\|\sigma^{\frac {1-s_\ell}{p}}\nabla \v^\ell\|_{L^2_t(L^{\frac{2p}{p-2}})}^2
&\leq C
\Bigl(\|\nabla \v^\ell\|_{L^2_t(L^2)}^{1-\frac 2p}
\|\sigma^{\frac {1-s_\ell}{2}}\nabla^2 \v^\ell\|_{L^2_t(L^2)}^{\frac 2p}\Bigr)^2
\leq \wt{\frak{G}}_{\ell,X}(t).
\end{split}\eeq
Then, due to
\begin{align*}
\d_X^\ell\nabla v
&=\sum_{i=0}^{\ell-1}
\d_X^i[\d_X;\nabla]\v^{\ell-1-i}
+\nabla\v^\ell
=-\sum_{i=0}^{\ell-1}\d_X^i(\nabla X\cdot\nabla\v^{\ell-1-i})
+\nabla \v^\ell ,
\end{align*}
we have (noticing $p>2$)
\beno\begin{split}
\|\sigma^{\frac{1- s_{\ell}}{p}}
(\d_X^\ell\nabla v-\nabla \v^\ell)&\|_{L^2_t(L^{\frac{2p}{p-2}})}^2
\leq
 \int^t_0 \sigma^{-1+s_0-\frac 2p s_\ell}\|\d_X^{\ell-1}\nabla X\|_{L^\infty}^2\,\dt'
\|\sigma^{\frac 12+\frac 1p- \frac{ s_0}{2}}\nabla v\|_{L^\infty_t(L^{\frac{2p}{p-2}})}^2
\\
&\quad
+\w{t}\sum_{m\leq\ell-2, i+j\leq\ell-1}\|\d_X^m\nabla X\|_{L^\infty_t(L^\infty)}^2
\|\sigma^{\frac 12+\frac 1p-\frac{s_{\ell-1}}{2}}
 \d_X^i\nabla\v^j\|_{L^\infty_t(L^{\frac{2p}{p-2}})}^2.
\end{split}\eeno
This together with  Lemma  \ref{lem:R,V} and \eqref{vellL4:v'} ensures that \eqref{vellL4:v'} also holds for $ \d_X^\ell\nabla v $.

Finally observing from \eqref{eq:X} and \eqref{S3eq0} that
\begin{align*}
D_t\d_X^{\ell-1}\nabla X
&=D_t\sum_{i=0}^{\ell-2}\d_X^i[\d_X; \nabla]\rx^{\ell-2-i}
+D_t\nabla\rx^{\ell-1}
\\
&=-\sum_{i=0}^{\ell-2}D_t\d_X^i (\nabla X\cdot\nabla\rx^{\ell-2-i})
+\nabla \v^\ell-\nabla v\cdot\nabla\rx^{\ell-1},
\end{align*}
we infer from   Lemma \ref{lem:R,V} and \eqref{vellL4:v'}
\beno\begin{split}
\|\sigma^{\frac{1- s_{\ell}}{p}}
&(D_t\d_X^{\ell-1}\nabla X-\nabla \v^\ell)\|_{L^2_t(L^{\frac{2p}{p-2}})}^2\\
\leq &
 \int^t_0 \sigma^{-1+s_0-\frac 2p s_\ell}
 \|\nabla   \rx^{\ell-1}\|_{L^\infty}^2\,\dt'
\|\sigma^{\frac 12+\frac 1p- \frac{ s_0}{2}}\nabla v\|_{L^\infty_t(L^{\frac{2p}{p-2}})}^2
\\
&
+\w{t}\sum_{m+n+\kappa+l \leq\ell-2 }
\|\sigma^{\frac 12+\frac 1p-\frac{s_{\ell-1}}{2}}
  \d_X^m D_t\nabla \rx^n\|_{L^\infty_t(L^{\frac{2p}{p-2}})}^2
\|   \d_X^\kappa\nabla\rx^l\|_{L^\infty_t(L^\infty)}^2\leq \wt{\frak{G}}_{\ell,X}(t).
\end{split}\eeno
This completes the proof of \eqref{vell:L4}.
\end{proof}

%%%%%%%%%%%%%%%%%%%%%%%%%%%%%%%%%%%%%
%%%%%%%%%%%%%%%%%%%%%%%%%%%%%%%%%%%%
\setcounter{equation}{0}
\section{Energy estimate of
 $D_t \v^\ell$ }\label{sec:Dvell}
 %%%%%%%%%%%%%%%%%%%%%%%%%%%%%%%
  %%%%%%%%% Lemma Dvell %%%%%%%%%%%

The goal of this section is to derive the time-weighted $H^1$ energy estimate for $D_t\v^\ell.$ To this end, for $F_{\ell D}$ given by \eqref{S1eq18},
we denote
  \begin{equation}\label{FellD}
  \begin{split}
  \widetilde F_{\ell D}\eqdefa &F_{\ell D}
  +2\p_\al v\cdot\nabla\p_\al \v^\ell+\Delta v\cdot\nabla \v^\ell
  -\nabla v\cdot\nabla{\rm \pi}^\ell
  +\rr^\ell D_t^2 v
  \\
  &
  +2D_t(\d_X^{\ell-1}\p_\al X\cdot\nabla\p_\al v)
  +D_t(\d_X^{\ell-1}\Delta X\cdot\nabla v)
  -D_t(\d_X^{\ell-1}\nabla X\cdot\nabla\pi).
  \end{split}
  \end{equation}
Then by \eqref{brho} and \eqref{S1eq18},  we find
  \begin{align*}
  |\widetilde F_{\ell D}|
  &\lesssim
  \sum_{i+j\leq \ell-1, l\leq \ell-2}\Bigl(|D_t^2 \v^i|
  +|\d_X^l D_t\nabla X| \bigl(|\d_X^i \nabla^2 \v^j|+|\d_X^i \nabla{\rm \pi}^j|\bigr)
  \\
&+|\d_X^l\nabla X| \bigl(|\d_X^i D_t \nabla^2 \v^j|+|\d_X^i D_t\nabla{\rm \pi}^j|\bigr)
+|\d_X^l D_t \Delta X| |\d_X^i\nabla \v^j|
  +|\d_X^l \Delta X| |\d_X^i D_t \nabla \v^j|\Bigr),
  \end{align*}
which together with \eqref{Vell}, \eqref{cVell} and Lemma \ref{lem:R,V} implies
 \begin{align*}
 \|\widetilde F_{\ell D}\|_{L^2} ^2
 &\leq
\cH_{\ell-1}^2(t)
 \sum_{i+j\leq\ell-1}
\Bigl(\|(D_t^2 \v^i, \d_X^i D_t\nabla^2 \v^j, \d_X^i D_t\nabla{\rm \pi}^j)\|_{L^2}^2
 +\|\d_X^i D_t\nabla \v^j\|_{L^{\frac{2p}{p-2}}}^2\Bigr)
 \\
 &\quad+
 C  \sum_{i+j\leq \ell-1, l\leq \ell-2}
 \|(\d_X^l D_t\nabla X, \d_X^i\nabla \v^j)\|_{L^\infty} ^2
\|(\d_X^i\nabla^2 \v^j, \d_X^i\nabla{\rm \pi}^j, \d_X^l D_t\Delta X)\|_{L^2}^2
\\
&\leq
\cH_{\ell-1}^2(t)\left(\sigma(t)^{-(3-s_{\ell-1})}\frak{B}_{\ell-1}(t)
+\sigma(t)^{-(4-2s_{\ell-1})}\right).
  \end{align*}
As a result, we deduce from \eqref{FellD} and Lemma \ref{lem:R,V} that
\beno
\begin{split}
  \|\sigma^{\frac{3-s_\ell}{2}}F_{\ell D}\|_{L^2_t(L^2)}^2
&\leq \cH_{\ell-1}^2(t)+
C\|\sigma \nabla v\|_{L^\infty_t(L^\infty)}^2
\|\sigma^{\frac{1-s_\ell}{2}}
(\nabla^2 \v^\ell, \nabla{\rm \pi}^\ell, D_t\d_X^{\ell-1}\Delta X)\|_{L^2_t(L^2)}^2
\\
&\quad
+C\|\sigma^{\frac32-\frac 1p-\frac{s_\ell}{2}}
   (\nabla^2 v, \nabla\pi)\|_{L^\infty_t(L^p)}^2
\|\sigma^{ \frac 1p }
 (\nabla \v^\ell, D_t\d_X^{\ell-1}\nabla X)\|_{L^2_t(L^{\frac{2p}{p-2}})}^2
 \\
 &+ C\int_0^t\bigl(\|(\rr^\ell,\d_X^{\ell-1}\nabla X)\|_{ L^\infty}^2
  +\| \d_X^{\ell-1}\Delta X\|_{L^p}^2\bigr)
  \\
  &\qquad\times
  \Bigl(
   \|\sigma^{\frac{3-s_\ell}{2}}(D^2_t v, D_t\nabla^2 v, D_t\nabla\pi)\|_{L^2}^2
  +\|\sigma^{\frac{3-s_\ell}{2}}D_t\nabla v\|_{L^{\frac{2p}{p-2}}} ^2 \Bigr) \,\dt'
 .
 \end{split}
\eeno
By virtue of \eqref{brho}, Corollary \ref{S2col3}, Proposition \ref{S6prop1} and Corollary   \ref{lem:vellH1}, we  arrive at
  \begin{lem}\label{lem:FellD}
  {\sl For $\ell=2,\cdots,k$, there holds
\begin{equation*}
\begin{split}
 \|\sigma^{\frac{3-s_\ell}{2}}F_{\ell D}\|_{L^2_t(L^2)}^2
\leq \cH_{\ell-1}^2  (t)\Bigl(\cA_\ell+
 \int^t_0
\bigl(\frak{B}_{0}(t')
+\sigma(t')^{-(1-\frac{\theta_0}2)}\bigr)
\|\na\rx^{\ell-1}(t')\|_{W^{1,p}}^2 \,\dt'\Bigr),\end{split}
\end{equation*} for $\frak{B}_{0}(t)$ given by  \eqref{S1eq25}. }\end{lem}

Let the operator $\cE_X$ be given by \eqref{S1eq6a}. Then
in view of  \eqref{divDv1}, we have
\beq\label{tildeDtv}
\frak{a}
\eqdefa D_t v\cdot\nabla X+\cE_X(v\cdot\nabla v) \andf \div D_t \v^1
=\div \frak{a}.
\eeq
Inductively, let us assume that
 $\div D_t \v^{\ell-1}=\div \frak{a}_{\ell-1}
 \hbox{ for }\ell\geq 2. $
We then deduce that
\begin{align*}
\div D_t \v^\ell
=\div(\d_X D_t \v^{\ell-1})
&=\div\bigl(D_t \v^{\ell-1}\cdot\nabla X+X\div D_t \v^{\ell-1}\bigr)\\
&=\div\bigl(
D_t \v^{\ell-1}\cdot\nabla X
+\cE_X \frak{a}_{\ell-1}\bigr) .
\end{align*}
This gives
\beq \label{S7eq1}
\fa_\ell\eqdefa D_t \v^{\ell-1}\cdot\nabla X
+\cE_X \frak{a}_{\ell-1} \andf \dive D_t \v^\ell=\dive\fa_\ell. \eeq
We thus get by induction that  for $\ell\geq 2$,
\begin{equation}\label{tildeDvell}
\begin{split}
\fa_{\ell}
&=\sum_{i=0}^{\ell-1} \cE_X^i (D_t \v^{\ell-1-i}\cdot\nabla X)
+\cE_X^{\ell}(v\cdot\nabla v)
\\
&=\sum_{i=0}^{\ell-1} \sum_{i_1+\cdots+i_r=i}
\d_X^{i_1}(D_t \v^{\ell-1-i}\cdot\nabla X)
\cdot\d_X^{i_2}(-\nabla X)^{i_3}
\cdots\d_X^{i_{r-1}} (-\nabla X)^{i_r}
\\
&\quad
+\sum_{j_1+\cdots+j_l=\ell, j_1\neq \ell}
\d_X^{j_1}(v\cdot\nabla v)
\cdot\d_X^{j_2}(-\nabla X)^{j_3}
\cdots\d_X^{j_{l-1}}(-\nabla X)^{j_l}
+\d_X^\ell (v\cdot\nabla v) .
\end{split}
\end{equation}

The main result concerning the estimate of $\fa_\ell$ is as follows:

\begin{lem}\label{lem:Dvell}
{\sl For $\ell=2,\cdots,k$,  let $\fa_\ell$ be given by \eqref{S7eq1}. Then  there holds
\begin{equation}\label{S9eq1}
\begin{split}
\|\sigma^{\frac{1-s_{\ell}}{2}}
 \fa_{\ell}\|_{L^2_t(L^2)}^2
+&\|\sigma^{1-\frac{s_{\ell}}{2}}
 \nabla \fa_{\ell}\|_{L^2_t(L^2)}^2
+\|\sigma^{\frac{3-s_{\ell}}{2}}
\na\dive\fa_\ell\|_{L^2_t(L^2)}^2
\leq \frak{G}_{\ell,X}(t).\end{split}
\end{equation}
Here and in all that follows, for $\frak{B}_\ell(t)$ determined by \eqref{eq:AB}, we always designate
\beq\label{S9eq8}
\frak{G}_{\ell,X}(t)\eqdefa \cH_{\ell-1}^2 (t)
\Bigl(\cA_\ell+ \int^t_0
\bigl(\frak{B}_{\ell-1}(t')
+\sigma(t')^{- (1-\f{\theta_0}2 )}
\bigr)
\|\na\rx^{\ell-1}(t')\|_{W^{1,p}}^2\,\dt'\Bigr).
\eeq}
\end{lem}

\begin{proof} Let
\beq\label{zell}
\frak{z}_\ell\eqdefa \fa_{\ell}-\v^\ell\cdot\nabla v-v\cdot\d_X^\ell\nabla v.
\eeq
Then we deduce from Lemma \ref{lem:R,V} and \eqref{tildeDvell} that
\beno\begin{split}
\|\frak{z}_{\ell}(t)\|_{L^2}
& \leq
  \bigl(\cH_{\ell-1} (t)+\|\d_X^{\ell-1}\nabla X\|_{L^\infty}\bigr)
   \sigma(t)^{-\left(1-\frac{s_{\ell-1}}{2}\right)}
  \\
&\quad\times \sum_{i +j\leq \ell-1}
 \bigl(\| \sigma^{1-\frac{s_{\ell-1}}{2}}D_t \v^i\|_{L^2}
 +\|\sigma^{\frac 12} \v^i\|_{L^\infty}
 \|\sigma^{\frac{1-s_{\ell-1}}{2}}\d_X^j\nabla v\|_{L^2}\bigr),
\end{split}\eeno
which together with Lemma \ref{lem:R,V} implies
\beq\label{zell:L2}\begin{split}
\|\sigma^{ \frac{1-s_{\ell}}{2} }\frak{z}_{\ell}(t)\|_{L^2} ^2
 &\leq
 \cH_{\ell-1}^2(t)\sigma(t)^{-(1-\theta_0)}
   \bigl(1+\|\na\rx^{\ell-1}\|_{W^{1,p}}^2\bigr) .
\end{split}\eeq

To handle the estimate of $\na\frak{z}_\ell,$
by virtue of \eqref{tildeDvell} and \eqref{zell}, we write
\begin{equation*}
\begin{split}
\na\frak{z}_{\ell}
=\frak{Z}_{\ell,1}
+\frak{Z}_{\ell,2}
+\frak{Z}_{\ell,3}
+\sum_{\substack{i+j=\ell\\ i\neq\ell, j\neq\ell}}\nabla(\v^i\cdot\d_X^j \nabla v),
\end{split}
\end{equation*}
where
\begin{align*}
\frak{Z}_{\ell,1}
&\eqdefa\sum_{i=1}^{\ell-1} \sum_{i_1+\cdots+i_r=i-1}
\Bigl(\d_X^{i_1}\bigl( \nabla X\cdot\nabla\d_X^{i_2}(D_t \v^{\ell-1-i}\cdot\nabla X)\bigr)
\cdots\d_X^{i_{r-1}} (-\nabla X)^{i_r} ,
\\
&\qquad
+\d_X^{i_1}(D_t \v^{\ell-1-i}\cdot\nabla X)\cdots
 \d_X^{i_{m-1}}\bigl(\nabla X\cdot\nabla  \d_X^{i_m}(-\nabla X)  ^{i_{m+1}}\bigr)
\cdots\d_X^{i_{r-1}} (-\nabla X)^{i_r}
\\
&\qquad
+\d_X^{i_1} ( D_t \v^{\ell-1-i} \cdot \nabla X )
\cdot\d_X^{i_2}(-\nabla X) ^{i_3}
\cdots\d_X^{i_m}\nabla (-\nabla X) \cdots\d_X^{i_{r-1}} (-\nabla X)^{i_r} \Bigr),
\end{align*}
and
\begin{align*} \frak{Z}_{\ell,2}
&\eqdefa \sum_{j_1+\cdots+j_l= \ell-1}
\Bigl(\d_X^{j_1}\bigl( \nabla X\cdot\nabla\d_X^{j_2}
(v\cdot\nabla v)\bigr)
\cdots\d_X^{j_{l-1}} (-\nabla X)^{j_l}
\\
&\qquad
+\d_X^{j_1}(v\cdot\nabla v)\cdots
 \d_X^{j_{n-1}}\bigl(\nabla X\cdot\nabla \d_X^{j_n}(-\nabla X)  ^{j_{n+1}}\bigr)
\cdots\d_X^{j_{l-1}} (-\nabla X)^{j_l}
\\
&\qquad
+\d_X^{j_1} ( v\cdot\nabla v )
\cdot\d_X^{j_2}(-\nabla X) ^{j_3}
\cdots\d_X^{j_n}\nabla (-\nabla X) \cdots\d_X^{j_{l-1}} (-\nabla X)^{j_l}\Bigr) ,
\end{align*}
and
\begin{align*}
\frak{Z}_{\ell,3}
&\eqdefa\sum_{i=0}^{\ell-1}\sum_{i_1+\cdots+i_r=i}
\d_X^{i_1} \nabla(D_t \v^{\ell-1-i}\cdot\nabla X)
\cdot\d_X^{i_2}(-\nabla X)^{i_3} \cdots\d_X^{i_{r-1}} (-\nabla X)^{i_r}
\\
&\quad
+ \sum_{j_1+\cdots+j_l=\ell, j_1\neq \ell}
\d_X^{j_1} \nabla(v\cdot\nabla v)
\cdot\d_X^{j_2}(-\nabla X)^{j_3} \cdots\d_X^{j_{l-1}} (-\nabla X)^{j_l}   .
\end{align*}
This gives rise to
\beno\begin{split}
 \|\nabla \frak{z}_{\ell}(t)\|_{L^2}
 \leq &
\cH_{\ell-1} (t)(1+\|\na\rx^{\ell-1}\|_{W^{1,p}})
\sum_{ i+j,l\leq\ell-1}
\Bigl(\| \d_X^i\nabla D_t\v^j\|_{L^2}
+\| D_t\v^i\|_{L^{\frac{2p}{p-2}}}
\\
&
+\sigma(t)^{-\frac{3-s_{\ell-1}}{2}}
\|\sigma^{ 1-\frac{ s_{\ell-1}}{2}}\d_X^i\nabla \v^j\|_{L^\infty}
\bigl(\|\sigma^{ \frac{1-s_{\ell-1}}{2}}\d_X^l\nabla v\|_{  L^2}
+\|\sigma^{ \frac{1 }{2}} \v^l\|_{L^{\frac{2p}{p-2}}}\bigr)
\\
&
+\sigma(t)^{-\frac{3-s_{\ell-1}}{2}}\|\sigma^{ \frac{1}{2}}\v^l\|_{L^\infty}
\|\sigma^{ 1-\frac{ s_{\ell-1}}{2}} \d_X^i\nabla\d_X^j\nabla v\|_{L^2}\Bigr),
\end{split}\eeno
from which and Lemma \ref{lem:R,V},  we infer that
\beq\label{zell':L2}\begin{split}
 \|\sigma^{1-\frac{s_{\ell}}{2}}\nabla \frak{z}_{\ell}(t)\|_{L^2}^2
& \leq
 \cH_{\ell-1}^2 (t)(1+\|\na\rx^{\ell-1}\|_{W^{1,p}}^2)
 \bigl(\frak{B}_{\ell-1}(t)+\sigma(t)^{-(1-\theta_0)}\bigr).
 \end{split}\eeq

Taking into account of \eqref{zell}, we  deduce from \eqref{zell:L2} and \eqref{zell':L2}  that
\beno \begin{split}
 \|\sigma^{\frac{1-s_{\ell}}{2}}&
\fa_{\ell}\|_{L^2_t(L^2)}^2
 \leq
\|\sigma^{\frac{1 }{p}}\d_X^\ell\nabla v\|_{L^2_t(L^{\frac{2p}{p-2}})}^2
\|\sigma^{\frac 12-\frac 1p-\frac{s_0}{2}}v\|_{L^\infty_t(L^p)}^2
\\
&\quad+\| \v^\ell\|_{L^\infty_t(L^2)}^2
\|\sigma^{\frac{1-s_\ell}{2}}\nabla v\|_{L^2_t(L^\infty)}^2
+\cH^2_{\ell-1}(t) \Bigl(1+
\int^t_0\sigma(t')^{-(1-\theta_0)}  \|\na\rx^{\ell-1}(t')\|_{W^{1,p}}^2 \,\dt'\Bigr) ,
\end{split}\eeno
and
\beno\begin{split}
 \|\sigma^{1-\frac{s_{\ell}}{2}}
\nabla \fa_{\ell}\|_{L^2_t(L^2)}^2
&\leq
\|\sigma^{\frac{1- s_\ell }{p}}
(\nabla \v^\ell, \d_X^\ell\nabla v)\|_{L^2_t(L^{\frac{2p}{p-2}})}^2
\|\sigma^{ 1-\frac 1p-\frac{s_0}{2} } \nabla v\|_{L^\infty_t(L^p)}^2
\\
&\quad
+\| \v^\ell\|_{L^4_t(L^4)}^2
\|\sigma^{1-\frac{s_\ell}{2}} \nabla^2 v\|_{L^4_t(L^4)}^2
+\|\sigma^{\frac 12}  v\|_{L^\infty_t(L^\infty)}^2
\|\sigma^{\frac{1-s_\ell}{2}} \nabla\d_X^\ell\nabla v\|_{L^2_t(L^2)}^2
\\
&\quad
+ \cH_{\ell-1}^2(t)\Bigl(1+ \int^t_0 \bigl(\frak{B}_{\ell-1}(t')+\sigma(t')^{-(1-\theta_0)}\bigr)
 \|\na\rx^{\ell-1}(t')\|_{W^{1,p}}^2\,\dt' \Bigr).
\end{split}\eeno
Together with Corollary \ref{S1col1},  Proposition \ref{S6prop1}  and Corollary \ref{lem:vellH1}, we conclude
\beq \label{S9eq9} \begin{split}\|\sigma^{\frac{1-s_{\ell}}{2}}
\fa_{\ell}\|_{L^2_t(L^2)}^2+&\|\sigma^{1-\frac{s_{\ell}}{2}}
\nabla \fa_{\ell}\|_{L^2_t(L^2)}^2\\
  \leq & \cH_{\ell-1}^2 (t)
\Bigl(\cA_\ell+ \int^t_0
\bigl(\frak{B}_{\ell-1}(t')
+\sigma(t')^{-(1-\frac{\theta_0}{2})}\bigr)
\|\na\rx^{\ell-1}(t')\|_{W^{1,p}}^2\,\dt'\Bigr). \end{split} \eeq

Finally let us turn to the estimate of  $\na\div\fa_\ell$.
Indeed due to   $\div \cE_Xf=\d_X\div f$ and \eqref{tildeDvell}, we write
\begin{align*}
\div \fa_{\ell}
&=\sum_{i=0}^{\ell-1} \d_X^i\div(D_t \v^{\ell-1-i}\cdot\nabla X)
+\d_X^\ell\div(v\cdot\nabla v)
\\
&=\sum_{i=0}^{\ell-1} \d_X^i (\nabla D_t \v^{\ell-1-i}:\nabla X)
+\d_X^\ell (\nabla v:\nabla v),
\end{align*}
so that it comes out
\beno
\begin{split}
& \|\sigma ^{ \frac{3-s_\ell}{2}}
\nabla\dive\fa_\ell \|_{L^2_t(L^2)}\leq
 \Bigl(\int_0^t\bigl(\|\sigma ^{ \frac{3-s_\ell}{2}}\na^2D_tv\|_{L^2}^2+\|\sigma ^{ \frac{3-s_\ell}{2}}\na D_tv\|_{L^{\f{2p}{p-2}}}^2\bigr)
\|\d_X^{\ell-1}\na X\|_{W^{1,p}}^2\,dt'\Bigr)^{\f12}
\\
&\qquad+
C R_{\ell-1}(t) \sum_{i+j\leq \ell-1}
\Bigl( \|\sigma ^{ \frac{3-s_\ell}{2}}
\nabla \d_X^i\nabla D_t \v^j \|_{L^2_t(L^2)}
+ \|\sigma^{ \frac{3-s_\ell}{2}}
 \d_X^i\nabla D_t \v^j\|_{L^2_t(L^{\frac{2p}{p-2}})}\Bigr)
 \\
 &\qquad
 +\|\sigma^{\frac 32-\frac 1p-\frac{s_0}{2}}\nabla^2 v\|_{L^\infty_t(L^p)}
 \|\sigma^{\frac 1p}\d_X^\ell\nabla v\|_{L^2_t(L^{\frac{2p}{p-2}})}
 +\|\sigma \nabla v\|_{L^\infty_t(L^\infty)}
 \|\sigma^{\frac{1-s_\ell}{2}}\nabla\d_X^\ell\nabla v\|_{L^2_t(L^2)}
 \\
&\qquad
 +\sum_{i,j\leq\ell-1}
 \|\sigma^{\frac 32-\frac 1p-\frac{s_{\ell-1}}{2}}\nabla\d_X^i\nabla v\|_{L^\infty_t(L^p)}
 \|\sigma^{\frac 12+\frac 1p}\d_X^j\nabla v\|_{L^\infty_t(L^{\frac{2p}{p-2}})}
 \|\sigma^{-\frac{1-\theta_0}{2}}\|_{L^2([0,t])},
\end{split}
\eeno
Then by virtue of \eqref{S9eq9}  and   Corollary \ref{S2col3}, Lemma \ref{lem:R,V},  \eqref{comm:X},
 Proposition \ref{S6prop1}  and Corollary \ref{lem:vellH1}, we
 complete  the proof of \eqref{S9eq1}.
\end{proof}

In view of Lemma \ref{S4lem01},
to close the time-weighted $H^1$ energy estimate of $D_t\v^\ell,$ we also need the estimate of $\dive D_t^2\v^\ell.$
Indeed taking into account of \eqref{divD2v},
we inductively assume that $\div D_t^2 \v^{\ell-1}=\div \frak{b}_{\ell-1}$,
then one has
\beq\label{S9eq3}
\begin{split}
\div D_t^2 {\v^\ell}
=\div(\d_X D_t^2 \v^{\ell-1})
&=\div\bigl(D_t^2 \v^{\ell-1}\cdot\nabla X
+ X\div D_t^2 \v^{\ell-1}\bigr)
\\
&=\div\bigl(D_t^2 \v^{\ell-1}\cdot\nabla X
+\cE_X \frak{b}_{\ell-1}\bigr)
\eqdefa \div \frak{b}_{\ell}.
\end{split}\eeq

\begin{lem}\label{S9lem1}
{\sl Let $\frak{b}_\ell$ be given by \eqref{S9eq3}.  Then for $\ell=2,\cdots,k$, one has
\begin{equation}\label{S9eq4}
\begin{split}
\|&\sigma^{\frac{3-s_{\ell}}{2}}
\frak{b}_\ell\|_{L^2_t(L^2)}^2
\leq  C\e\|\s^{\f{3-s_\ell}2}D_t^2\v^\ell\|_{L^2_t(L^2)}^2+ C_\e\frak{G}_{\ell,X}(t),
%&+\int_0^t\bigl(\|\s^{\f12}\na v\|_{L^\infty}^2+\|\s^{\left(\f54-\f{s_\ell}2\right)}(\na v_t,D_t\na v)%\|_{L^2}^2\bigr)\bigl(\|\s^{\f{1}4}\v^\ell\|_{L^4}^2+\|\s^{1-\f{s_\ell}2}D_t\v^\ell\|_{L^2}^2\bigr)\,dt',
\end{split}
\end{equation}
with $\frak{G}_{\ell,X}(t)$ given by \eqref{S9eq8}.
}
\end{lem}

\begin{proof} For $\frak{b}_0$ given by \eqref{S4eq7p}, we thus get, by using induction to \eqref{S9eq3}, that
\begin{align*}
&\frak{b}_{\ell}
=\sum_{i=0}^{\ell-1} \cE_X^i (D_t^2 \v^{\ell-1-i}\cdot\nabla X)
+\cE_X^\ell \frak{b}_0
\\
&=\sum_{i=0}^{\ell-1}\sum_{i_1+\cdots+i_r=i}
 \d_X^{i_1} (D_t^2 \v^{\ell-1-i}\cdot\nabla X)
 \cdots\d_X^{i_{r-1}}(-\nabla X)^{i_r}
 +\sum_{j_1+\cdots+j_l=\ell}
 \d_X^{j_1}\frak{b}_0 \cdots\d_X^{j_{l-1}}(-\nabla X)^{j_l}.
\end{align*}
It is easy to observe from  \eqref{S4eq7p} that
$$
\frak{b}_0=v\cdot(\nabla v_t+D_t \nabla v)+D_t v\cdot\nabla v
= 2v\cdot D_t\nabla v-v\cdot\nabla(v\cdot\nabla v)+D_t v\cdot\nabla v,
$$
 we obtain
\beq\label{S9eq6}
\begin{split}
& \|\sigma^{ \frac{3-s_\ell}{2}} \frak{b}_{\ell} \|_{L^2_t(L^2)}^2
 \leq
 \int_0^t\bigl\|\s^{\f{3-s_\ell}2}(D_t^2v,\frak{b}_0 )\bigr\|_{L^2}^2
 \|\na\rx^{\ell-1}\|_{L^\infty}^2\,\dt'
 \\
 &+
 2\|\s^{\f12}v\|_{L^\infty_t(L^\infty)}^2
\|\sigma^{1-\frac{s_{\ell}}{2}}\d_X^\ell D_t\nabla v\|_{L^2_t(L^2)}^2
+  \|\s \na v\|_{L^\infty_t(L^\infty)}^2
\|\s^{\f{1-s_\ell}2}D_t\v^\ell\|_{L^2_t(L^2)}^2
\\
&+C\|\sigma^{\frac{1}{p}}\v^\ell\|_{L^\infty_t(L^{\frac{2p}{p-2}})}^2
\|\sigma^{\frac 32-\frac 1p-\frac{s_0}{2}}
(D_t\nabla v, \nabla(v\otimes\nabla v))\|_{L^2_t(L^p)}^2
\\
&+ \|\sigma^{\frac 1p}\d_X^\ell \nabla v\|_{L^2_t(L^{\frac{2p}{p-2}})}^2
\|\sigma^{ \frac 32-\frac 1p-\frac{s_0}{2}}
(v\otimes\nabla v, D_t v)\|_{L^\infty_t(L^p)}^2
\\
 &\ +C(1+R_{\ell-1}(t))^2\sum_{i,j\leq\ell-1}\Bigl(
\w{t}\|\sigma^{1-\frac{s_{\ell-1}}{2}}D_t \v^i\|_{L^\infty_t(L^2)}^2
\|\sigma \d_X^j\nabla v\|_{L^\infty_t(L^\infty)}^2\\
&\
+ \|\sigma ^{ \frac{3-s_\ell}{2}}  D_t^2 \v^i  \|_{L^2_t(L^2)}^2
+\|\sigma \v^i\|_{L^\infty_t(L^\infty)}^2
\|\sigma^{1-\frac{s_{\ell}}{2}}(\d_X^j\nabla(v\cdot\nabla v), D_t\d_X^j\nabla v)\|_{L^2_t(L^2)}^2
\Bigr).
\end{split}
\eeq
Noticing that
\begin{align*}
\d_X^\ell D_t\nabla v
&=\d_X^\ell [D_t; \nabla] v
+\sum_{i=0}^{\ell-1}\d_X^i[\d_X; \nabla]D_t \v^{\ell-1-i}
+\nabla D_t  \v^\ell,
\end{align*}
we obtain
\begin{align*}
\|\sigma^{1-\frac{s_\ell}{2}}  \d_X^\ell D_t\nabla v \|_{L^2_t(L^2)}
\leq&
\|\sigma^{1-\frac{s_\ell}{2}}  \nabla D_t  {\v^\ell} \|_{L^2_t(L^2)}
+2\|\sigma^{\frac 1p}  \d_X^\ell \nabla v \|_{L^2_t(L^{\frac{2p}{p-2}})}
\|\sigma^{1-\frac 1p-\frac{s_0}{2}} \nabla v \|_{L^\infty_t(L^p)}\\
&+\Bigl(\int_0^t\|\s^{1-\f{s_\ell}2}\na D_tv\|_{L^2}^2
\|\na\rx^{\ell-1}\|_{L^\infty}^2\,dt'\Bigr)^{\f12}
\\
&
+\sum_{i,j\leq\ell-1}
\|\sigma   \d_X^i  \nabla v \|_{L^\infty_t(L^\infty)}
\|\sigma^{ \frac{1-s_{\ell-1}}{2}}  \d_X^j \nabla v \|_{L^\infty_t(L^2)}
\|\sigma^{-\frac{1-\theta_0}{2}}\|_{L^2([0,t])}
\\
&
+C(R_{\ell-1}(t)+R^2_{\ell-1}(t))\sum_{i+j\leq\ell-1}
\|\sigma^{1-\frac{s_\ell}{2}}  \d_X^i  \nabla D_t \v^j \|_{L^2_t(L^2)},
\end{align*}
from which, Corollary \ref{S1col1}, Lemma \ref{lem:R,V} and Corollary \ref{lem:vellH1},  we infer
\beq\label{S9eq5}
\begin{split}
\|\sigma^{1-\frac{s_\ell}{2}}  \d_X^\ell D_t\nabla v \|_{L^2_t(L^2)}^2
\leq &
 \|\sigma^{1-\frac{s_\ell}{2}}  \nabla D_t  {\v^\ell} \|_{L^2_t(L^2)}^2\\
&+\wt{\frak{G}}_{\ell,X}(t)
+\int_0^t\|\s^{1-\f{s_\ell}2}\na D_tv\|_{L^2}^2\|\na\rx^{\ell-1}\|_{L^\infty}^2\,\dt' .
\end{split}\eeq
%The same estimate holds for $\|\sigma^{1-\frac{s_\ell}{2}}  \d_X^\ell\nabla v_t \|_{L^2_t(L^2)}.$

 Yet for any $\e>0,$ we get, by applying \eqref{S8eq3} to the equation \eqref{eq:Dvell}, that
 \beno
 \begin{split}
 \|\sigma^{1-\frac{s_\ell}{2}}  \nabla D_t  {\v^\ell} \|_{L^2_t(L^2)}^2
 \leq &
 C\|\s^{\f{1-s_\ell}2}D_t \v^\ell\|_{L^2_t(L^2)}^2
 +\e\|\sigma^{\frac{3-s_\ell}{2}}   D_t ^2 {\v^\ell} \|_{L^2_t(L^2)}^2
 \\
 &+C_\e\left(\|\sigma^{\frac{1-s_\ell}{2}}  \fa_\ell  \|_{L^2_t(L^2)}^2
 +\|\sigma^{1-\frac{s_\ell}{2}}  \nabla \fa_\ell  \|_{L^2_t(L^2)}^2
 +\|\sigma^{\frac{3-s_\ell}{2}} F_{\ell D}   \|_{L^2_t(L^2)}^2\right).
 \end{split}
 \eeno
 Inserting the above estimate into \eqref{S9eq5} and using  Corollary \ref{lem:vellH1}, Lemmas \ref{lem:FellD} and \ref{lem:Dvell}, we achieve
\beq\label{S7ineq:Dv'}
\begin{split}
\|\sigma^{1-\frac{s_\ell}{2}}  \d_X^\ell D_t\nabla v \|_{L^2_t(L^2)}^2
\leq & \e\|\sigma^{\frac{3-s_\ell}{2}}   D_t^2  {\v^\ell} \|_{L^2_t(L^2)}^2
+ C_\e\frak{G}_{\ell,X}(t).
%\cH_{\ell-1} (t)
%\Bigl(\cA_\ell+ \int^t_0
%\bigl(\frak{B}_{\ell-1}(t')
%+\sigma^{-\left(1-\f{\theta_0}2\right)}\bigr)
%\|\na\rx^{\ell-1}\|_{W^{1,p}}^2\,\dt'\Bigr).
\end{split}
\eeq

Finally, we notice that
\begin{align*}
\sum_{j\leq\ell-1}\|\sigma^{1-\frac{s_\ell}{2}} \d_X^j\nabla (v\cdot\nabla v)\|_{L^2_t(L^2)}
\leq &
C\w{t}^{\f12}
\sum_{i+l\leq\ell-1}
\Bigl(\|\sigma^{\frac 12}\v^i\|_{L^\infty_t(L^\infty)}
\|\sigma^{1-\frac{s_{\ell-1}}{2}} \d_X^l\nabla^2 v\|_{L^\infty_t(L^2)}
\\
&\qquad+\|\sigma^{1 -\frac{s_{\ell-1}}{2}}\d_X^i\nabla v\|_{L^\infty_t(L^\infty)}
\|\sigma^{\frac 12 }\d_X^l\nabla v\|_{L^\infty_t(L^2)}\Bigr).
\end{align*}
Substituting the above inequality and \eqref{S7ineq:Dv'} into \eqref{S9eq6} and using Corollaries \ref{S1col1} and \ref{S2col3}, Lemma \ref{lem:R,V},   Corollary \ref{lem:vellH1}, we conclude the proof of \eqref{S9eq4}.
 \end{proof}

\begin{prop}\label{S9prop1}
{\sl Let $A_{\ell1}(t), A_{\ell2}(t)$ be given by \eqref{Jell}. Then  under the assumptions of Proposition \ref{prop:Jell}, we have
\beq\label{S9eq1h}
A_{\ell1}(t)+A_{\ell2}(t)+\|\rx\|_{L^\infty_t(W^{2,p})}\leq \cH_\ell(t).
\eeq}
\end{prop}

\begin{proof} We first get, by applying Lemma \ref{S4lem01} to the System \eqref{eq:Dvell}, that
\beno
\begin{split}
A_{\ell2}^2(t)\leq &
\cA_0\bigl(\|\s^{\f{1-s_\ell}2}(D_t\v^\ell,\fa_\ell)\|_{L^2_t(L^2)}^2+\|\s^{1-\f{s_\ell}2}\na \fa_\ell\|_{L^2_t(L^2)}^2+\|\s^{\f{3-s_\ell}2}(\na\dive\fa_\ell,\frak{b}_\ell,F_{\ell D})\|_{L^2_t(L^2)}^2\bigr),
%+C\int_0^t\|\na v\|_{L^\infty}\|\sigma^{\frac{3-s_\ell}2} \na D_t \v^\ell \|_{L^2}^2\,dt',
\end{split} \eeno
from which, Proposition \ref{S6prop1}, Corollary \ref{lem:vellH1} and Lemmas \ref{lem:FellD}, \ref{lem:Dvell} and \ref{S9lem1}, we infer
\beno
\begin{split} A_{\ell1}^2(t)
+A_{\ell2}^2(t)\leq & C\e \|\s^{\f{3-s_\ell}2}D_t^2\v^\ell\|_{L^2_t(L^2)}^2+C_\e\frak{G}_{\ell,X}(t),
%+ C\int_0^t
%\|\na v\|_{L^\infty}
%+\cB_0(t)\bigr)\\
%\times
%\bigl(\|\v^\ell\|_{L^2}^2+\|\s^{\f{1-s_\ell}2}\na\v^\ell\|_{L^2}^2+\|\s^{1-\f{s_\ell}2}D_t\v^\ell\|_{L^2}^2
%\|\s^{\f{3-s_\ell}2}
%\na D_t\v^\ell\|_{L^2}^2 \,dt',
\end{split}
\eeno for $A_{\ell1}(t), A_{\ell 2}(t)$ and $\frak{G}_{\ell,X}(t)$ given by \eqref{Jell} and \eqref{S9eq4} respectively.
Taking $\e$ so small that $C\e\leq \f12$ ensures
\beq\label{S9eq2}
A_{\ell1}^2(t)+A_{\ell2}^2(t)\leq  \frak{G}_{\ell,X}(t).
\eeq

By virtue of \eqref{S9eq2}, we get, by applying \eqref{S2eq6}, that
 for any $r\in [2,\infty[$
 \begin{equation*}\label{bound:Dvell,q}
 \begin{split}
& \| \sigma^{1-\frac{s_\ell}{2}} D_t {\v^\ell} \|_{L^{\frac{2r}{r-2}}_t(L^r)}
 \leq \frak{G}_{\ell,X}^{\frac 12}(t),
 \end{split}
 \end{equation*}
  which together with the fact: $p<2/(1-s_\ell)$, ensures that
\beq\label{Dvell:p}\begin{split}
 \|D_t {\v^\ell}\|_{L^1_t(L^p)}
& = \| \sigma ^{1-\frac{s_\ell}{2}} D_t {\v^\ell} \|_{L^{\frac{2p}{p-2}}_t(L^p)}
 \| \sigma ^{-1+\frac{s_\ell}{2}} \|_{L^{\frac{2p}{p+2}}([0,t])}
 \leq  \frak{G}_{\ell,X}^{\frac 12}(t).
\end{split}\eeq

On the other hand, taking into account of \eqref{S9eq10} and $\v^\ell=X\cdot\na\v^{\ell-1},$
 we get, by using the  $L^p$ type energy estimate and Lemma \ref{lem:R,V},  that
\begin{equation}\label{bound:Xell}
\begin{split}
\|\rx^{\ell-1}(t)\|_{L^p}\leq &
\|\p_{X_0}^{\ell-1}X_0\|_{L^p}+\|{\v^\ell}\|_{L^1_t(L^p)}
\\
\leq
&\|\p_{X_0}^{\ell-1}X_0\|_{L^p}
+\w{t}\|X\|_{L^\infty_t(L^\infty)}
\|\sigma^{1 -\frac{s_{\ell-1}}{2}}\na \v^{\ell-1}\|_{L^\infty_t(L^p)}
\leq \cA_\ell+\cH_{\ell-1}(t).
\end{split}
\end{equation}
Applying $\Delta$ to \eqref{S9eq10} gives
$$
\d_t\Delta \rx^{\ell-1} +v\cdot\nabla \Delta \rx^{\ell-1}
=-\Delta v\cdot\nabla \rx^{\ell-1}  -2\p_\al
v\cdot\nabla\p_\al\rx^{\ell-1}  +\Delta\v^\ell,
$$ from which, we infer
\beno\begin{split}
\|\D\rx^{\ell-1}(t)\|_{L^{p}}
\leq \|\D\p_{X_0}^{\ell-1}X_0\|_{L^p}
+\int^t_0 (\|\Delta v\|_{L^p}+\|\nabla v\|_{L^\infty})\|\na\rx^{\ell-1}\|_{W^{1,p}}\,\dt'
+\|\Delta {\v^\ell}\|_{L^1_t(L^p)}.
\end{split}\eeno
Summing up the above inequality with \eqref{bound:Xell}, and then applying Gronwall's inequality to the resulting inequality and using  Corollary \ref{S1col1}, we achieve
 \beno
\begin{split}
\|\rx^{\ell-1}\|_{L^\infty_t(W^{2,p})} \leq
\exp\bigl(\cA_0t\bigr)\bigl(\cA_\ell+\cH_{\ell-1}(t)+  \|\Delta  {\v^\ell}\|_{L^1_t(L^p)}\bigr),
\end{split} \eeno
which together with the second inequality of \eqref{est:Fell,p}  ensures that
\beno \begin{split}
&\|\rx^{\ell-1}\|_{L^\infty_t(W^{2,p})}
\leq
 \cH_{\ell-1} (t) \Bigl(\cA_\ell+  \|D_t {\v^\ell}\|_{L^1_t(L^p)}
 +\int^t_0\sigma^{-\left(\frac 32-\frac 1p-\frac{s_{\ell}}{2}\right)}
    \|\na \rx^{\ell-1}\|_{W^{1,p}} \,\dt'\Bigr).
 \end{split} \eeno
 Inserting  \eqref{Dvell:p} into the above inequality and using   the fact that (noticing also $\frac 1p>\frac{1-s_\ell}{2}$)
\begin{align*}
&\int^t_0\sigma^{-(\frac 32-\frac 1p-\frac{s_{\ell-1}}{2})}
  \|\na\rx^{\ell-1}\|_{W^{1,p}} \,\dt'
  \leq \w{t}^{\f12}
\Bigl(\int^t_0\sigma^{-(1-\frac{\th_0}{2})}
  \|\na\rx^{\ell-1}\|_{W^{1,p}}^2 \,\dt' \Bigr)^{\frac 12},
\end{align*}
we  infer
\beno \begin{split}
  \|\rx^{\ell-1}\|_{L^\infty_t(W^{2,p})}^2
    &\leq \frak{G}_{\ell,X}(t).
 \end{split} \eeno
Taking into account of the definition of $\frak{G}_{\ell,X}(t)$ given by \eqref{S9eq8}, we get, by applying  Gronwall's inequality and Lemma \ref{lem:R,V}, that
\beq\label{Xell:W2p}
  \|\rx^{\ell-1}\|_{L^\infty_t(W^{2,p})}
 \leq\cH_\ell(t) .
 \eeq
 Substituting \eqref{Xell:W2p} into \eqref{S9eq2} leads to \eqref{S9eq1h}. This completes the proof of Proposition \ref{S9prop1}.
\end{proof}

With Proposition \ref{S9prop1}, to complete the proof of Proposition \ref{prop:Jell}, it remains to prove that
\beq\label{S9eq11}
\|\v^\ell\|_{{L}^\infty_t(\cB^{s_\ell})}+\|\na
\v^\ell\|_{{L}^2_t(\cB^{s_\ell})}\leq \cH_\ell(t).
\eeq
The proof of \eqref{S9eq11} follows exactly the same argument as those in proofs of Propositions
\ref{S2prop1} and    \ref{S4prop3}. It is quite involved but does not contain new ideas. We skip the details here.

%%%%%%%%%%%%%%%%%%%%%%%%%%%%%%%%%%%%%%%%%%%%%%%%%%%%%%%%%%%%%%%%%%%%%%%
%%%%%%%%%%%%%%%%% Appendix %%%%%%%%%%%%%%%%%%%%%%%%%%%%%%%%%%
%%%%%%%%%%%%%%%%%%%%%%%%%%%%%%%%%%%%%%%%%
\setcounter{equation}{0}
\appendix

%%%%%%%%%%%%%%%%%%%%%%%%%%%%%%%%%%%%%
%%%%%%%%%%%%%%%%%%%%%%%%%%%%%%%%%%%%%
%%%%%%%%%%%%%%%%%%%%%%%%%%%%%%%%%%%
\section{Proof of Lemma \ref{lem:comm,X}  }\label{appA}

In this appendix we present the proof of  Lemma \ref{lem:comm,X}, which basically follows the same arguments as the proof of Lemma 4.1    in \cite{LZ}.

%\subsection{Some identities} The goal of this subsection is to
Let us first present some identities which will be used in what follows.
Let   $(m,n)$ be nonnegative integer pair such that $m+n\leq \ell-1$, $f$ be a smooth enough function. Then it follows from \eqref{S3eq0} that
\beq\label{appeq0}
\begin{split}
\d_X^{m+1}\nabla\d_X^{\ell-m-1}f
-\nabla\d_X^\ell f
&=
 \sum_{\kappa=0}^m \d_X^\kappa[\d_X; \nabla]\d_X^{\ell-\kappa-1}f
\\
&=
-\sum_{\kappa=0}^m\sum_{l=0}^\kappa
 C_\kappa^l \d_X^l\nabla X_\alpha
 \d_X^{\kappa-l}\d_\alpha \d_X^{\ell-\kappa-1}f;
 \end{split}
 \eeq
 and
\beq\label{appeq1}
\begin{split}
&\d_X^{m+1}\nabla^2\d_X^{\ell-m-1}f
-\nabla^2\d_X^\ell f =
 \sum_{\kappa=0}^m \d_X^\kappa[\d_X; \nabla^2]\d_X^{\ell-\kappa-1}f
\\
&=
-\sum_{\kappa=0}^m\sum_{l=0}^\kappa
 C_\kappa^l\bigl( \d_X^l\nabla^2 X_\alpha
 \d_X^{\kappa-l}\d_\alpha \d_X^{\ell-\kappa-1}f
 +2\d_X^l\nabla X_\alpha \,\d_X^{\kappa-l}\nabla\d_\alpha\d_X^{\ell-\kappa-1}f\bigr);
 \end{split}
 \eeq
and
\beq\label{appeq3}
\begin{split}
\d_X^m\nabla&\d_X^{n+1}\nabla\p_X^{\ell-m-n-1}f-\d_X^m\nabla^2\p_X^{\ell-m}f
=
\sum_{\kappa=0}^n\d_X^m \nabla
\d_X^\kappa[\d_X;\nabla]\p_X^{\ell-m-\kappa-1}f
\\
&=
-\sum_{\kappa=0}^n\sum_{l=0}^\kappa  C_\kappa^l \d_X^m\nabla
\bigl( \d_X^l\nabla X_\al\,
 \d_X^{\kappa-l}\p_\al \p_X^{\ell-m-\kappa-1}f\bigr)
\\
&=
- \sum_{q=0}^m\sum_{\kappa=0}^n\sum_{l=0}^\kappa
C_m^q C_\kappa^l \Bigl(\d_X^q\nabla \d_X^l \nabla X_\al\,
 \d_X^{\kappa-l+m-q}\p_\al \p_X^{\ell-m-\kappa-1}f
 \\
 &\qquad\qquad\qquad\qquad\qquad
+ \d_X^{l+q}\nabla X_\al
 \d_X^{ m-q}\,\na
 \d_X^{\kappa-l}\p_\al\p_X^{\ell-m-\kappa-1}f\Bigr);
\end{split}\eeq
and
\beq\label{appeq4}
\begin{split}
\d_X^{m+1}\nabla\d_X^n\nabla\p_X^{\ell-m-n-1}f
&-\nabla\d_X^{m+n+1}\nabla\p_X^{\ell-m-n-1}f\\
&=\sum_{\kappa=0}^m \d_X^\kappa
 [\d_X; \nabla]\d_X^{m-\kappa+n}\nabla\p_X^{\ell-m-n-1}f\\
 &=\sum_{\kappa=0}^m \sum_{l=0}^\kappa C_\kappa^l
\d_X^l \nabla X_\al\,
\d_X^{\kappa-l}\p_\al\d_X^{m-\kappa+n}
\nabla\p_X^{\ell-m-n-1}f.
\end{split}\eeq

%\subsection{Proof of Lemma \ref{lem:comm,X}}

\subsection{Proof of \eqref{comm:X}}\label{subs:comm,X}
We shall first present the estimate to the second term in \eqref{comm:X}.
It is easy to observe that \eqref{comm:X}  holds trivially  when $i+j=0$.
Suppose by induction that \eqref{comm:X} is valid for any pair of nonnegative integers
$(i,j)$ with $i+j\leq \ell -1$,   and it suffices to show \eqref{comm:X} in the following two cases:
$$
 i\leq \ell-1, \,  j\leq \ell\hbox{ such that }i+j\leq \ell\quad
\hbox{ or }\quad i\leq \ell,\, j\leq \ell-1\hbox{ such that }i+j\leq \ell.
$$

%%%%%%%%%%%%%%%%%%
\noindent$\bullet$
{Case $  i\leq \ell-1$, $ j\leq \ell$ with $i+j\leq \ell.$}

%%%%%%%%%%%%%%%%%%%%%%%
We first deduce from \eqref{appeq3}
and \eqref{Rell}  that for nonnegative integers $m,n\leq \ell-1$ with $m+n\leq \ell-1$
\begin{align*}
\bigl\|\d_X^m&\nabla\d_X^{n+1}\nabla\rx^{\ell-m-n-1}
-\d_X^m\nabla^2\rx^{\ell-m}\bigr\|_{L^\infty_t(L^p)}
\\
&\leq
C \sum_{q=0}^m\sum_{\kappa=0}^{n}\sum_{l=0}^\kappa\Bigl(
\bigl\|\d_X^q \nabla\d_X^l \nabla X\bigr\|_{L^\infty_t(L^p)}
\bigl\|\d_X^{\kappa-l+m-q}
\nabla \rx^{\ell-m-\kappa-1}\bigr\|_{L^\infty_t(L^\infty)}
\\
&\quad\qquad\qquad\qquad
+
\|  \d_X^{l+q}\nabla X\|_{L^\infty_t(L^{\infty)}}
\bigl\|\d_X^{m-q} \nabla\d_X^{\kappa-l}
\nabla \rx^{\ell-m-\kappa-1}X\bigr\|_{L^\infty_t(L^p)}\Bigr)\\
&\leq
C\sum_{l\leq \ell-1}R_{l+1}(t)R_{\ell-l}(t)
  \leq
C R_{\ell}^2(t) .
\end{align*}
This in turn shows that for any pair of nonnegative integers $(i,j)$ with $i\leq \ell-1$, $j\leq \ell$, $i+j\leq \ell$,
\begin{equation}\label{comm:X1}
\begin{split}
& \bigl\|\d_X^i \nabla  \d_X^j \nabla\rx^{\ell-i-j}
    - \d_X^i\nabla^2\rx^{\ell-i}
    \bigr\|_{L^\infty_t(L^p)}
   \leq C R_\ell^2(t).
    \end{split}
    \end{equation}
   Taking $i=0$ in \eqref{comm:X1} leads to
\begin{equation}\label{comm:X1a}
\begin{split}
& \bigl\|  \nabla  \d_X^j \nabla\rx^{\ell-j}
    -  \nabla^2 \rx^{\ell}
    \bigr\|_{L^\infty_t(L^p)}
   \leq C R_\ell^2(t),
   \quad\forall\ j\leq\ell.
    \end{split}
    \end{equation}

%%%%%%%%%%%%%%%%%%%%%%%%%
  \noindent$\bullet$
 {Case $  i\leq \ell$, $ j\leq \ell-1$ with $i+j\leq \ell.$}
  %%%%%%%%%%%%%%%%%%%%%%

It follows from \eqref{appeq4} that for nonnegative integer pair $(m,n)$ satisfying $ m+n\leq \ell-1,$
$$\longformule{\bigl\|\d_X^{m+1}\nabla\d_X^n\nabla\rx^{\ell-m-n-1}
-\nabla\d_X^{m+n+1}\nabla\rx^{\ell-m-n-1}\bigr\|_{L^p}
}{{}\leq
\sum_{\kappa=0}^m \sum_{l=0}^\kappa C_\kappa^l
\bigl\|\d_X^l \nabla X\bigr\|_{L^\infty}
\bigl\|\d_X^{\kappa-l}\nabla\d_X^{m-\kappa+n}
\nabla\rx^{\ell-m-n-1}\bigr\|_{L^p}
\leq CR_\ell^2(t).
}$$ This together with \eqref{comm:X1a} ensures that  for $ (i,j)$ satisfying $  i\leq\ell,$ $  j\leq \ell-1,$
    and $ i+j\leq \ell,$
\begin{equation}\label{comm:X2}
\begin{split}
\bigl\| \d_X^i\nabla  \d_X^j \nabla\rx^{\ell-i-j}
    - \nabla^2\rx^\ell
    \bigr\|_{L^\infty_t(L^p)}
    \leq &
     \bigl\| \d_X^i\nabla  \d_X^j \nabla\rx^{\ell-i-j}
    - \nabla\d_X^{i+j}\nabla\rx^{\ell-i-j}
    \bigr\|_{L^\infty_t(L^p)}
    \\
&
    +
    \bigl\|\nabla\d_X^{i+j}\nabla\rx^{\ell-i-j}
    - \nabla^2\rx^\ell
    \bigr\|_{L^\infty_t(L^p)}
 \leq C  R_\ell^{2}(t) .
    \end{split}
    \end{equation}
 Taking $j=0$ in \eqref{comm:X2} gives rise to
\begin{equation}\label{comm:X2a}
\begin{split}
&\bigl\| \d_X^i\nabla^2  \rx^{\ell-i}
    - \nabla^2 \rx^\ell
    \bigr\|_{L^\infty_t(L^p)}
     \leq C  R_\ell^{2}(t),
    \quad\forall \   i\leq\ell.
    \end{split}
    \end{equation}
Combining \eqref{comm:X1} with \eqref{comm:X2a}, we obtain for  $(i,j)$ satisfying
$  i\leq\ell-1,$
$ j\leq \ell,$ and
    $ i+j\leq \ell$ that
\begin{equation*}
\begin{split}
\bigl\| \d_X^i\nabla  \d_X^j \nabla\rx^{\ell-i-j}
    - \nabla^2\rx^\ell
    \bigr\|_{L^\infty_t(L^p)}
    \leq  &     \bigl\| \d_X^i\nabla  \d_X^j \nabla\rx^{\ell-i-j}
    - \d_X^i\nabla^2 \rx^{\ell-i}
    \bigr\|_{L^\infty_t(L^p)}\\
   & +\bigl\|\d_X^i\nabla^2 \rx^{\ell-i}
    - \nabla^2\rx^\ell
    \bigr\|_{L^\infty_t(L^p)}
     \leq C R_\ell^{2}(t) ,
    \end{split}
    \end{equation*}
 which together with \eqref{comm:X2} ensures that \eqref{comm:X2} holds for all nonnegative  integer pair $(i,j)$ satisfying $i+j\leq \ell.$

 \smallbreak

Along the same line, it follows from \eqref{appeq0} that for any integer $0\leq m\leq\ell-1$,
\beno\begin{split}
\bigl\|\p_X^{m+1}\na\rx^{\ell-m-1}-\na\rx^\ell\bigr\|_{L^\infty_t(L^p)}
\leq &C\sum_{\kappa=0}^{m}\sum_{l=0}^\kappa
 \|\p_X^l\na X\|_{L^\infty_t(L^\infty)}\|\p_X^{\kappa-l}\na\rx^{\ell-\kappa-1}\|_{L^\infty_t(L^p)},
\end{split}
\eeno
which together with the definition of $R_\ell(t)$ given by \eqref{Rell}  implies
\beno
\|\p_X^i\na\rx^{\ell-i}-\na\rx^\ell\|_{L^\infty_t(L^p)}\leq CR_\ell^2(t),
\quad\forall\ 0\leq i\leq\ell.
\eeno
This together with \eqref{comm:X2}
% \eqref{comm:X2}
shows \eqref{comm:X}.

 %%%%%%%%%%%%%%%%%%%%%%%%
\subsection{Proof of \eqref{comm:v}}
%%%%%%%%%%%%%%%%%%%%%

\subsubsection{The estimates of $(\d_X^{i} \nabla \v^{\ell-i}
   - \nabla \v^\ell)$ and  $(\d_X^{i} \nabla \rm{\pi}^{\ell-i}
   - \nabla \rm{\pi}^\ell)$.}

It is easy to observe from \eqref{appeq0} that for any nonnegative integer $m\leq \ell-1$
and for any $r\in [1,+\infty]$ that
\begin{align*}
& \| \d_X^{m+1} \nabla \v^{\ell-m-1}
   - \nabla \v^\ell \|_{L^r}
   \leq
 C\sum_{\kappa=0}^m\sum_{l=0}^\kappa
 \| \d_X^l \nabla X \|_{L^\infty}
   \|\d_X^{\kappa-l}\nabla \v^{\ell-\kappa-1}\|_{L^r},
   \quad\forall\ m\leq\ell-1.
\end{align*}
Hence we get   for
$r_1\in \left\{2, \f{2p}{p-2}, +\infty\right\}$
and for $1\leq i\leq\ell$,
\begin{equation}\label{comm:v1}
\begin{split}
&  \bigl\|\sigma ^{\left(1-\frac{1}{r_1}-\frac{s_{\ell-1}}{2}\right)}
 (\d_X^{i} \nabla \v^{\ell-i}
   - \nabla \v^\ell) \|_{L^\infty_t(L^{r_1})}
   \\
&   \leq
C\sum_{\kappa=0}^{i-1}\sum_{l=0}^\kappa
 \| \d_X^l \nabla X \|_{L^\infty_t(L^\infty)}
   \bigl\|\sigma ^{\left(1-\frac{1}{r_1}-\frac{s_{\ell-1}}{2}\right) }
   \d_X^{\kappa-l}\nabla \v^{\ell-\kappa-1}\bigr\|_{L^\infty_t(L^{r_1})}
   \leq CR_\ell(t)\frak{A}_{\ell-1}(t).
\end{split}
\end{equation}
Along the same line, we obtain for $r\in \{2,p\}$,
\beq \label{comm:pi1}
 \bigl\|\sigma^{\left(\frac 32-\frac 1r-\frac{s_{\ell-1}}{2}\right)}
 ( \d_X^i \nabla{\rm \pi}^{\ell-i}  -\nabla{\rm \pi}^\ell) \bigr\|_{L^\infty_t(L^r)} \leq CR_\ell(t)\frak{A}_{\ell-1}(t).\eeq

%%%%%%%%%%%%%%
\subsubsection{The estimate of $(\d_X^i\nabla  \d_X^j \nabla\v^{\ell-i-j}
    - \nabla^2 \v^\ell)$}

It follows from \eqref{appeq3} that  for the nonnegative integer pair $(m,n)$ with
$m+n\leq \ell-1$ and for  $r\in \{2,p\}$
\begin{align*}
& \bigl\|\sigma^{\left(\frac 32-\frac 1r-\frac{s_{\ell-1}}{2}\right)}
(\d_X^m\nabla\d_X^{n+1}\nabla\v^{\ell-m-n-1}
-\d_X^m\nabla^2\v^{\ell-m})\|_{L^r}
\\
&\quad\leq
C \sum_{q=0}^m \sum_{\kappa=0}^n\sum_{l=0}^\kappa\Bigl(
 \|\d_X^q \nabla \d_X^l \nabla X \|_{L^p}
 \bigl\|\sigma^{1-\frac{s_{\ell-1}}{2}}
 \d_X^{\kappa-l+m-q}\nabla \v^{\ell-m-\kappa-1}\bigr\|_{L^{\frac{2p}{p-2}}\cap L^\infty}
\\
&\qquad\qquad\qquad\qquad+
 \| \d_X^{l+q}\nabla X \|_{L^{\infty}}
 \bigl\|\sigma^{\left(\frac 32-\frac 1r-\frac{s_{\ell-1}}{2}\right)}
 \d_X^{m-q}\nabla\d_X^{\kappa-l}
 \nabla \v^{\ell-m-\kappa-1}\bigr\|_{L^r}\Bigr),
\end{align*}
which together with \eqref{Rell} and \eqref{Vell} implies that for  $r\in \{2,p\}$ and any nonnegative integer pair $ (i,j)$ satisfying $
    i\leq\ell-1,
     j\leq \ell$ with $i+j\leq \ell$
\begin{equation}\label{comm:v2}
\begin{split}
& \bigl\|\sigma ^{\left(\frac 32-\frac 1r-\frac{s_{\ell-1}}{2}\right)}
 ( \d_X^i \nabla  \d_X^j \nabla\v^{\ell-i-j}
    - \d_X^i\nabla^2 \v^{\ell-i}) \bigr\|_{L^\infty_t(L^r)}
     \leq CR_\ell(t) \frak{A}_{\ell-1}(t).
    \end{split}
    \end{equation}

Similarly for non-negative integer pair $(m,n)$ satisfying $m+n\leq \ell-1,$  we deduce from \eqref{appeq4} that for $r\in \{2,p\}$
\begin{align*}
 \|\d_X^{m+1}\nabla\d_X^n\nabla\v^{\ell-m-n-1}
&-\nabla\d_X^{m+n+1}\nabla\v^{\ell-m-n-1}\|_{L^r}
\\
&\leq C\sum_{\kappa=0}^m \sum_{l=0}^\kappa
\bigl\|\d_X^l \nabla X\bigr\|_{L^\infty}
 \|\d_X^{\kappa-l}\nabla\d_X^{m-\kappa+n}
\nabla \v^{\ell-m-n-1}\|_{L^r},
\end{align*}
which ensures that for $r\in \{2,p\}$ and any nonnegative integer pair $ (i,j)$ satisfying $
    i\leq\ell,
     j\leq \ell-1$ with $i+j\leq \ell$
\begin{equation}\label{comm:v2a}
\begin{split}
& \bigl\|\sigma ^{\left(\frac 32-\frac 1r-\frac{s_{\ell-1}}{2}\right)}
 ( \d_X^i \nabla  \d_X^j \nabla\v^{\ell-i-j}
    - \nabla\p_X^{i+j}\nabla \v^{\ell-i-j}) \bigr\|_{L^\infty_t(L^r)}
     \leq CR_\ell(t) \frak{A}_{\ell-1}(t).
    \end{split}
    \end{equation}
With \eqref{comm:v2} and \eqref{comm:v2a}, we can repeat the proof of \eqref{comm:X} to conclude that
 for any integer pair $(i,j)$ satisfying $i+j\leq \ell$ and $r\in \{2,p\}$
\begin{equation}\label{comm:v3}
\begin{split}
& \bigl\|  \sigma ^{\left(\frac 32-\frac 1r-\frac{s_{\ell-1}}{2}\right)}
 (\d_X^i\nabla  \d_X^j \nabla\v^{\ell-i-j}
    - \nabla^2 \v^\ell )
   \bigr\|_{L^\infty_t(L^r)}
  \leq C R_\ell(t) \frak{A}_{\ell-1}(t).
    \end{split}
    \end{equation}
Along with  \eqref{comm:v1} and \eqref{comm:pi1}, we complete the proof of  \eqref{comm:v}.

%%%%%%%%%%%%%%%%%%%%%%%%%%%%%%%
\subsection{Proof of \eqref{comm:v''}}
%%%%%%%%%%%%%%%%%%%%%%%
\subsubsection{The estimates of $\bigl(\d_X^i \nabla D_t \v^{\ell-i}
-\nabla D_t \v^\ell,\,
\d_X^i D_t\nabla \v^{\ell-i}- \nabla D_t \v^\ell,
\,\d_X^i D_t\nabla \pi^{\ell-i}- \nabla D_t \pi^\ell\bigr)$}

It is easy to observe from \eqref{appeq0} and $[D_t; \d_X]=0$ that for nonnegative integer $m\leq\ell-1$,
\begin{align*}
 \d_X^{m+1}\nabla D_t \v^{\ell-m-1}-\nabla D_t \v^\ell
&=
-\sum_{\kappa=0}^m\sum_{l=0}^\kappa C_\kappa^l
\d_X^l\nabla X_\al\d_X^{\kappa-l}\p_\al D_t \v^{\ell-\kappa-1},
\end{align*}
and
\begin{align*}
&\d_X^{m+1} D_t \nabla \v^{\ell-m-1}- D_t\nabla \v^\ell
=D_t(\d_X^{m+1}   \nabla \v^{\ell-m-1}-  \nabla \v^\ell)
\\
&=-\sum_{\kappa=0}^m\sum_{l=0}^\kappa  C_\kappa^l
\left(\d_X^l D_t \nabla X_\al
\d_X^{\kappa-l}\p_\al  \v^{\ell-\kappa-1}
+ \d_X^l \nabla X_\al
\d_X^{\kappa-l}D_t \p_\al  \v^{\ell-\kappa-1}\right),
\end{align*}
so that we write
\begin{equation*}
\begin{split}
&\bigl\|\sigma^{\left(\frac 32-\frac{1}{r_3}-\frac{s_{\ell-1}}{2}\right)}
\bigl( \d_X^{m+1}\nabla D_t \v^{\ell-m-1}-\nabla D_t \v^\ell,
\d_X^{m+1} D_t \nabla \v^{\ell-m-1}- D_t\nabla \v^\ell \bigr) \bigr\|_{L^{r_3}}
\\
&\leq
C\sum_{\kappa=0}^m\sum_{l=0}^\kappa
\Bigl( \|\d_X^l\nabla X\|_{L^\infty_t(L^\infty)}
\bigl\|\sigma^{\left(\frac 32-\frac{1}{r_3}-\frac{s_{\ell-1}}{2}\right)}
(\d_X^{\kappa-l}\nabla D_t \v^{\ell-\kappa-1},
\d_X^{\kappa-l}D_t \nabla \v^{\ell-\kappa-1})\bigr\|_{L^{r_3}}
\\
&\qquad
+\sigma(t)^{-  \frac{1-s_{\ell}}{2} }
\|\sigma^{1-\frac{s_\ell}{2}}\d_X^l D_t \nabla X\|_{L^\infty_t(L^\infty)}
\bigl\|\sigma^{\left(1-\frac{1}{r_3}-\frac{s_{\ell-1}}{2}\right)}
\d_X^{\kappa-l}\nabla  \v^{\ell-\kappa-1}\|_{L^\infty_t(L^{r_3})}\Bigr),
\end{split}
\end{equation*}
which together with the fact
\begin{align*}
\bigl\|\sigma^{\left(\frac 32-\frac{1}{r_3}-\frac{s_{\ell-1}}{2}\right)}
(D_t\nabla \v^\ell-\nabla D_t \v^\ell)\bigr\|_{L^{r_3}}
&=
\bigl\|\sigma^{\left(\frac 32-\frac{1}{r_3}-\frac{s_{\ell-1}}{2}\right)}
\nabla v\cdot\nabla \v^\ell\bigr\|_{L^{r_3}}
\\
&\leq
\sigma^{-\frac{1-s_\ell}{2}}
\|\sigma^{1-\frac{s_0}{2}}
\nabla v\|_{L^\infty_t(L^\infty)}
\|\sigma^{1-\frac{1}{r_3}-\frac{s_\ell}{2}}
\nabla \v^\ell\|_{L^\infty_t(L^{r_3})},
\end{align*}
ensures that for nonnegative integer $i\leq\ell$ and   $r_3=2$ or $\frac{2p}{p-2}$
\begin{equation}\label{comm:v''1}
\begin{split}
\bigl\|\sigma^{\left(\frac 32-\frac{1}{r_3}-\frac{s_{\ell-1}}{2}\right)}
\bigl(\d_X^{i}\nabla D_t \v^{\ell-i}-\nabla D_t \v^\ell,&
\d_X^{i} D_t \nabla \v^{\ell-i}- D_t\nabla \v^\ell \bigr) \bigr\|_{ L^{r_3}}^2
\\
&\leq
C R_\ell^2(t)\frak{B}_{\ell-1}(t)
+C \frak{A}_{\ell}^4(t)\sigma^{-(1-s_\ell)}.
\end{split}
\end{equation}
The same argument yields  for $m\leq\ell-1$ that
\begin{align*}
&
\bigl\|\sigma^{\frac{3-s_{\ell-1}}2}
(\d_X^{m+1} D_t \nabla {\rm \pi}^{\ell-m-1}- D_t\nabla {\rm\pi}^\ell)\bigr\|_{L^2}
\\
&\leq
C \sum_{\kappa=0}^m\sum_{l=0}^\kappa
\Bigl( \|\d_X^l\nabla X\|_{L^\infty_t(L^\infty)}
\bigl\|\sigma^{\frac{3-s_{\ell-1}}2}
\d_X^{\kappa-l}D_t \nabla {\rm\pi}^{\ell-\kappa-1}\bigr\|_{L^{2}}
\\
&\qquad
+\sigma(t)^{-  \frac{1-s_{\ell}}{2} }
\|\sigma^{1-\frac{s_\ell}{2}}\d_X^l D_t \nabla X\|_{L^\infty_t(L^\infty)}
\bigl\|\sigma^{1-\frac{s_{\ell-1}}2}
\d_X^{\kappa-l}\nabla  {\rm\pi}^{\ell-\kappa-1}\|_{L^\infty_t(L^{2})}\Bigr),
\end{align*}
which together with the fact
\begin{align*}
\bigl\|\sigma^{ \frac{3-s_{\ell-1}}{2}}
(D_t\nabla {\rm\pi}^\ell-\nabla D_t {\rm\pi}^\ell)\|_{L^2}
&=
\bigl\|\sigma^{ \frac{3-s_{\ell-1}}{2}}
(\nabla v\cdot\nabla {\rm\pi}^\ell)\bigr\|_{L^2}
\\
&\leq
\sigma^{-\frac{1-s_\ell}{2}}
\|\sigma^{1-\frac{s_0}{2}}
\nabla v\|_{L^\infty_t(L^\infty)}
\|\sigma^{1- \frac{s_\ell}{2}}
\nabla {\rm\pi}^\ell\|_{L^\infty_t(L^2)},
\end{align*}
implies that
\beq\label{comm:v''2}
 \bigl\|\sigma(t)^{ \frac{3-s_{\ell-1}}2 }
 (\d_X^i D_t\nabla{\rm\pi}^{\ell-i}- \nabla D_t{\rm\pi}^\ell)(t)  \bigr\|_{L^2 }^2\leq C R_{\ell}^2(t)\frak{B}_{\ell-1}(t)
+C\frak{A}_{\ell}^4(t)\s^{-(1-s_\ell)}.
 \eeq

%%%%%%%%%%%%%%%%%%
\subsubsection{The estimate of $(\d_X^{i}D_t\nabla^2 \v^{\ell-i}
-\nabla^2 D_t {\v^\ell})$}
We first deduce from \eqref{appeq1} that for nonnegative integer $m\leq\ell-1$
\begin{equation*}\begin{split}
\|\sigma^{\frac{3-s_{\ell-1}}{2}}
(D_t\d_X^{m+1}&\nabla^2 \v^{\ell-m-1}
-D_t\nabla^2 \v^\ell)\|_{L^2}
\\
\leq
C\sum_{\kappa=0}^m\sum_{l=0}^\kappa
\Bigl( &\sigma^{-\frac{1-s_\ell}{2}}
\|\sigma^{1-\frac{s_\ell}{2}}\d_X^l D_t\nabla^2 X\|_{L^\infty_t(L^2)}
 \|\sigma^{1-\frac{s_{\ell-1}}{2}}
  \d_X^{\kappa-l}\nabla \v^{\ell-\kappa-1}\|_{L^\infty_t(L^\infty)}
  \\
&
 +\|\d_X^l  \nabla^2 X \|_{L^\infty_t(L^p)}
 \|\sigma^{\frac{3-s_{\ell-1}}{2}}
 \d_X^{\kappa-l} D_t \nabla \v^{\ell-\kappa-1}\|_{L^{\frac{2p}{p-2}}}
 \\
&
 +\sigma^{-\frac{1-s_\ell}{2}}
 \|\sigma^{1-\frac{s_\ell}{2}}\d_X^l D_t\nabla X \|_{L^\infty_t(L^\infty)}
\|\sigma^{1-\frac{s_{\ell-1}}{2}}
\d_X^{\kappa-l}\nabla^2\v^{\ell-\kappa-1}\|_{L^\infty_t(L^2)}
\\
&
 +\|\d_X^l \nabla X\|_{L^\infty_t(L^\infty)}
   \|\sigma^{\frac{3-s_{\ell-1}}{2}}
   \d_X^{\kappa-l}D_t\nabla^2\v^{\ell-\kappa-1}\|_{L^2}
 \Bigr),
 \end{split}
\end{equation*}
 which implies that for any nonnegative integer $i\leq\ell$,
\begin{equation*}\begin{split}
&\|\sigma^{\frac{3-s_{\ell-1}}{2}}
(D_t\d_X^{i}\nabla^2 \v^{\ell-i}
-D_t\nabla^2 \v^\ell)\|_{L^2}^2
\leq CR_{\ell}^2(t)\frak{B}_{\ell-1}(t)
+C\frak{A}_{\ell}^4(t)\sigma^{-(1-s_\ell)}.
 \end{split}
\end{equation*}
We notice then
\begin{align*}
& \bigl\|\sigma^{\frac{3-s_{\ell-1}}2}
(D_t\nabla^2 \v^\ell
-\nabla^2 D_t\v^\ell\bigr)\|_{L^2}
= \bigl\|\sigma^{\frac{3-s_{\ell-1}}2}
(\nabla^2 v\cdot\nabla \v^\ell
+2\nabla v\cdot\nabla^2 \v^\ell)\|_{L^2}
\\
&\leq
\sigma(t)^{-\f{1-s_\ell}2}
\Bigl( \|\sigma^{1-\frac{s_{0}}2}
 \nabla^2 v\|_{L^2}
 \|\sigma^{1-\frac{s_\ell}{2}}
 \nabla \v^\ell\|_{L^\infty}
 +2\|\sigma^{1-\frac{s_0}{2}}\nabla v\|_{L^\infty}
 \bigl\| \sigma^{1-\frac{s_\ell}{2}}\nabla^2 \v^\ell\|_{L^2}\Bigr)
\\
&\leq
C  \frak{A}^2_\ell(t)\sigma(t)^{-\frac{1-s_\ell}2},
\end{align*}
from which and the the above inequality,  we infer that for any nonnegative integer $i\leq\ell$,
\begin{equation}\label{comm:v''5}\begin{split}
&\|\sigma^{\frac{3-s_{\ell-1}}{2}}
(\d_X^{i}D_t\nabla^2 \v^{\ell-i}
-\nabla^2D_t \v^\ell)\|_{L^2}^2
\leq CR_{\ell}^2(t)\frak{B}_{\ell-1}(t)
+C\frak{A}_{\ell}^4(t)\sigma^{-(1-s_\ell)}.
 \end{split}
\end{equation}

\subsubsection{The esimate of $(\nabla\d_X^{i}\nabla D_t \v^{\ell-i}-\nabla^2 D_t \v^\ell)$}

Again it is easy to observe from \eqref{appeq0} that  for nonnegative integer $m\leq\ell-1$
\begin{align*}
\|&\nabla\d_X^{m+1}\nabla D_t \v^{\ell-m-1}-\nabla^2 D_t \v^\ell\|_{L^2}
=
\|\nabla(\d_X^{m+1}\nabla D_t \v^{\ell-m-1}-\nabla D_t \v^\ell)\|_{L^2}
\\
&\leq
C \sum_{\kappa=0}^m\sum_{l=0}^\kappa
\Bigl( \|\nabla\d_X^l\nabla X\|_{L^p}
\|\d_X^{\kappa-l}\nabla D_t \v^{\ell-\kappa-1}\|_{L^{\frac{2p}{p-2}}}
+
\|\d_X^l\nabla X\|_{L^\infty}
\|\nabla\d_X^{\kappa-l}\nabla D_t \v^{\ell-\kappa-1}\|_{L^2}\Bigr),
\end{align*}
 which  implies for any integer $0\leq i\leq\ell$ that
\begin{equation}\label{comm:v''3}
\bigl\|\sigma ^{\frac{3-s_{\ell-1}}2 }
(\nabla\d_X^{i}\nabla D_t \v^{\ell-i}-\nabla^2 D_t \v^\ell)(t)\bigr\|_{L^2}^2
\leq
CR_\ell^2(t)\frak{B}_{\ell-1}(t).
\end{equation}
By summing up the Estimates \eqref{comm:v''1}, \eqref{comm:v''2}, \eqref{comm:v''5} and \eqref{comm:v''3}, we achiveve  \eqref{comm:v''}.
 This completes the proof of the lemma.

 \smallbreak

\noindent {\bf Acknowledgments.} We would like to thank Professor
Jean-Yves Chemin and Professor Rapha\"{e}l Danchin for  profitable discussions of this topic. Part of this work was done
when we were visiting Morningside Center of Mathematics (MCM), Chinese Academy of Sciences.
We appreciate the hospitality
and the financial support from MCM.
{}
X. Liao is supported by SFB 1060, Universit\"at Bonn at the end of the work.
P. Zhang is partially supported
by NSF of China under Grant   11371347 and innovation grant from National Center for
Mathematics and Interdisciplinary Sciences.

\end{document}